\documentclass[11pt]{article}
\usepackage{amsfonts}
\usepackage{amsmath,amssymb}
\usepackage{mathrsfs}
\usepackage[dvips]{graphicx}
\usepackage{algorithm}
\usepackage{algorithmicx}
\usepackage{algpseudocode}
\usepackage{yhmath}
\usepackage{qtree}
\usepackage{xcolor}
\usepackage{epsfig}
\usepackage{lineno,hyperref}
\usepackage{mathtools}
\usepackage{amsthm}
\usepackage[T1]{fontenc}
\usepackage[utf8]{inputenc}
\usepackage{booktabs}
\usepackage{empheq}
\usepackage{authblk}
\usepackage[toc]{appendix}
\numberwithin{equation}{section}
\usepackage{verbatim}
\usepackage{graphicx}
\usepackage{geometry}
\usepackage[many]{tcolorbox}
\usetikzlibrary{shadows}
\usepackage{caption}
\usepackage{subfigure}

\usepackage{amsmath}
\allowdisplaybreaks[1]

\newtcolorbox{shadedbox}{
  breakable,
  enhanced jigsaw,
  colback=white,
}

\numberwithin{equation}{section}
\newtheorem{thm}{Theorem}[section]
\newtheorem{lem}{Lemma}[section]
\newtheorem{rem}{Remark}[section]
\newtheorem{prop}{Proposition}[section]

\newcommand{\ed}{\end {document}}
\begin{document}
\title{A decreasing upper bound of energy for time-fractional phase-field equations}

\author[1]{Chaoyu Quan}
\author[2,3]{Tao Tang}
\author[4,5]{Boyi Wang}
\author[4,3]{Jiang Yang}

\affil[1]{\small SUSTech International Center for Mathematics, Southern University of Science and Technology, Shenzhen, China (\href{mailto:quancy@sustech.edu.cn}{quancy@sustech.edu.cn}).}
\affil[2]{\small Division of Science and Technology, BNU-HKBU United International College, Zhuhai, Guangdong, China (\href{mailto:ttang@uic.edu.cn}{ttang@uic.edu.cn}).}
\affil[3]{\small Guangdong Provincial Key Laboratory of Computational Science and Material Design, Southern University of Science and Technology, Shenzhen, China  (\href{mailto:yangj7@sustech.edu.cn}{yangj7@sustech.edu.cn}).}
\affil[4]{\small Department of Mathematics, Southern University of Science and Technology, Shenzhen, China.}
\affil[5]{\small Department of Mathematics, National University of Singapore, Singapore (\href{mailto:boyiwang@u.nus.edu}{boyiwang@u.nus.edu}).}

\date{}

\maketitle
\begin{abstract}
In this article, we study the energy dissipation property of time-fractional Allen--Cahn equation. 
We propose a decreasing upper bound of energy that decreases with respect to time and coincides with the original energy at $t = 0$ and as $t$ tends to $\infty$. 
This upper bound can also be viewed as a nonlocal-in-time modified energy, the summation of the original energy and an accumulation term due to the memory effect of time fractional derivative.
In particular, this indicates that the original energy indeed decays w.r.t. time in a small neighborhood at $t=0$.  
We illustrate the theory mainly with the time-fractional Allen--Cahn equation, but it could be applied to other time-fractional phase-field models such as the Cahn--Hilliard equation.

On the discrete level, the first-order L1 and second-order L2 schemes for time-fractional Allen--Cahn equation have similar decreasing modified energies, so that the stability can be established.
Some numerical results are provided to illustrate the behavior of this modified energy and to verify our theoretical results.
\end{abstract}

\begin{keywords}
{time-fractional Allen--Cahn equation, energy dissipation, L1 approximation, L2 approximation}
\end{keywords}

\section{Introduction}
Phase-field models have various applications in diverse regions such as hydrodynamics, material sciences, image processing and biology simulation, etc.
Most phase-fields models are derived as gradient flows associating with some specific energy functional, such as the Ginzburg--Landau energy for Allen--Cahn equations and Cahn--Hilliard equations, Swift--Hohenberg energy for phase-field crystal models.
Seeking numerical solutions of phase-field equations has attracted a lot of  attentions in the passed decade, which could be a delicate task: intrinsic properties of the solution shall be recovered on the discrete level (energy dissipation, maximum principle)
and the presence of small parameter $\varepsilon > 0$ can generate practical difficulties.
%So far there have been plenty of numerical analysis and computation aiming to handle this task, see, e.g., \cite{du2020review,hughes2011,shen2019new,tang2020A} and the references therein.
There have been plenty of numerical schemes for phase-field equations, including the convex-splitting schemes \cite{elliott1993global,eyre1998unconditionally,wang2010unconditionally,chen2012linear}, the stabilization schemes  \cite{zhu1999coarsening,xu2006stability,shen2010numerical}, the implicit-explicit (IMEX) schemes \cite{tang2016implicit,li2016characterizing,li2021stability}, the operator splitting methods \cite{li2021strang-a,li2022stability}, the scalar auxiliary variable (SAV) schemes \cite{shen2018scalar,shen2019new}, and the exponential time differencing (ETD) schemes \cite{du2020review,fu2022energy}.

Recently much interest has arisen in the study of the time-fractional phase-field (TFPF) equations.
For instance, phase-field framework has been successfully employed to describe the evolution of structural damage and fatigue \cite{caputo2015}, in which the damage is described by a variable order time fractional
derivative.
In \cite{zhao2018Time} %%\cite{WangH2018}
, the TFPF models account for the anomalously subdiffusive transport behavior in heterogeneous porous materials.
%The energy dissipation behavior of time fractional Allen--Cahn and Cahn--Hilliard equations is observed numerically.
Liu et al. study the coarsening dynamics for the time-fractional Cahn-Hilliard (TFCH) model based on numerical observations in \cite{zhao2018Time}
%%\cite{LIU20181876}
, while Chen et al. consider the time-fractional molecular beam epitaxy model in \cite{chen2019accurate}.%/cite{aaaeacpc}
These problems are challenging due to the existence of both nonlocality and nonlinearity.
It is natural to extend the relevant discrete level intrinsic properties,
i.e., the maximum principle and energy stability to handle the
TFPF equations, e.g., \cite{du2020time,Luchko2017,lisalgado21}.

The Allen-Cahn (AC) model is a popular phase-field model with the governing equation
\begin{equation}
  \partial_t u= \gamma(\varepsilon^2\Delta u - F'(u)),\quad (t,x)\in  (0,T) \times \Omega \label{eq2},
\end{equation}
where $\varepsilon>0$ is the interface width, $\gamma>0$ is the diffusion mobility constant, and
\begin{align}
F(u) = \frac14(u^2-1)^2
\end{align}
is the double well potential.
The energy functional of the AC equation \eqref{eq3} is
\begin{equation}
  E(u) \coloneqq \int_\Omega \left(\frac{\varepsilon^2}{2}|\nabla u|^2 +F(u)\right) \,{\mathrm d}x.\label{glenergy}
 \end{equation}
With homogeneous Dirichlet/Neumann or periodic boundary condition, this energy decreases with respect to time:
\begin{equation}
  \frac{\mathrm d}{\mathrm d t} E(u) =  -\gamma^{-1}\int_\Omega (\partial_t u)^2 \,{\mathrm d}x \leq 0,
\end{equation}
i.e., the so-called energy dissipation law.

In this work we are concerned with the time-fractional Allen--Cahn (TFAC) equation:
\begin{equation}
  \partial_t^\alpha u= \gamma\left(\varepsilon^2\Delta u -F'(u)\right),\quad(t,x)\in  (0,T) \times \Omega \label{eq3},
\end{equation}
where $\alpha\in(0,1)$ and $\partial_t^\alpha$ is the Caputo fractional derivative  defined by
\begin{displaymath}
\partial_t^\alpha u \coloneqq \frac{1}{\Gamma(1-\alpha)}\int_0^t\frac{\partial_\tau u(\tau)}{(t-\tau)^\alpha}\,{\mathrm d}\tau,\quad t >0.
\end{displaymath}
It is still an open question if the solution of TFAC equation still preserves the energy dissipation law.
As for ordinary fractional integral equation, Volterra considered the energy law and proposed an energy equation \cite[p.193]{volterra2005theory}. 
In \cite{tang2019energy}, Tang, Yu, and Zhou proved the energy boundedness for different TFPF equations that the energy is bounded by initial energy: $\forall t\geq 0,$
\begin{equation}
E(t) \leq E(0).
\end{equation}
Later in \cite{du2020time}, Du, Yang, and Zhou studied the well-posedness, regularity, and maximal principle of the TFAC equation, and observed numerically the fractional energy law
\begin{equation}\label{eq:fraclaw}
\partial_t^\alpha E \le 0,
\end{equation}
which is proved theoretically by some of us in \cite{quan2020define}. Fritz, Khristenko, and Wohlmuth proposed the equivalence between the time-fractional gradient flow and an integer-order gradient flow in the augmented Hilbert space \cite{fritz2021equivalence}, where a dissipation-preserving augmented energy $E^{\mathrm{aug}}$ is constructed.
For example for the TFAC equation, 
\begin{equation}
E^{\mathrm{aug}}(t) = E(t) + \frac12 \int_0^1\int_\Omega c_0^2(\theta)\left| \int_0^t e^{-c_1(\theta)\cdot (t-s)} \mu\,{\mathrm d} s \right|^2 \,{\mathrm d} x \, w_{\alpha,1}(\theta)\,{\mathrm d}\theta,
\end{equation}
where $c_0(\theta) = (1-\theta)^{-1},~c_1(\theta) =\theta (1-\theta)^{-1}$ and $w_{\alpha,1} = c_1^{1-\alpha}(\theta)/\left(\Gamma(1-\alpha) \Gamma(\alpha)\right)$. 
A variational energy law is proposed by Liao, Tang, and Zhou for the TFAC equation in \cite{liao2021energy} as follows
\begin{equation}
E_\alpha(t) = E(t) + \frac12 \mathcal I_t^\alpha \| \delta_u E\|^2,
\end{equation}
where $\mathcal I_t^\alpha$ denotes the Riemann-Liouville fractional integration operator of order $\alpha\in(0,1)$.
However it is not obvious to generalize the variational energy properly to the TFCH model. 

On the discrete level, the discrete fractional energy law $\overline \partial_t^\alpha E\leq 0$ also holds for L1 schemes of TFPF equations \cite{quan2020numerical}.
It is further shown that the energy boundedness by initial energy can be ensured for arbitrary nonuniform time meshes.
In recent interesting works \cite{hou2021robust,hou2021highly}, Hou and Xu split the nonlocal time-fractional derivative to local and nonlocal terms for the TFAC equation, and treat the derived nonlocal term with the SAV technique, so that the modified discrete energy of L1 and L2 schemes decreases w.r.t. time.

In this article, we show the following identity on the original energy of TFAC equation
\begin{align}
\boxed{
  \gamma\Gamma(1-\alpha) \frac{\mathrm{d}}{\mathrm{d}t} E(t) =  - \frac{\mathrm{d}}{\mathrm{d}t} {D_{\alpha}(t)}-\alpha D_{\alpha+1}(t),}
\end{align}
where the nonlocal term $D_{\alpha}(t)\geq 0$ is well defined in \eqref{energy-diss-1}. 
Then a decreasing upper bound $\tilde E(t)$ of the original energy is deduced
\begin{equation}
\boxed{
 \tilde E(t) = E(t)+\frac 1 {\gamma\Gamma(1-\alpha)} D_{\alpha}(t).}
\end{equation}
This energy bound functional is the sum of the original energy and a nonnegative term arisen from the time-fractional Caputo derivative (see Figure \ref{fig-intro} for graphical illustration of $\tilde E$ and $E$).
Particularly $\tilde E$ has the following features:
\begin{itemize}
%\item $\tilde E(t) \geq E(t)$;
\item $\tilde E$ decreases w.r.t. time;
\item $\tilde E = E$ at $t= 0$ and $\tilde E\rightarrow E$ as $t\to\infty$;
\item for any fixed $ t\in(0,\infty)$, $\tilde E(t)\rightarrow E(t)$ as $\alpha\to 1$.
\end{itemize}
Clearly, such result can be viewed as a generalization of the boundedness by initial energy proposed in \cite{tang2019energy}.
It also indicates that the original energy indeed decays w.r.t. time in some small neighborhood at $t=0$.  
However, it is still unknown if this holds true in longer time.
The construction of upper bound for TFAC equation can be generalized to the time-fractional Cahn--Hilliard (TFCH) equation. 
Furthermore, on the discrete level, similar decreasing discrete upper bounds can be obtained for the L1-type and L2-type schemes of the TFAC equation. 

\begin{figure}[htp!]
  \centering
  \includegraphics[width=0.4\textwidth]{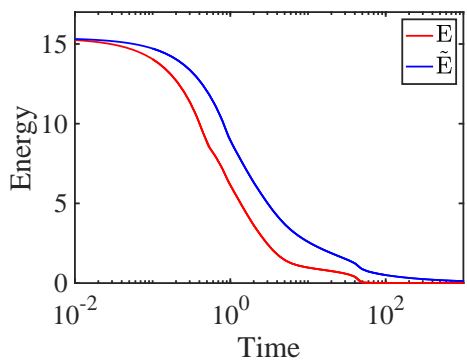}
  \centering
  \caption{Numerical comparison of the energy bound $\tilde E$ and the original Ginzburg-Landau energy $E$ of the time-fractional Allen--Cahn equation.}\label{fig-intro}
\end{figure}

This paper is organized as follows.
In Section \ref{sect2}, we recall some preliminary knowledges on the time-fractional Caputo derivative and the regularity of the solution to TFAC equation.
We introduce a new property on the Caputo derivative in Section \ref{sect3.1}, based on which,  a decreasing energy bound is constructed for the TFAC equation in Section \ref{sect3.2}.
Then we show the relation between the original energy and the energy bound in Section \ref{sect3.3}.
Correspondingly, we propose the discrete modified energy of the first-order L1 implicit-explicit scheme in Section \ref{sect4.1} and the second-order L2 implicit-explicit  scheme for the TFAC equation in Section \ref{sect4.2}.
The analysis for the TFAC equation is extended to the TFCH equation in Section \ref{sect5}. 
We provide some numerical experiments to verify our theoretical results in Section \ref{sect6}.
Several concluding remarks are given in the last section.

\section{Preliminaries}\label{sect2}
We introduce some definitions and theoretical results that are useful in later analysis.

For abbreviation, let $u(t)=u(t,\cdot)$ denote $u(t,x)$ and $\|\cdot\|$ denote the standard spatial $L^2$ norm $\|\cdot\|_{L^2(\Omega)}$.
We consider an equivalent definition of the Caputo derivative obtained from integrating by part as follows (see \cite{allen2016parabolic})
\begin{equation}
\begin{aligned}\label{newdef}
\partial^\alpha_t u(t) &= \frac{1}{\Gamma(1-\alpha)}\int_0^t \frac{\partial_\tau(u(\tau)-u(t))}{(t-\tau)^\alpha}\,{\mathrm d}\tau\\
&= \frac{1}{\Gamma(1-\alpha)}\lim_{\delta\rightarrow 0}\left(\frac{u(\tau)-u(t)}{(t-\tau)^\alpha}\Bigg|_{\tau = 0}^{\tau=t-\delta} -\alpha \int_0^{t-\delta}\frac{u(\tau)-u(t)}{(t-\tau)^{\alpha+1}}\,{\mathrm d}\tau\right)\\
 &= \frac{1}{\Gamma(1-\alpha)}\Bigg(\frac{u(t)-u(0)}{t^\alpha}+\alpha \int_0^{t}\frac{u(t)-u(\tau)}{(t-\tau)^{\alpha+1}}\,{\mathrm d}\tau\Bigg).
 \end{aligned}
 \end{equation}

We recall the definition of Bochner-Sobolev space $H^s(0, T;L^2(\Omega))$, see for example \cite{du2020time} for the settings. 
For any $0<s<1$, one can define Sobolev-Slobodeckii seminorm $|\cdot |_{H^{s}(0,T;L^2(\Omega))}$ by 
\begin{equation} |v|_{H^{s}(0,T;L^2(\Omega))}=\int_0^T\int_0^T\frac{\|(v(t)-v(\tau))\|^2}{|t-\tau|^{1+2s}} \,{\mathrm d}\tau{\mathrm d}t,\label{sobo-1}
\end{equation} 
and the full norm $\|\cdot\|_{H^{s}(0,T;L^2(\Omega))}$ by \begin{equation} \|v\|_{H^{s}(0,T;L^2(\Omega))}=\|v\|_{L^2(0,T;L^2(\Omega))}^2+|v|^2_{H^{s}(0,T;L^2(\Omega))}.\label{sobo-2}\end{equation}

We recall the following regularity result for TFAC equation with Dirichlet boundary condition obtained by Du et al. \cite{du2020time}.However, the authors claim in \cite[Remark 4.1]{du2020time}  that the results also hold for the homogeneous Neumann or periodic boundary condition.
 \begin{prop} \label{theorem1} (\cite{du2020time}) Assuming that the initial data $u(0,x) = u_0(x)\in H^2(\Omega)\bigcap H_0^1(\Omega)$. 
 Then, for $\alpha\in (0,1)$ and $T>0$, there exists a unique solution to the system \eqref{eq3} with homogeneous Dirichlet boundary condition satisfying \begin{align} & u\in H^\alpha(0,T;L^2(\Omega))\bigcap L^2(0,T;H^2(\Omega)\cap H_0^1(\Omega)),~u\in L^\infty((0,T)\times \Omega),\nonumber\\
  &u\in C^\alpha([0,T]; L^2(\Omega))\bigcap C([0,T];H^2(\Omega)\cap H_0^1(\Omega)),~ \partial_t^\alpha u\in C([0,T];L^2(\Omega)),\nonumber\\
  &  \partial_t u\in C([0,T]; L^2(\Omega)),~\|\partial_tu(t)\|\le Ct^{\alpha-1} \text{ for } t\in (0,T].\nonumber
  \end{align} Here and in the following $C$ denotes one constant depending upon $u_0$ and $T$, and changing from line to another.
  \end{prop}

\section{Decreasing energy bound}
In this part, we consider the TFAC equation
\begin{subequations}
  \begin{empheq}[left=\empheqlbrace]{align}
    &\partial_t^\alpha u= \gamma \left(\varepsilon^2\Delta u -f(u)\right),\quad x\in \Omega,\quad t>0 \label{eq01}\\
    &u(0,x) = u_0(x),\quad x\in\Omega
  \end{empheq}
\end{subequations}
with $f(u) = u^3-u$. For the simplicity, we consider the homogeneous Dirichlet boundary condition so that Proposition \ref{theorem1} holds. The cases of other boundary conditions can be done in the same way.

\subsection{A property of Caputo derivative}\label{sect3.1}
Before proposing the decreasing upper bound of the original energy, we first propose a property of Caputo derivative.

  \begin{lem}\label{lemma1111}
    Assuming that $u_0(x)\in H^2(\Omega)\bigcap H_0^1(\Omega)$. 
    For any $t\in (0,T]$, it holds
   \begin{align}
    \Gamma(1-\alpha)\langle\partial_t^\alpha u, \partial_t u\rangle = \frac{\mathrm{d}}{\mathrm{dt}}D_{\alpha}(t)+\alpha D_{\alpha+1}(t)\label{eqlemma1111},
  \end{align}
  where
  \begin{equation}\label{energy-diss-1}
    D_{\alpha}(t) = \frac{\Vert u(t)-u(0)\Vert^2}{2t^\alpha}
    +\frac{\alpha}{2} \int_0^t\frac{\|u(t)-u(\tau)\|^2}{{(t-\tau)^{\alpha+1}}}\,{\mathrm d}\tau
  \end{equation}
  is wellposed.

  \end{lem}
  \begin{proof}
  According to the definition \eqref{newdef} of Caputo derivative, we have
  \begin{align}
   \Gamma(1-\alpha)\langle\partial_t^\alpha u, \partial_t u\rangle = \left\langle\frac{u(t)-u(0)}{t^\alpha}
    +\alpha\int_0^{t}\frac{u(t)-u(\tau)}{(t-\tau)^{\alpha+1}}\,{\mathrm d}\tau,\partial_t u\right\rangle.\label{eq000}
  \end{align}
  Replacing $\partial_t u(t)$ in \eqref{eq000} by $\partial_t \left(u(t)-u(\tau)\right)$, we have
  \begin{align}
    &\Gamma(1-\alpha)\langle\partial_t^\alpha u, \partial_t u\rangle \notag \\
    & = \frac{\langle u(t)-u(0),\partial_t(u(t)-u(0))\rangle}{t^\alpha}+\alpha \int_0^{t}\frac{\langle u(t)-u(\tau),\partial_t(u(t)-u(\tau))\rangle}{(t-\tau)^{\alpha+1}}\,{\mathrm d}\tau\notag\\
    &= {I}_1+\alpha {I}_2,\label{LHS}
    \end{align}
    where
    \begin{equation}
    {I}_1 = \frac{\langle u(t)-u(0),\partial_t(u(t)-u(0))\rangle}{t^\alpha}\notag
    \end{equation}
    and
    \begin{equation}
    {I}_2 =\int_0^{t}\frac{\langle u(t)-u(\tau),\partial_t(u(t)-u(\tau)\rangle}{(t-\tau)^{\alpha+1}}\,{\mathrm d}\tau.\notag
    \end{equation}
    As for $I_1$, by writing the numerator as a derivative of the $L^2(\Omega)$ norm, direct calculation leads to
    \begin{align}
    {I}_1 &= \frac{1}{2t^\alpha}\frac{\mathrm{d}}{\mathrm{d}t}\Vert u(t)-u(0)\Vert^2
    = \frac {\mathrm d}{\mathrm{d}t} \Bigg(\frac{\Vert u(t)-u(0)\Vert^2}{2t^\alpha} \Bigg)+ \alpha\frac{\Vert u(t)-u(0)\Vert^2}{2t^{\alpha+1}}\label{I1}.
    \end{align}
    As for $I_2$, we have, for $t>0$,
    \begin{align}
    & \frac{\mathrm{d}}{\mathrm{d}t} \int_0^t\frac{\left\Vert u(t)-u(\tau)\right\Vert^2}{{(t-\tau)^{\alpha+1}}}\,{\mathrm d}\tau 
    = \frac{\mathrm{d}}{\mathrm{d}t}\lim_{\delta\rightarrow0} \int_0^{t-\delta}\frac{\Vert u(t)-u(\tau)\Vert^2}{{(t-\tau)^{\alpha+1}}}\,{\mathrm d}\tau \notag\\
    = &\lim_{\delta\rightarrow0}\frac{\Vert u(t)-u(t-\delta)\Vert^2}{\delta^{\alpha+1}}
    +2\int_0^{t}\frac{\langle u(t)-u(\tau),\partial_t(u(t)-u(\tau)\rangle}{(t-\tau)^{\alpha+1}}\,{\mathrm d}\tau
    -(\alpha+1)\int_0^t\frac{\Vert u(t)-u(\tau)\Vert ^2}{(t-\tau)^{\alpha+2}}\,{\mathrm d}\tau\nonumber\\
    = &2{I}_2-(\alpha+1)\int_0^t\frac{\Vert u(t)-u(\tau)\Vert^2}{(t-\tau)^{\alpha+2}}\,{\mathrm d}\tau,\nonumber
    \end{align}
    which gives
    \begin{align}
    {I}_2 = \frac12\frac {\mathrm d}{\mathrm{d}t}\Bigg(\int_0^t\frac{\Vert u(t)-u(\tau)\Vert ^2}{{(t-\tau)^{\alpha+1}}}\,{\mathrm d}\tau \Bigg)+\frac{(\alpha+1)}{2}\int_0^t\frac{\Vert u(t)-u(\tau)\Vert ^2}{(t-\tau)^{\alpha+2}}\,{\mathrm d}\tau\label{I2}.
    \end{align}
    Thus,
    \begin{align}\label{eq:3.8}
      \Gamma(1-\alpha)\langle\partial_t^\alpha u, \partial_t u\rangle
      & = \frac {\mathrm d}{\mathrm{d}t} \Bigg(\frac{\Vert u(t)-u(0)\Vert^2}{2t^\alpha}+\frac\alpha2\int_0^t\frac{\Vert u(t)-u(\tau)\Vert ^2}{{(t-\tau)^{\alpha+1}}}\,{\mathrm d}\tau \Bigg) \notag\\
       &
       + \frac{\alpha \Vert u(t)-u(0)\Vert^2}{2t^{\alpha+1}} + \frac{\alpha(\alpha+1)}{2}\int_0^t\frac{\Vert u(t)-u(\tau)\Vert ^2}{(t-\tau)^{\alpha+2}}\,{\mathrm d}\tau.
    \end{align}
     We then obtain
    \begin{align}
      \Gamma(1-\alpha)\langle\partial_t^\alpha u, \partial_t u\rangle = \frac{\mathrm{d}}{\mathrm{d}t}D_{\alpha}(t)+\alpha D_{\alpha+1}(t),\label{AC-ed-1}
    \end{align}
    where $D_{\alpha}(t)$ is defined in \eqref{energy-diss-1}.
    
   The remaining task is to show the wellposedness of the terms on the right-hand side of \eqref{eq:3.8}.
   According to Proposition \ref{theorem1}, we have $\|\partial_t u(t,\cdot)\|\le Ct^{\alpha-1}$ for a fixed $t\in (0,T]$.
   Then we have
   \begin{align}\label{well-1}
&\frac{\|u(t)-u(0)\|^2}{{t^{\alpha+1}}}=\frac{\|\int_0^t\partial_\xi u(\xi,\cdot)\,{\mathrm d}\xi\|^2}{{t^{\alpha+1}}}\le \frac{(\int_0^t\|\partial_\xi u(\xi,\cdot)\|\,{\mathrm d}\xi)^2}{{t^{\alpha+1}}}\le Ct^{\alpha-1}
\end{align}
and
\begin{align}\label{well-2}
\int_0^t\frac{\|u(t)-u(\tau)\|^2}{{(t-\tau)^{\alpha+2}}}\,{\mathrm d}\tau &=\int_0^t\frac{\|\int_\tau^t\partial_\xi u(\xi)\,{\mathrm d}\xi\|^2}{{(t-\tau)^{\alpha+2}}}\,{\mathrm d}\tau
\le \int_0^t\frac{(\int_\tau^t\|\partial_\xi u(\xi)\|\,{\mathrm d}\xi)^2}{{(t-\tau)^{\alpha+2}}}\,{\mathrm d}\tau\nonumber\\
&\le C \int_0^t\frac{(\int_\tau^t \xi^{\alpha-1} \,{\mathrm d}\xi)^2}{{(t-\tau)^{\alpha+2}}}\,{\mathrm d}\tau
=\frac C{\alpha}\int_0^t\frac{(t^\alpha-\tau^\alpha)^2}{{(t-\tau)^{\alpha+2}}}\,{\mathrm d}\tau\nonumber\\
&= \frac{Ct^{\alpha-1}}{\alpha}\int_0^1\frac{(1-s^\alpha)^2}{{(1-s)^{\alpha+2}}}\,{\mathrm d}s
\le Ct^{\alpha-1}.
\end{align}
Similarly we have \begin{align}\label{well-4}
&\frac{\|u(t)-u(0)\|^2}{{t^{\alpha}}}\le Ct^{\alpha},\quad\int_0^t\frac{\|u(t)-u(\tau)\|^2}{{(t-\tau)^{\alpha+1}}}\,{\mathrm d}\tau\le Ct^{\alpha}.
\end{align}
Therefore the well-posedness of $D_\alpha$ and $D_{\alpha+1}$ in the above proof is established. 
  \end{proof}

\subsection{Upper bound of energy}\label{sect3.2}
We propose a decreasing upper bound functional of the original energy for the TFAC equation as follows
\begin{align}\label{modifiedEac}
  \tilde E(t) &= E(t) + \frac 1 {\gamma\Gamma(1-\alpha) }D_{\alpha}(t),
\end{align}
where $E(t)$ is the original Ginzburg--Landau energy given by \eqref{glenergy} and $D_{\alpha}(t)$ is given by \eqref{energy-diss-1}.
We show that this ``modified energy'' decreases w.r.t. time
\begin{equation}
  \frac{\mathrm{d}}{\mathrm{d}t}\tilde E(t)\leq 0, \quad t\ge 0. \label{dis-law-1}
\end{equation}
We are ready to state prove our main result on the dissipation of the modified energy.
\begin{thm}\label{theorem_of_ac}
Assume the initial data $u_0\in H^2(\Omega)\bigcap H_0^1(\Omega)$.
The solution $u$ to the TFAC equation \eqref{eq01} satisfies
\begin{equation}
    E(t)\leq\tilde E(t)\leq\tilde E(s)\leq E(0),\quad \forall 0\le s\le t\le T. \label{dis-law-3}
\end{equation}
where $E(t)$ is the original Ginzburg--Landau energy \eqref{glenergy}  and $\tilde E(t)$ is the upper bound \eqref{modifiedEac} (or modified energy).

\end{thm}
\begin{proof}
Multiplying the equation \eqref{eq01} with $\partial_t u$, integrating the resulting one over $\Omega$ and using $\langle \cdot, \cdot \rangle$ to denote the $L^2(\Omega)$ inner product, we have
\begin{equation}
\left\langle \partial_t^\alpha u,\partial_t u\right\rangle =\gamma\left \langle\varepsilon^2\Delta u - f(u),\partial_t u\right\rangle.
\end{equation}
The right-hand side of this equation is
\begin{equation}
\text{(RHS)}:= \gamma\langle\varepsilon^2\Delta u - f(u),\partial_t u\rangle = -\gamma\frac{\mathrm{d}}{\mathrm{d}t} E(t), \label{RHS_of_ac}
\end{equation}
while according to  Lemma \ref{lemma1111},  the left-hand side can be written as
\begin{align}
\text{(LHS)} := \langle\partial_t^\alpha u, u_t\rangle = \frac{1}{\Gamma(1-\alpha)}\frac{\mathrm{d}}{\mathrm{d}t}{D_{\alpha}(t)}+\frac{\alpha}{\Gamma(1-\alpha)}{D_{\alpha+1}(t)}\label{lemma1112},
  \end{align}
 where $D_{\alpha}(t)$ is given by \eqref{energy-diss-1}.
We then have
\begin{align}
  \frac{\mathrm{d}}{\mathrm{d}t} \left(E(t) +  \frac{1}{\gamma\Gamma(1-\alpha)}{D_{\alpha}(t)}\right)+\frac{\alpha}{\gamma\Gamma(1-\alpha)} D_{\alpha+1}(t)=0.\label{theoremproofend}
\end{align}
Note that $D_{\alpha+1}(t)$ is positive..
The desired result \eqref{dis-law-3} in Theorem \ref{theorem_of_ac} follows by integrating \eqref{theoremproofend} over the interval $(s,t)$ for $0\le s\le t\le T$.
\end{proof}

\subsection{Relation between $\tilde E$ and $E$}\label{sect3.3}
We show the relation between $\tilde E(t)$ and $E(t)$.
Precisely speaking, the energy functional $\tilde E(t)$ coincides the original Ginzburg-Landau energy $E(t)$ at $t=0$ and as $t\to\infty$, if the solution $u(t)$ of the TFAC equation \eqref{eq01} converges strongly to some steady state $u_\infty(x)$ in $L^2(\Omega)$ as $t\to \infty$.
Before proving this result, we first propose two useful lemmas.

\begin{lem}[maximum principle \& H\"older continuity]\label{lem_Holder} 
Let $u(t,x)$ be the unique solution of the system \eqref{eq01}. Assume that $u_0\in H^2(\Omega)\cap H_0^1(\Omega)$ and $-1\le u_0\le 1$ on $\bar{\Omega}$. Then we have
\begin{itemize}
\item[(1)] $-1\le u(t,x)\le 1$ for any $x\in \bar{\Omega}$ and $t\ge 0$;
\item[(2)] $\exists C>0$, s.t. $\|u(t,\cdot)-u(\tau,\cdot)\| \leq C |t-\tau|^\alpha$ for any $t,\tau\ge 0$.
\end{itemize}
\end{lem}
\begin{proof}
See Appendix \ref{append3} for details.
\end{proof}

\begin{lem}\label{lem-4.1}
Let $v\in C^\alpha([0,\infty);L^2(\Omega))$.
If
\begin{itemize}
\item $\exists M>0$, s.t. $\|v(t,\cdot)\| \leq M$;
\item $\exists C>0$, s.t. $\|v(t,\cdot)-v(\tau,\cdot)\| \leq C |t-\tau|^\alpha$ for any $t,\tau\ge 0$;
\item $\|v(t,\cdot) -v_\infty\|\xrightarrow{t\rightarrow +\infty} 0$ for some $v_\infty(x)$,
\end{itemize}
we then have
\begin{equation}
\int_0^t \frac{\|v(t)-v(\tau)\|^2}{(t-\tau)^{\alpha+1}} \, {\mathrm d}\tau\rightarrow 0,\quad\mbox{as } t\to \infty.
\end{equation}
\end{lem}
\begin{proof}
Given an arbitrary small value $\delta>0$, it is sufficient to show that there exists some $t_*>0$ s.t.
\begin{equation}
\int_0^t \frac{\|v(t)-v(\tau)\|^2}{(t-\tau)^{\alpha+1}}  \, {\mathrm d}\tau
= \int_0^t \frac{\|v(t)-v(t-s)\|^2}{s^{\alpha+1}}  \, {\mathrm d}s
< \delta,\quad\mbox{when } t> t_*.
\end{equation}

Firstly, there exists sufficiently small $t_1>0$ such that
\begin{equation}
\int_0^{t_1} \frac{\|v(t)-v(t-s)\|^2}{s^{\alpha+1}}\,{\mathrm d}s\leq C^2 \int_0^{t_1} s^{\alpha-1} \,{\mathrm d}s = \frac{C^2 }{\alpha} t_1^{\alpha} <\frac{\delta}{3}.
\end{equation}
Secondly, there exists sufficiently large $t_2>0$ such that
\begin{equation}
\int_{t_2}^{t} \frac{\|v(t)-v(t-s)\|^2}{s^{\alpha+1}}\,{\mathrm d}s
\leq M^2 \int_{t_2}^{\infty} s^{-1-\alpha} \,{\mathrm d}s = \frac{M^2 }{\alpha} t_2^{-\alpha} <\frac{\delta}{3}.
\end{equation}
Note that $t_1$ and $t_2$ are fixed now.
Since $v(t)\rightarrow v_*$, there exists some large $t_3>0$ s.t.
\begin{equation}
\|v(t,\cdot)-v_\infty\| < \delta_1 \coloneqq \sqrt{\frac{\alpha \delta}{12(t_2^{-\alpha}-t_1^{-\alpha})}},\quad \forall t> t_3.
\end{equation}
As a consequence, when $t>t_* = t_2+t_3$, we then have
\begin{equation}
\int_{t_1}^{t_2} \frac{\|v(t)-v(t-s)\|^2}{s^{\alpha+1}}\,{\mathrm d}s < \int_{t_1}^{t_2} \frac{(2\delta_1)^2}{s^{\alpha+1}}\,{\mathrm d}s = (2\delta_1)^2 \alpha^{-1} \left(t_2^{-\alpha} -t_1^{-\alpha}\right) = \frac\delta 3.
\end{equation}
Summing the above inequalities, we conclude that when $t>t_* = t_2+t_3$,
\begin{equation}
\int_{0}^{t} \frac{\|v(t)-v(t-s)\|^2}{s^{\alpha+1}}\,{\mathrm d}s <\delta.
\end{equation}
The proof is completed.
\end{proof}

\begin{prop}[asymptotic analysis]
\label{Prop-ACEt} 
Let $\tilde E(t)$ and $E(t)$ be defined by \eqref{modifiedEac} and \eqref{glenergy}. 
Assume that $u_0\in H^2(\Omega)\bigcap H_0^1(\Omega)$ and $u(t)$ is a global bounded solution of \eqref{eq01}.  Then, for any $\alpha\in(0,1)$, we have
\begin{align}
\tilde E(0) = E(0).\label{AC-prop1} 
\end{align}
Moreover, assuming that $u(t)$ converges to some steady state $u_\infty$  strongly in $L^2(\Omega)$,
then we have
\begin{align}
\tilde E(t)\xrightarrow{t\rightarrow +\infty} E(t).\label{AC-prop2}
\end{align}
In addition, for any $ t\in(0,\infty)$,
\begin{align}
\tilde E(t)\xrightarrow{\alpha\rightarrow 1} E(t).\label{AC-prop3}
\end{align}
\end{prop}
\begin{proof}
Firstly, we denote
  \begin{equation}
    T_1(t)= \frac{\|u(t,\cdot)-u(0,\cdot)\|^2}{2\gamma t^\alpha},\quad T_2(t) = \frac{\alpha}{2\gamma} \int_0^t\frac{\|u(t,\cdot)-u(\tau,\cdot)\|^2}{{(t-\tau)^{\alpha+1}}}d\tau, \notag
  \end{equation}
  so that $\tilde E = E+\frac1{\Gamma(1-\alpha)}(T_1+T_2)$.
According to Lemma \ref{lem_Holder}, $u\in C^\alpha([0,\infty); L^2(\Omega))$, i.e., 
 \begin{equation}
\sup_{t,\tau\in[0,\infty)}\frac{\|u(t,\cdot)-u(\tau,\cdot)\|}{|t-\tau|^\alpha}<\infty.
 \end{equation}
Thus $T_1(t)$  and $T_2(t)$ are continuous at $t=0$. 
The result \eqref{AC-prop1} then follows. 
 
Secondly, since $u(t)$ is a global bounded solution of \eqref{eq01},  we have $T_1(t)\to 0$ as $t\to \infty$. 
According to Lemma \ref{lem-4.1}, we have $T_2(t)\to 0$ as $t\to \infty$. 
Thus, the result \eqref{AC-prop2} is proved. 

Finally \eqref{AC-prop3} holds true, because for any fixed $t\in(0,T)$, $T_1(t)$ and $T_2(t)$ are bounded, and $\frac{1}{\Gamma(1-\alpha)}\rightarrow 0$ as $\alpha\rightarrow1$.
\end{proof}

\section{Numerical schemes}\label{sect4}

\subsection{L1-IMEX scheme with stabilization}\label{sect4.1}
We adopt the L1 approximation \cite{sun2006fully,lin2007finite} of Caputo derivative that appears on the left-hand side in \eqref{eq01}. 
For simplicity, we consider uniform time mesh here.
Let $\triangle t= \frac T N$ be the time step size and $t_n = n\triangle t, ~0\leq n \leq N $. 
The L1 approximation is written as
\begin{equation}
  \overline\partial_t^\alpha u^n \coloneqq \sum_{k = 1}^n b_{n-k}\left(u^k-u^{k-1}\right),\label{eqdisCapdef}
\end{equation}
where $\overline\partial_t^\alpha$ is the discrete fractional derivative with coefficients
\begin{equation}
  b_{k} =\frac{1}{\Gamma(2-\alpha)\triangle t^{\alpha}}\left[(k+1)^{1-\alpha}-k^{1-\alpha}\right],\quad  0\leq k\leq n-1.
\end{equation}
\eqref{eqdisCapdef} can be recast as
\begin{equation}
  \overline\partial_t^\alpha u^n = \sum_{k = 1}^{n-1} {(b_{n-k-1}-b_{n-k})(u^n-u^k)} + b_{n-1}(u^n-u^0),\label{eqdisCapdef-new}.
\end{equation}
We propose the following lemma on the L1 operator $\overline\partial_t^\alpha$.

\begin{lem} \label{resultofdiscretelemma}
For the L1 approximation \eqref{eqdisCapdef}, the following inequality hold:
  \begin{align}
    \langle\overline\partial_t^\alpha u^n &,u^n-u^{n-1}\rangle\geq D^n-D^{n-1},
    \end{align}
  where
  \begin{equation}
  D^n =     \left\{
  \begin{aligned}
      &0, &&n=0,\\
      &\frac12 b_{0} \|u^1-u^0\|^2, &&n=1,\\
      &\frac12\sum_{k = 1}^{n-1} (b_{n-k-1}-b_{n-k})\|u^n-u^k\|^2 + \frac12 b_{n-1} \|u^n-u^0\|^2, &&n\geq2.
  \end{aligned}
  \right.
  \end{equation}
\end{lem}
  \begin{proof} In the case of $n = 1$, we have
    \begin{align}
      \langle\overline\partial_t^\alpha u^1 ,u^1-u^{0}\rangle\geq  \frac12 b_{0} \|u^1-u^0\|^2 - 0= D^1-D^{0}.
    \end{align}
In the case of $n = 2$, we have
    \begin{align}
      \langle\overline\partial_t^\alpha u^2 ,u^2-u^1 \rangle \geq  \frac12(b_0-b_1)\|u^2-u^1\|^2+\frac12b_1\|u^2-u^1\| - \frac12b_1\|u^1-u^0\|^2\notag\\
      \geq\frac12(b_0-b_1)\|u^2-u^1\|^2+\frac12b_1\|u^2-u^1\| - \frac12b_0\|u^1-u^0\|^2 = D^2-D^1.
    \end{align}
In the case of $n\geq3$, multiplying \eqref{eqdisCapdef-new} by $u^n-u^{n-1}$ and integrating over $\Omega$, we have
    \begin{align}
      & \langle\overline\partial_t^\alpha u^n ,u^n-u^{n-1}\rangle \notag\\
      &= \sum_{k = 1}^{n-1} (b_{n-k-1}-b_{n-k})\langle u^n-u^k,u^n-u^k-(u^{n-1}-u^k)\rangle\notag\\
      & \quad + b_{n-1} \langle u^n-u^0,u^n-u^0-(u^{n-1}-u^0)\rangle\notag\\
      %& = \frac12\sum_{k = 1}^{n-1} (b_{n-k-1}-b_{n-k})\left[\|u^n-u^k\|^2-\|u^{n-1}-u^k\|_2^2+\|u^{n}-u^{n-1}\|^2\right]\notag\\
     & \geq \frac12\sum_{k = 1}^{n-1} (b_{n-k-1}-b_{n-k})\|u^n-u^k\|^2 -\frac12\sum_{k = 1}^{n-2} (b_{n-k-1}-b_{n-k})\|u^{n-1}-u^k\|^2\notag\\
     & \quad + \frac 12 b_{n-1}  \|u^{n}-u^0\|^2 - \frac 12 b_{n-1} \|u^{n-1}-u^0\|^2 \notag\\
      &\geq \frac12\sum_{k = 1}^{n-1} (b_{n-k-1}-b_{n-k})\|u^n-u^k\|^2 -\frac12\sum_{k = 1}^{n-2} (b_{n-k-2}-b_{n-k-1})\|u^{n-1}-u^k\|^2 \notag\\
      & \quad + \frac 12 b_{n-1}  \|u^{n}-u^0\|^2 - \frac 12 b_{n-2} \|u^{n-1}-u^0\|^2 \notag\\
      & = D^n-D^{n-1}.
    \end{align}
\end{proof}

We adopt the L1 approximation for the Caputo derivative and use the stabilization technique for the bulk force term. 
Consider the following semidiscrete L1 implicit-explicit scheme for equation \eqref{eq01}:
\begin{align}
  \overline\partial_t^\alpha u(t_n)= \gamma\left(\varepsilon^2\Delta u^{n} - f(u^{n-1})-S(u^n-u^{n-1})\right),\label{schemeac}
\end{align}
where $S\geq 0$ is some stabilization constant. 

\begin{thm}[modified energy dissipation]
 If $S\geq2$ and $\|u^0\|_\infty\leq 1$, then the scheme \eqref{schemeac} is unconditionally energy stable and satisfies a modified energy dissipation law
  \begin{align}
    \tilde E(u^n)\leq\tilde E(u^{n-1}) \quad \forall  n\geq 1,
  \end{align}
  where
  \begin{align}\label{dis-eac}
    \tilde E(u^n) &=  E(u^n)+\frac1{2\gamma}\sum_{k = 1}^{n-1} (b_{n-k-1}-b_{n-k})\Vert u^n-u^k\Vert^2 +\frac1{2\gamma} b_{n-1}\Vert u^n-u^0\Vert^2 .
  \end{align}
\end{thm}
\begin{proof}
It is already known (cf. \cite{tang2019energy}) that in the case of $S\geq 2$, the maximal principle holds, i.e., $\|u^n\|_\infty \leq 1$.
Multiplying \eqref{schemeac} by $(u^{n}-u^{n-1})$ and integrating the equation over $\Omega$, we have
  \begin{align}
    % &\langle\frac{\alpha(u^n-u^{n-1})}{\Gamma(2-\alpha)\triangle t^\alpha}+\frac{1}{\Gamma(1-\alpha)}\Bigg(\sum_{k=1}^{n-1}\frac{u^n-u^{k}}{\Delta t^\alpha}(\frac{1}{(n-k)^\alpha}-\frac{1}{(n-k+1)^\alpha})\Bigg) \notag\\
    % +&\frac{u^n-u^0}{\Gamma(1-\alpha)t_n^\alpha}
    \frac1\gamma\langle\overline\partial_t^\alpha u(t_n), u^n-u^{n-1}\rangle = \langle \varepsilon^2\Delta u^{n} -f(u^{n-1})-S(u^n-u^{n-1}),u^n-u^{n-1}\rangle\label{equation4.1}
  \end{align}
 Denote the left-hand side and the right-hand side of \eqref{equation4.1} by (LHS) and (RHS) respectively.
Then, we have, for $n\geq1$,
\begin{align}
(\mathrm{RHS}) =& \langle \varepsilon^2\Delta u^{n} -f(u^{n-1})-S(u^n-u^{n-1}),u^n-u^{n-1}\rangle\notag\\
=&-\frac{\varepsilon^2}2\Vert\nabla u^n\Vert^2+\frac{\varepsilon^2}2\Vert\nabla u^{n-1}\Vert^2-\frac{\varepsilon^2}2\Vert\nabla u^n-\nabla u^{n-1}\Vert^2\notag\\
&-\langle f(u^{n-1})(u^n-u^{n-1})+S(u^n-u^{n-1})^2,1\rangle\notag\\
=&-\frac{\varepsilon^2}2\Vert\nabla u^n\Vert^2+\frac{\varepsilon^2}2\Vert\nabla u^{n-1}\Vert^2-\frac{\varepsilon^2}2\Vert\nabla u^n-\nabla u^{n-1}\Vert^2\notag\\
&-\langle F(u^n)-F(u^{n-1})+(S-\frac12f'(\xi^n))(u^n-u^{n-1})^2,1\rangle\label{theoremacresultp1}.
\end{align}
According to Lemma \ref{resultofdiscretelemma}, we have
\begin{align}
  (\mathrm{LHS})  = \frac1\gamma\langle\overline\partial_t^\alpha u(t_n) ,u^n-u^{n-1}\rangle \geq \frac1\gamma(D^n-D^{n-1}).\label{theoremacresultp2}
\end{align}
As a result, combining \eqref{theoremacresultp1} and \eqref{theoremacresultp2}, we have
\begin{align}
\tilde E(u^n)-\tilde E(u^{n-1}) \leq -\langle(S-\frac12f'(\xi^n))(u^n-u^{n-1})^2,1\rangle\leq 0,
\end{align}
where we use the fact that $|f'(\xi^n)| = |3(\xi^n)^2-1|\leq 4$, $0<\xi^n<1$.
\end{proof}
{\begin{rem}
 For the discrete modified energy $\tilde E(u^n)$ defined in \eqref{dis-eac}, it is easy to check on the discrete level that \eqref{AC-prop1} and \eqref{AC-prop3} in Proposition \ref{Prop-ACEt} hold. 
Furthermore, later numerical results show that \eqref{AC-prop2} is true on the discrete level. 
\end{rem}}

\subsection{L2-IMEX scheme with stabilization}\label{sect4.2}
We study the decreasing energy bound of higher-order schemes for the TFAC equation.
Consider the L2 approximation \cite{lv2016error} of time fractional derivative:
\begin{equation}\label{eq:L2}
  \begin{array}{r@{}l}
	\begin{aligned}
	 L_1^\alpha  u  & = &&\frac{1}{\Gamma(2-\alpha)\Delta t^{\alpha}} \left( u^1- u^0\right),\\
 	 L_n^\alpha  u & = &&\frac{1}{\Gamma(3-\alpha)\Delta t^\alpha} {\Bigg \{} \sum_{j=1}^{n-1} \left(a_j u^{n-j-1}+b_j u^{n-j}+c_j u^{n-j+1}\right) \\
	 & &&+  \frac{\alpha}{2} u^{n-2}-2 u^{n-1}+\frac{4-\alpha}{2}u^n {\Bigg\}},\quad n\geq 2,\\	
	\end{aligned}
  \end{array}
\end{equation}
where
\begin{equation}\label{eq:abc}
  \begin{array}{r@{}l}
	\begin{aligned}
	a_j & = -\frac{3}{2}(2-\alpha)(j+1)^{1-\alpha}+\frac 1 2 (2-\alpha)j^{1-\alpha} + (j+1)^{2-\alpha}-j^{2-\alpha}, \\
	b_j & = 2(2-\alpha)(j+1)^{1-\alpha}-2(j+1)^{2-\alpha}+2j^{2-\alpha},\\
	c_j & = -\frac 1 2 (2-\alpha)\left((j+1)^{1-\alpha}+j^{1-\alpha}\right)+(j+1)^{2-\alpha}-j^{2-\alpha}.
	\end{aligned}
  \end{array}
\end{equation}
Note that the relationship
$a_j+b_j+c_j = 0$ holds and the error of the L2 approximation to the Caputa derivative is $O(\Delta t^{3-\alpha})$.
Here, we reformulate \eqref{eq:L2} to be
\begin{equation}\label{eq:L2reform}
  \begin{array}{r@{}l}
	\begin{aligned}
	 L_1^\alpha  u & = \frac{1}{\Gamma(3-\alpha)\Delta t^\alpha} (r_1 + d_1) \delta u^1, \\
 	 L_n^\alpha  u %& = \frac{1}{\Gamma(3-\alpha)\Delta t^\alpha} {\bigg \{} \frac{4-\alpha}{2}\delta u^n - \frac{\alpha}{2} \delta u^{n-1}+ \sum_{j=1}^{n-1} \left(c_j \delta u^{n-j+1} - a_j \delta u^{n-j} \right){\bigg\}}\\
	 & = \frac{1}{\Gamma(3-\alpha)\Delta t^\alpha} \bigg\{ \frac{3\alpha}{2} \delta u^n - \frac{\alpha}{2} \delta u^{n-1}    +  \sum_{j=1}^{n} d_j \delta u^{n-j+1}  -c_n \delta u^1 \bigg\}, \quad n \geq 2,
	\end{aligned}
  \end{array}
\end{equation}
where
\begin{equation}
\delta u^j = u^j-u^{j-1},
\end{equation}
\begin{equation}\label{eq:r1}
r_1 = 2-\alpha - d_1 = 2 + \frac 1 2 \alpha - \left(\frac \alpha 2 +1\right) 2^{1-\alpha}>0,
\end{equation}
and
\begin{equation}\label{eq:dj}
d_j = \left\{
  \begin{array}{r@{}l}
	\begin{aligned}
	& c_1 + 2-2\alpha, &&  j = 1,\\
	& c_j - a_{j-1}, && j = 2,\ldots,n. \\
	\end{aligned}
  \end{array}	
  \right.
\end{equation}
See \cite{quan2021energy} for the detailed calculations. 

\begin{lem} \label{lem:L2}
For the L2 approximation \eqref{eq:L2reform}, the following inequality hold:
  \begin{align}
    \langle L_n^\alpha u^n &,\delta u^n\rangle\geq   \tilde D^n-\tilde D^{n-1} + \frac\alpha{\Gamma(3-\alpha)\Delta t^\alpha} \|\delta u^n\|^2, \quad n\geq 2,
    \end{align}
  where
  \begin{equation}\label{eq:Dtilde}
  \begin{aligned}
    \tilde D^n &=  \frac{1}{\Gamma(3-\alpha)\Delta t^\alpha} \bigg [\frac\alpha 4 \| u^n-u^{n-1}\|^2 +
    \frac12\sum_{j = 1}^{n-1} (d_{n-j}-d_{n-j+1})\|u^n-u^j\|^2 \\
    &\qquad\qquad\qquad\qquad+ \frac{c_n}{2} \|u^n-u^1\|^2-\frac{a_n}{2} \|u^{n}-u^0\|^2 \bigg ].
    \end{aligned}
    \end{equation}
\end{lem}
  \begin{proof}
  According to \eqref{eq:L2reform}, we have
    \begin{align}
      {\Gamma(3-\alpha)\Delta t^\alpha}\langle L_n^\alpha u^n ,\delta u^n\rangle = Q_1+Q_2+Q_3,
    \end{align}
    where
    \begin{equation}
    \begin{aligned}
    &  Q_1 = \left\langle \frac{3\alpha}2\delta u^n-\frac\alpha 2 \delta u^{n-1},\delta u^n \right\rangle,\\
    & Q_2 = \sum_{j=1}^{n} d_j \left\langle \delta u^{n-j+1}, \delta u^n\right\rangle, \\
    & Q_3 = \left\langle -c_n\delta u^1, \delta u^n\right\rangle.
    \end{aligned}
    \end{equation}
    Clearly, we have
    \begin{equation}\label{eq:Q1}
    Q_1 \geq \alpha \|\delta u^n\|^2 + \frac\alpha 4 (\|\delta u^n\|^2-\|\delta u^{n-1}\|^2).
    \end{equation}
    Further,
    \begin{equation}\label{eq:Q2}
    \begin{aligned}
    Q_2 & = \sum_{j=1}^{n} d_{n-j+1} \left\langle \delta u^{j}, \delta u^n\right\rangle  \\
    &= \sum_{j = 1}^{n-1} (d_{n-j}-d_{n-j+1})\langle u^n-u^j,u^n-u^j-(u^{n-1}-u^j)\rangle\\
      & \quad + d_{n-1} \langle u^n-u^0,u^n-u^0-(u^{n-1}-u^0)\rangle\\
     & \geq \frac12\sum_{j = 1}^{n-1} (d_{n-j}-d_{n-j+1})\|u^n-u^j\|^2 -\frac12\sum_{j = 1}^{n-2} (d_{n-j}-d_{n-j+1})\|u^{n-1}-u^j\|^2\\
     & \quad + \frac 12 d_{n}  \|u^{n}-u^0\|^2 - \frac 12 d_{n} \|u^{n-1}-u^0\|^2 \\
      &\geq \frac12\sum_{j = 1}^{n-1} (d_{n-j}-d_{n-j+1})\|u^n-u^j\|^2 -\frac12\sum_{j = 1}^{n-2} (d_{n-j-1}-d_{n-j})\|u^{n-1}-u^j\|^2 \\
      & \quad + \frac 12 d_{n}  \|u^{n}-u^0\|^2 - \frac 12 d_{n-1} \|u^{n-1}-u^0\|^2+\frac12(d_{n-1}-d_n)\|u^{n-1}-u^0\|^2,
    \end{aligned}
    \end{equation}
    where we use the fact $d_{n-j}-d_{n-j+1}\leq d_{n-j-1}-d_{n-j}$ (see \cite{quan2021energy} for this property).
    Moreover, we have
    \begin{equation}\label{eq:Q3}
    \begin{aligned}
    & Q_3  = \left\langle -c_n\delta u^1, \delta u^n\right\rangle \\
    & = c_n \left\langle u^n - u^1, \delta u^n  \right\rangle
    - c_n \left\langle u^n - u^0, \delta u^n  \right\rangle\\
    & = \frac{c_n}{2} (\|u^n-u^1\|^2-\|u^{n-1}-u^1\|^2 -\|u^{n}-u^0\|^2+\|u^{n-1}-u^0\|^2) \\
    & \geq \frac{c_n}{2} (\|u^n-u^1\|^2- \|u^{n}-u^0\|^2) -\frac{c_{n-1}}{2} (\|u^{n-1}-u^1\|^2-\|u^{n-1}-u^0\|^2)\\
    & - \frac12(c_{n-1}-c_n)\|u^{n-1}-u^0\|^2
    \end{aligned}
    \end{equation}
    where we use the fact $c_{n-1}\geq c_n$ (see \cite{quan2021energy}).
    Summing up \eqref{eq:Q1}--\eqref{eq:Q3} and using the fact (see \cite{quan2021energy} for this property)
    \begin{equation}
    {d_{n-1}-d_n\geq c_{n-1}-c_n,}
    \end{equation}
     we then have
    \begin{equation}
    \begin{aligned}
    \langle L_n^\alpha u^n ,\delta u^n\rangle \geq \tilde D^n - \tilde D^{n-1} +\frac\alpha{\Gamma(3-\alpha)\Delta t^\alpha} \|\delta u^n\|^2,
    \end{aligned}
    \end{equation}
    where $\tilde D^n$ is given by
    \eqref{eq:Dtilde}.
    Note that in \eqref{eq:Dtilde}, $a_n<0$ and consequently $\tilde D^n \geq 0$ holds.
    The proof is completed.
\end{proof}

We consider the following second-order L2 Adam--Bashforth scheme for the TFAC equation:
\begin{equation}\label{eq:L2AB}
  \begin{array}{r@{}l}
	\begin{aligned}
L_{n} ^\alpha u& =\gamma\left[ \varepsilon^2 \Delta u^{n} - 2 \tilde f(u^{n-1})+\tilde f(u^{n-2}) - S \Delta t (u^{n}-u^{n-1}) \right],
	\end{aligned}
  \end{array}
\end{equation}
where $S\ge 0$ is some stabilization coefficient and we adopt the derivative of a truncated double-well potential as the force term (see for example \cite{shen2010numerical})
\begin{equation}\label{eq:trunc_f}
\tilde f(u) = \tilde F'(u)
\end{equation}
with
  \begin{equation}
    \tilde F(u)=\left\{
    \begin{array}{lr}
    \begin{aligned}
    & \frac{1}{2}(3M^2-1)(u-M)^2+(M^3-M)(u-M)+\frac14(M^2-1)^2,&&  u\geq M,\\
    & \frac14(u^2-1)^2, && u\in[-M,M],\\
    & \frac{1}{2}(3M^2-1)(u+M)^2-(M^3-M)(u+M)+\frac14(M^2-1)^2,&&  u\leq -M.\label{che_potentialeq-truncated}
    \end{aligned}
    \end{array}\right.
  \end{equation} 
  Here $M\geq 1$ is some constant. 
In this case, we have
  \begin{equation}
  \max_{u\in\mathbb R} |F''(u)|\leq 3M^2-1.
  \end{equation}
We state and prove the following result on the energy dissipation for this scheme.
\begin{thm}[modified energy dissipation]
If
\begin{equation}
S \geq \frac{3\alpha (3M^2-1)}{2(1+\alpha)}\left(\frac{3\gamma\Gamma(3-\alpha) (3M^2-1)}{2\alpha(1+\alpha)}\right)^{\frac1\alpha},
\end{equation}
  the scheme \eqref{eq:L2AB} satisfies the following modified energy dissipation
\begin{equation}\label{eq:thm6.1}
 \tilde E(u^n)\leq\tilde E(u^{n-1}),\quad \forall n\geq 2,
\end{equation}
where
\begin{equation}
 \tilde E(u^n) = E(u^n) + \frac1\gamma\tilde D^n + \frac12(3M^2-1)\| u^n-u^{n-1}\|^2,
 \end{equation}
 and $\tilde D^n$ given by \eqref{eq:Dtilde}.
\end{thm}
\begin{proof}
Multiplying \eqref{eq:L2AB} by $\delta u^n$ and integrating over $\Omega$, we then obtain
\begin{equation}
    \begin{aligned}
    \frac 1 \gamma\langle L_n^\alpha u^n ,\delta u^n\rangle
    &\leq    E(u^{n-1})- E(u^{n}) +\frac12 \langle f'(\xi^n_1)(u^n-u^{n-1})^2,1\rangle \\
    & \quad -\langle f'(\xi^n_2)(u^{n-1}-u^{n-2}),u^n-u^{n-1}\rangle - S\Delta t \|\delta u^n\|^2\\
    & \leq E(u^{n-1})- E(u^{n}) + (L-S\Delta t)\|\delta u^n\|^2 + \frac L2\|\delta u^{n-1}\|^2.
    \end{aligned}
\end{equation}
where we use
\begin{equation}
|f'(\xi)| \leq 3M^2-1=:L.
\end{equation}
According to Lemma \ref{lem:L2}, we then have
\begin{equation}
\begin{aligned}
 &E(u^{n})- E(u^{n-1}) +\frac1\gamma (\tilde D^n-\tilde D^{n-1})+
 \\
 &\qquad\left(\frac\alpha{\gamma\Gamma(3-\alpha)\Delta t^\alpha}+S\Delta t -L\right) \|\delta u^n\|^2
  - \frac L2\|\delta u^{n-1}\|^2
 \leq 0.
 \end{aligned}
\end{equation}
If
\begin{equation}
S \geq \frac{3\alpha L}{2(1+\alpha)}\left(\frac{3\gamma\Gamma(3-\alpha)L}{2\alpha(1+\alpha)}\right)^{\frac1\alpha},
\end{equation}
it is not difficult to verify
\begin{equation}
\frac\alpha{\gamma\Gamma(3-\alpha)\Delta t^\alpha}+S\Delta t -L\geq \frac L2,\quad \forall \Delta t>0,
\end{equation}
which implies that \eqref{eq:thm6.1} holds true.
\end{proof}

\section{Extension to TFCH equation}\label{sect5}
In this section, we extend the construction of energy upper bound to the TFCH equation
    \begin{equation}
      \partial_t^\alpha u= \gamma \Delta\left(-\varepsilon^2\Delta u +F'(u)\right),\quad x\in \Omega,\quad 0<t<T \label{eq02},
    \end{equation}
with homogeneous Dirichlet/Neumann or periodic boundary condition.
The analysis is similar, excepted for a minor difference and careful computations, where the energy bound of CH equations diverges from that of AC equations obtained in the former subsection.
For simplicity, we consider the homogeneous Dirichlet boundary condition and present some  results briefly. 

Let $\Psi(t,x)$ be the solution of 
\begin{equation}
    -\Delta \Psi (t, x) =u(t,x)-u(0,x),\quad x\in \Omega, t\ge 0\label{eq0192-1}
    \end{equation}
with the homogeneous Dirichlet boundary condition, zero mean $\int_\Omega \Psi (t, x)dx = 0$, and zero initial condition $\Psi(0,x)=0$.
We define the modified energy (upper bound)
\begin{align}
\widetilde E(t) =& E(t) + \frac{1}{\gamma\Gamma(1-\alpha)}D_{ \Delta,\alpha}(t), \notag\\
D_{ \Delta,\alpha}(t) = & \frac{\|\nabla\Psi (t,\cdot)\|^2}{2t^\alpha} + \frac{\alpha}{2} \int_0^t\frac{\|\nabla (\Psi (t,\cdot)-\Psi(\tau,\cdot))\|^2}{{(t-\tau)^{\alpha+1}}}\,\mathrm{d}\tau, \label{ch-energy}
\end{align}
where $E(t)$ is the original Ginzburg--Landau energy \eqref{glenergy}.
Now we state our result on the decreasing dissipation-preserving energy functional as follows.
\begin{thm} Let $u$ be the solution to the problem \eqref{eq02}.
Assume that $D_{ \Delta,\alpha}(t)$ and $D_{ \Delta,\alpha+1}(t)$ are well-defined for $0<t\le T$. Then we have the following decreasing energy dissipation law
\begin{equation}
E(t)\leq\widetilde E(t)\leq\widetilde E(s)\leq E(0),\quad \forall 0\le s\le t\le T.\label{result-ch}
\end{equation}
\end{thm}

\begin{proof}
Similarly to the proof of Theorem \ref{theorem_of_ac},
we have
  \begin{align}
  &\frac{\mathrm{d}}{\mathrm{d}t}\left(E(t)+\frac{\|\nabla\Psi (t,\cdot)\|^2}{2\gamma\Gamma(1-\alpha)t^\alpha}+\alpha\int_0^t\frac{\|\nabla (\Psi (t,\cdot)-\Psi(\tau,\cdot))\|^2}{{2\gamma\Gamma(1-\alpha)(t-\tau)^{\alpha+1}}}\,{\mathrm d}\tau\right)  \notag\\
  &=-\left(\frac{\alpha \|\nabla\Psi (t,\cdot)\|^2}{2\gamma\Gamma(1-\alpha)t^{\alpha+1}}+\alpha(\alpha+1)\int_0^t\frac{\|\nabla (\Psi (t,\cdot)-\Psi(\tau,\cdot))\|^2}{2\gamma\Gamma(1-\alpha)(t-\tau)^{\alpha+2}}\,{\mathrm d}\tau\right)\leq0,\label{result_CH}
  \end{align}
 that is,
 \begin{align}
 &\frac{\mathrm{d}}{\mathrm{d}t}\left(E(t)+ \frac{1}{\gamma\Gamma(1-\alpha)}D_{ \Delta,\alpha}(t)\right)
  =-\frac{\alpha}{\gamma\Gamma(1-\alpha)} D_{ \Delta,\alpha+1}(t)\leq0,
 \end{align}
Integrating \eqref{result_CH} over $(s,t)$ leads to the desired dissipation-preserving law.
\end{proof}

We also give a proposition to reveal the relationship between $\widetilde E_{CH}(t)$ and $E(t)$.

\begin{prop} \label{Prop-CHEt} Let $\widetilde E(t)$ and $E(t)$ are defined in \eqref{ch-energy} and \eqref{glenergy}. Assuming that $u(t)$ is a global bounded solution of \eqref{eq02}.  Then, for any $\alpha\in(0,1)$, we have
\begin{align}
\widetilde E(0) = E(0).\label{CH-prop1}
\end{align}
Assuming that $u(t)$ converges to some steady state $u_\infty$  strongly in $L^2(\Omega)$ and $\exists C>0$, s.t. $\|u(t,\cdot)-u(\tau,\cdot)\| \leq C |t-\tau|^\alpha$ for any $t,\tau\ge 0$,
we then have
\begin{align}
 \widetilde E(t)\xrightarrow{t\rightarrow +\infty} E(t).\label{CH-prop2}
 \end{align}
Besides, for any $t\in (0,\infty)$,
\begin{align}
\widetilde E(t)\xrightarrow{\alpha\rightarrow 1} E(t).\label{CH-prop3}
\end{align}
\end{prop}

On the discrete level, similar modified energy could be obtained. 
For example, consider the semidiscrete L1-IMEX scheme for the TFCH equation
\begin{align}
&\overline\partial_t^\alpha u^n =\gamma\Delta\left( -\varepsilon^2\Delta u^{n} + \tilde f(u^{n-1})+S(u^n-u^{n-1})\right),\label{schemech}
\end{align}
where $\tilde f(u)$ is the truncated double-well potential \eqref{eq:trunc_f}.
\begin{thm} If $S\geq \frac L{2}$, the scheme \eqref{schemech} satisfies a discrete energy dissipation law as below
  \begin{align}
  \widetilde E(u^n)\leq\widetilde E(u^{n-1}),
  \end{align}
  where
  \begin{align}
    \widetilde E(u^n) &=  E(u^n)+\frac1{2\gamma}\sum_{k = 1}^{n-1} (b_{n-k-1}-b_{n-k})\Vert \nabla \Psi^{n,k}\Vert^2 + \frac1{2\gamma}b_{n-1} \Vert \nabla \Psi^{n,0}\Vert^2
  \end{align}
and
  \begin{align}
    -\Delta\Psi^{n,k} = u^n-u^k\quad \forall 0\leq k\leq n-1.\label{explainCH}
  \end{align}
\end{thm}
\begin{proof}
The proof is similar to the case of the TFAC equation and is omitted here.
\end{proof}

\section{Numerical experiments}\label{sect6}
In this section, some numerical implementations for solving the time-fractional phase-field equations are presented.
\subsection{TFAC equation}
Consider the TFAC equation with $\Omega = [-1,1]^2$ and $\varepsilon = 0.025$.
The diffusion constant is chosen as $\gamma = 2$. We use the stablized scheme \eqref{schemeac} with $S = 20$.
Moreover, $128\times 128$ Fourier modes and $\tau = 0.01$ are taken. The initial state is given as
\begin{equation}
  \phi_0(x,y) = \tanh\left(\frac1{\sqrt2\varepsilon}\left(\frac{2r}{3}-\frac14-\frac{1+\cos(4\theta)}{16}\right)\right),
\end{equation}
with
\begin{equation}
   r= \sqrt{x^2+y^2},\quad \theta = \mathrm{arctanh}\left(\frac{x}{y}\right).
\end{equation}

\begin{figure}[!t]
  \centering
  \subfigure[$\alpha = 0.9,~t=1$]{
  \begin{minipage}[t]{0.25\textwidth}
  \centering
  \includegraphics[width=\textwidth]{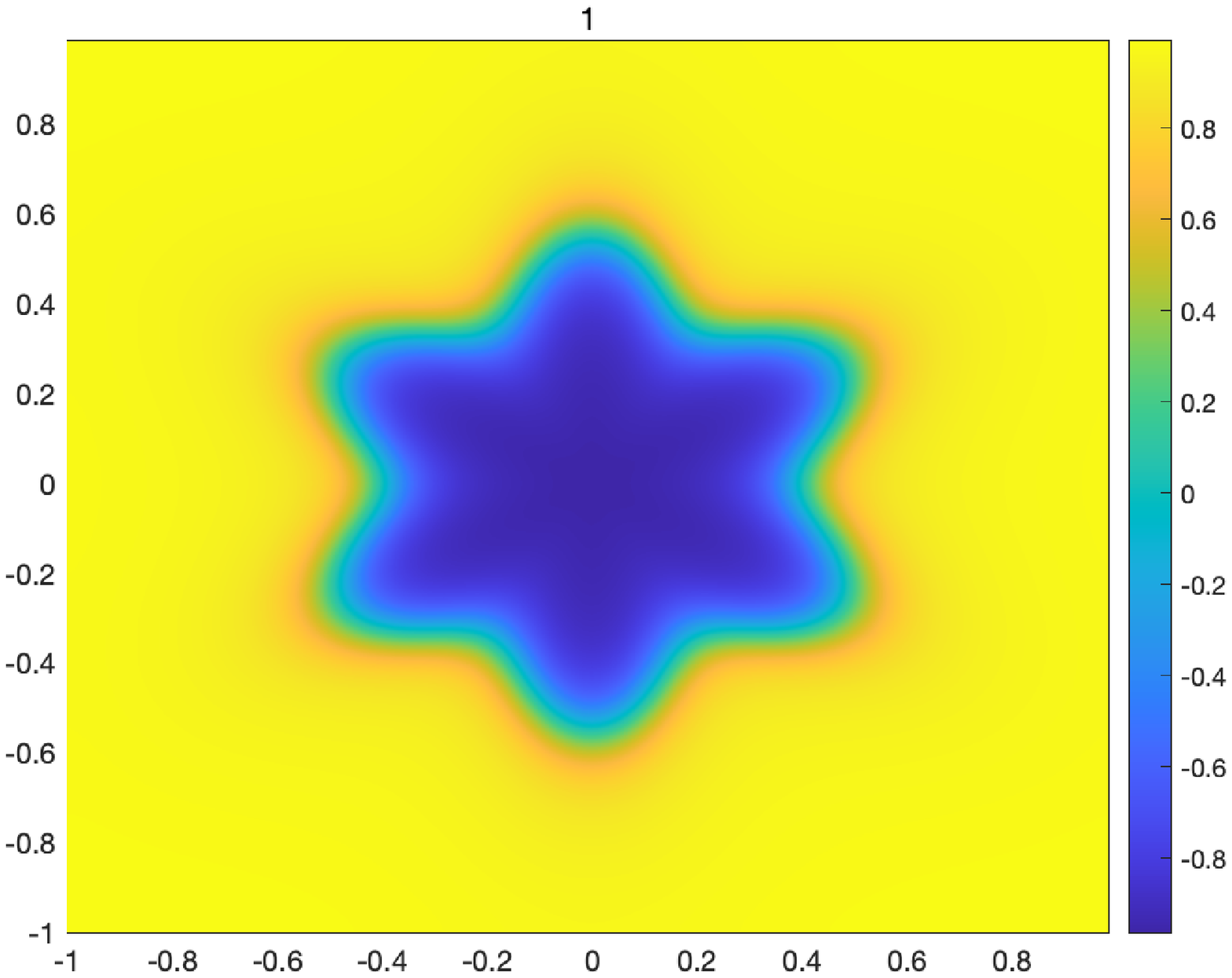}
  \end{minipage}
  }%
  \subfigure[$\alpha = 0.9,~t=8$]{
  \begin{minipage}[t]{0.25\textwidth}
    \centering
    \includegraphics[width=\textwidth]{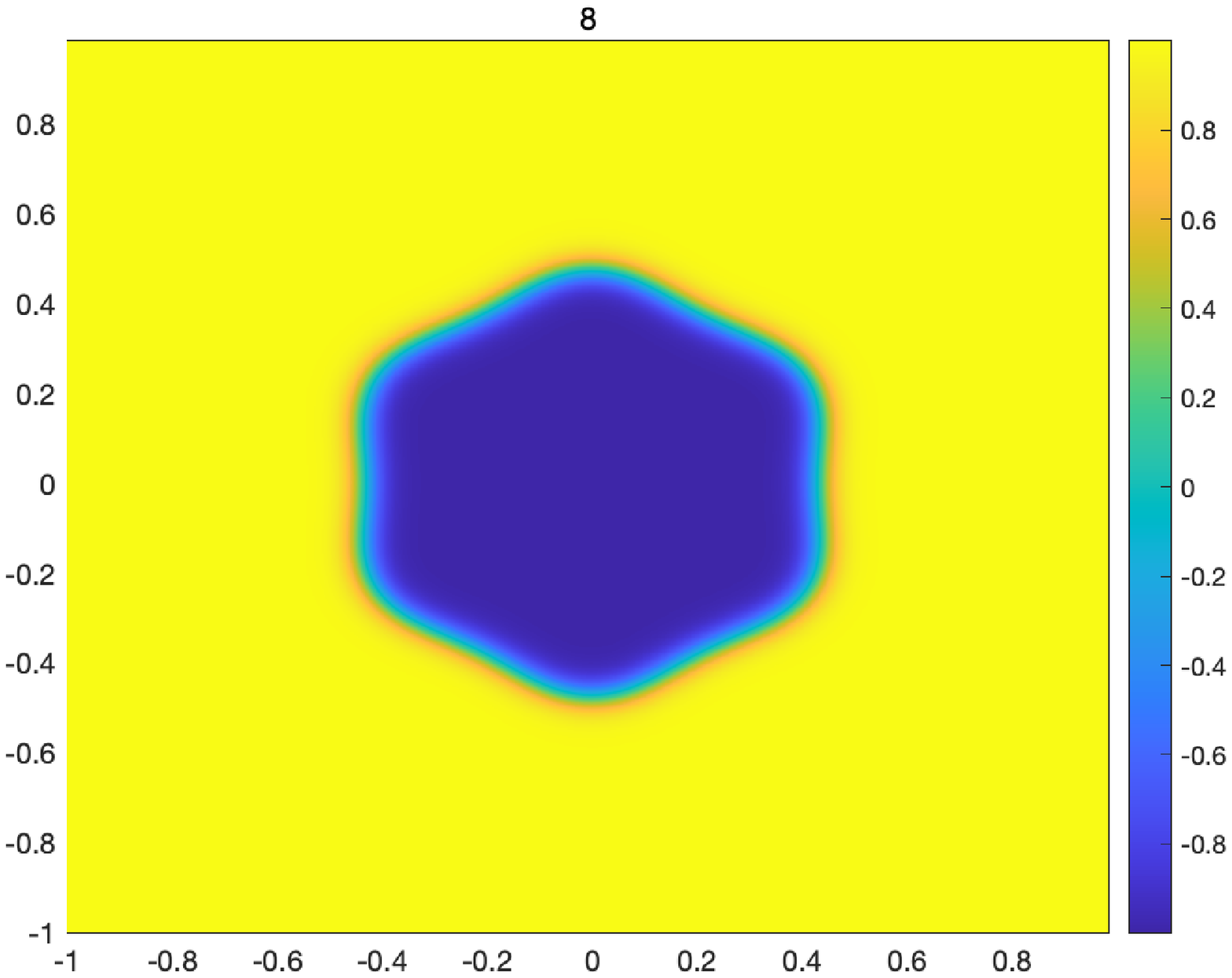}
  \end{minipage}
  }%
  \subfigure[$\alpha = 0.9,~t=64$]{
    \begin{minipage}[t]{0.25\textwidth}
      \centering
      \includegraphics[width=\textwidth]{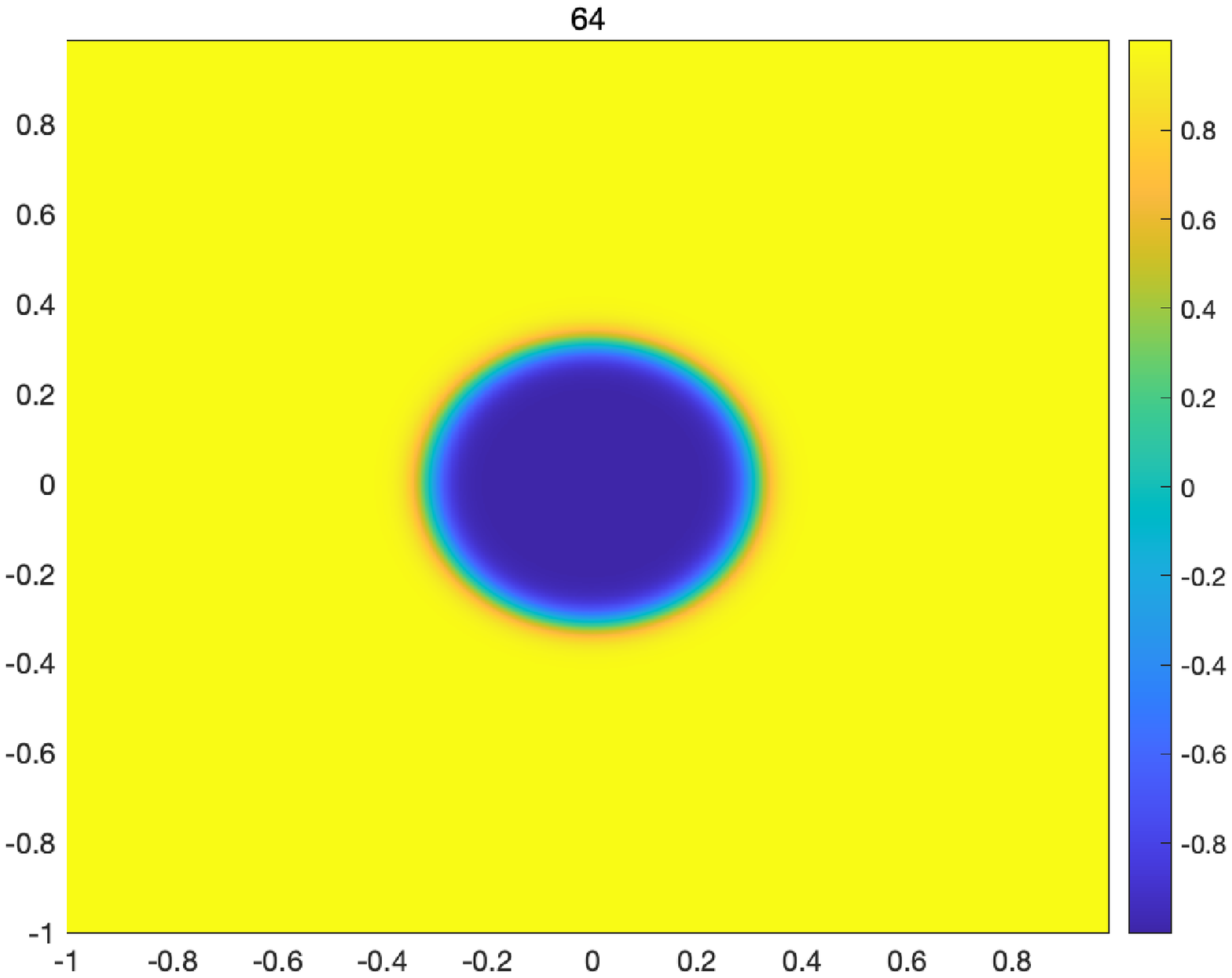}
    \end{minipage}
  }%

  \subfigure[$\alpha = 0.6,~t=1$]{
    \begin{minipage}[t]{0.25\textwidth}
    \centering
    \includegraphics[width=\textwidth]{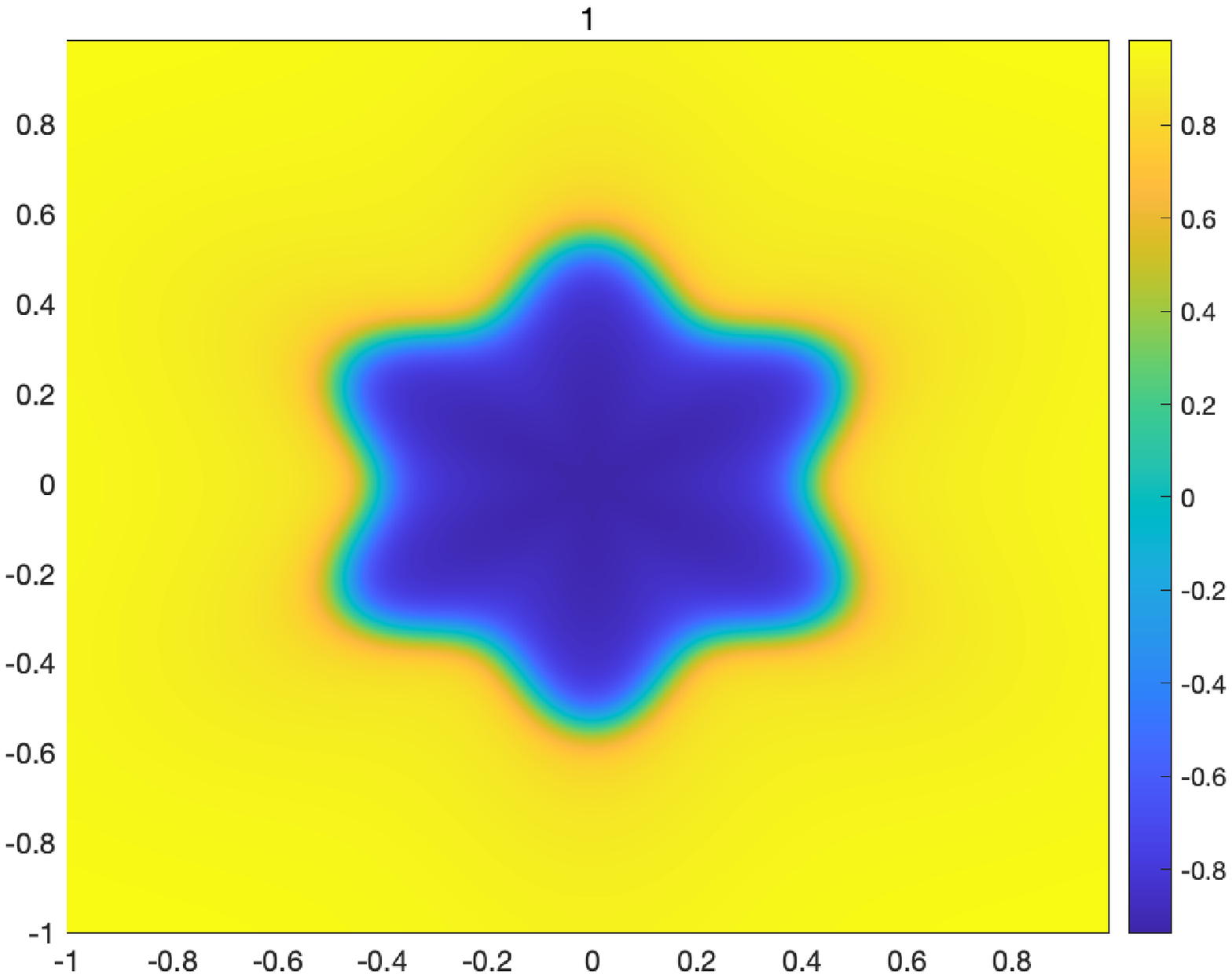}
    \end{minipage}
    }%
    \subfigure[$\alpha = 0.6,~t=8$]{
    \begin{minipage}[t]{0.25\textwidth}
      \centering
      \includegraphics[width=\textwidth]{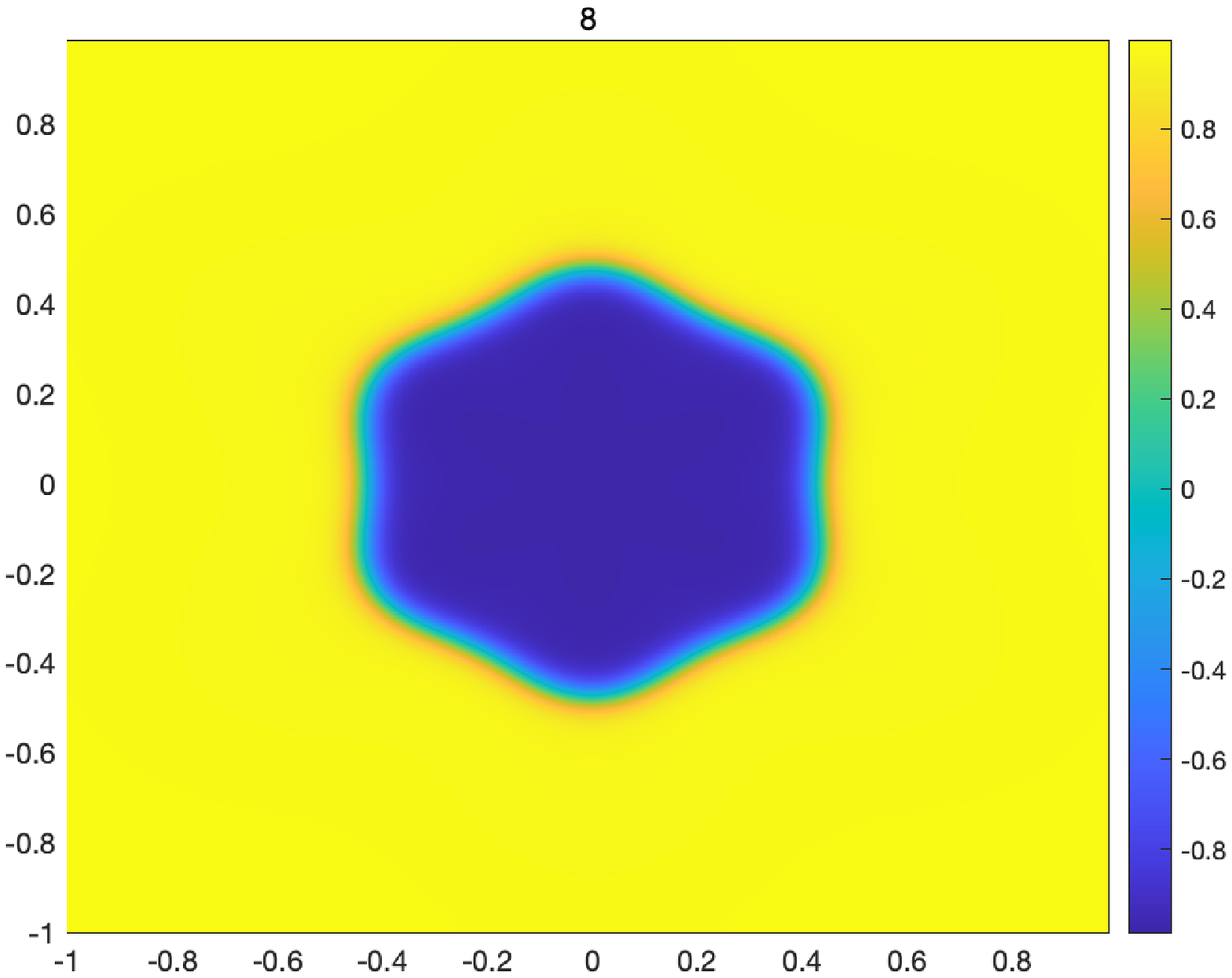}
    \end{minipage}
    }%
    \subfigure[$\alpha = 0.6,~t=64$]{
      \begin{minipage}[t]{0.25\textwidth}
        \centering
        \includegraphics[width=\textwidth]{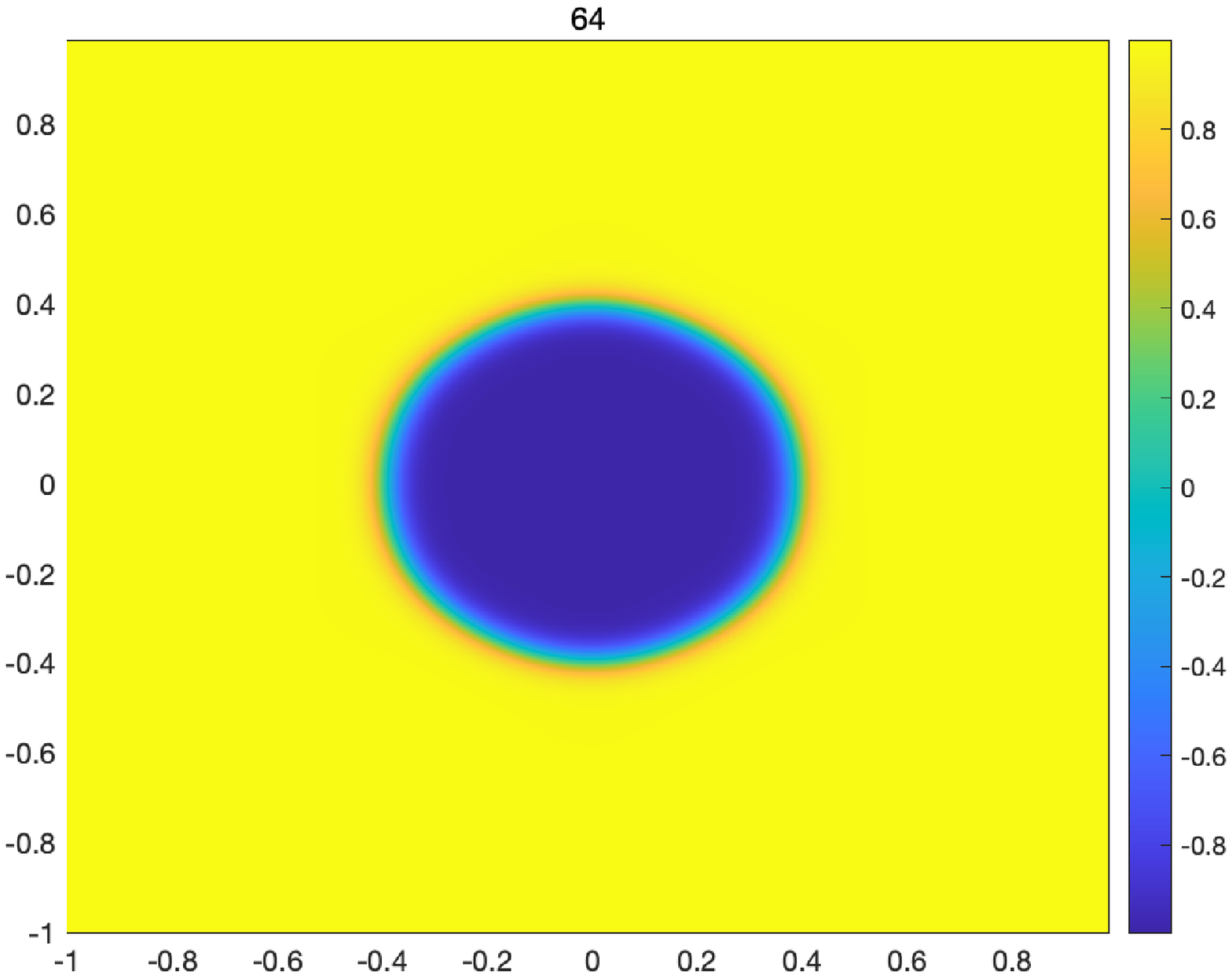}
      \end{minipage}
    }%

    \subfigure[$\alpha = 0.3,~t=1$]{
      \begin{minipage}[t]{0.25\textwidth}
      \centering
      \includegraphics[width=\textwidth]{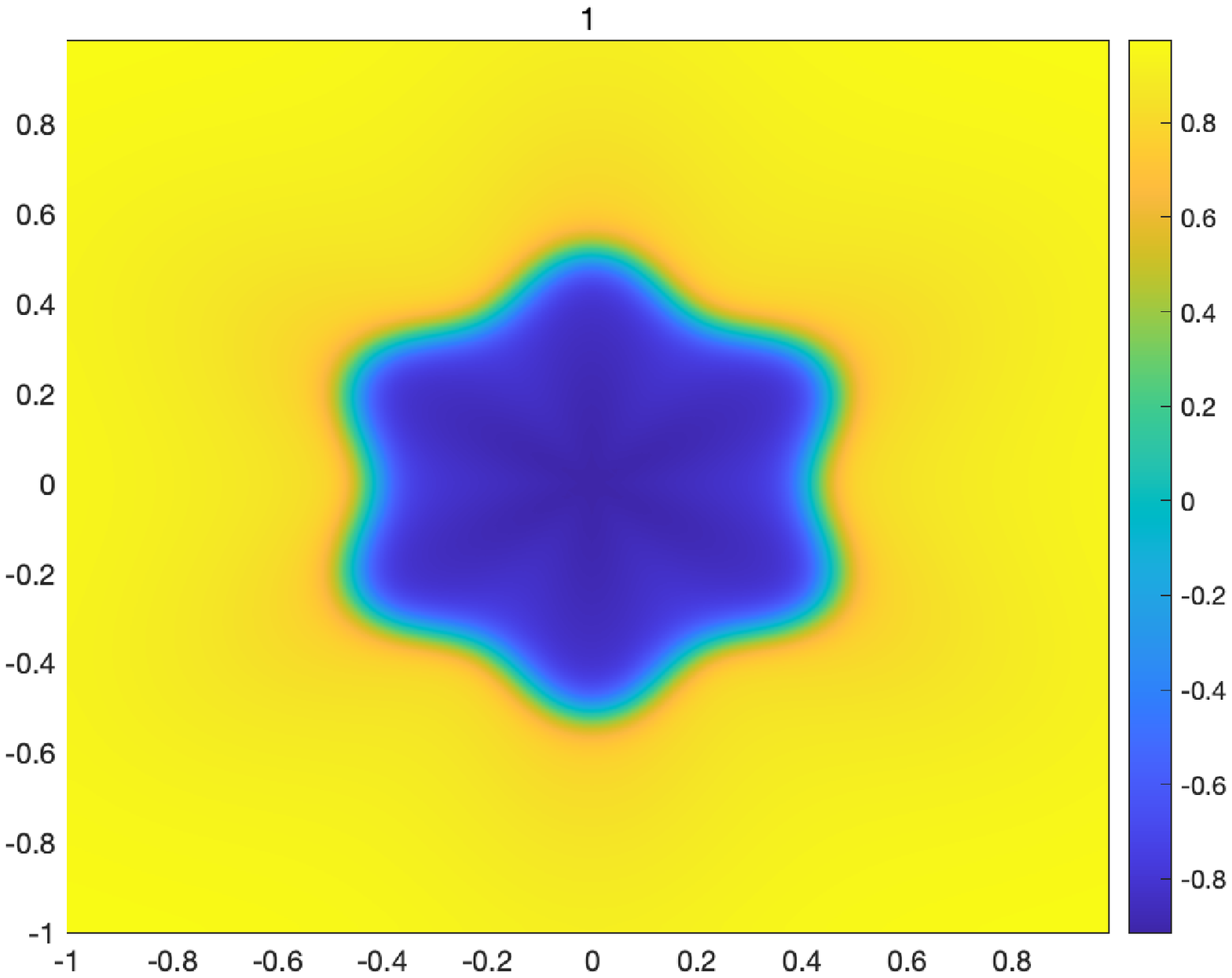}
      \end{minipage}
      }%
      \subfigure[$\alpha = 0.3,~t=8$]{
      \begin{minipage}[t]{0.25\textwidth}
        \centering
        \includegraphics[width=\textwidth]{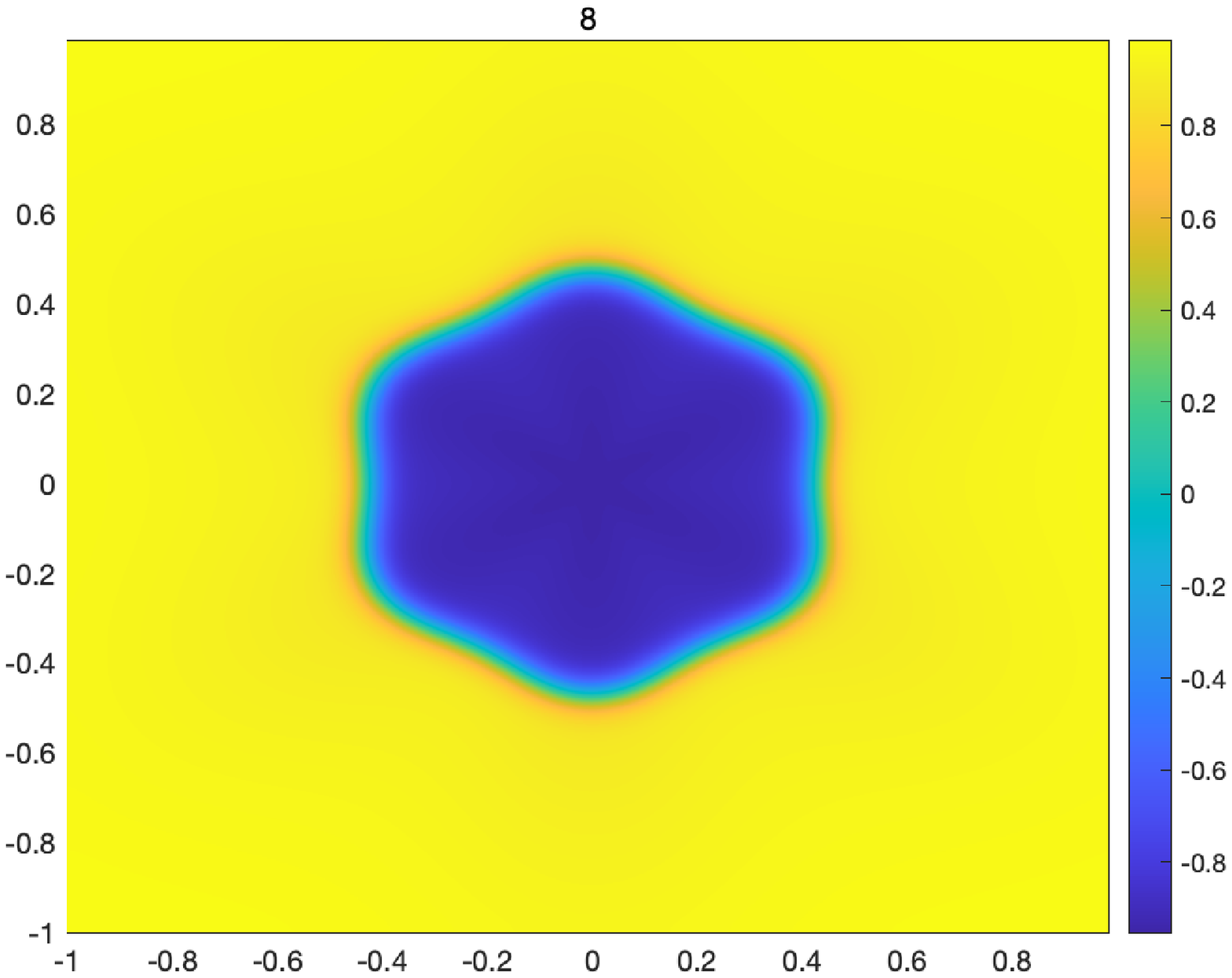}
      \end{minipage}
      }%
      \subfigure[$\alpha = 0.3,~t=64$]{
        \begin{minipage}[t]{0.25\textwidth}
          \centering
          \includegraphics[width=\textwidth]{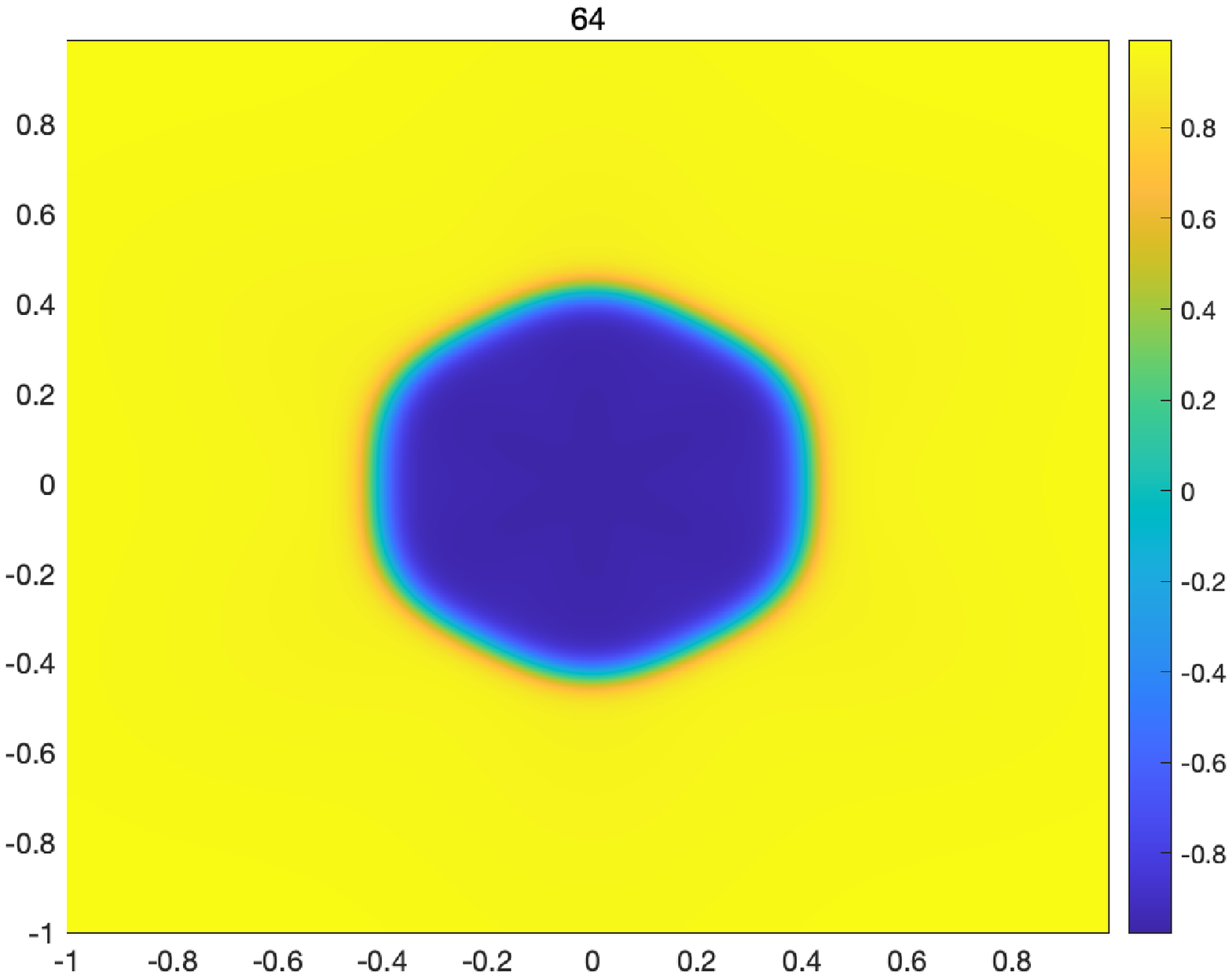}
        \end{minipage}
      }%
      \caption{Snapshots for the TFAC equation for $\alpha = 0.9, 0.6, 0.3$ (top, middle, bottom row, respectively, computed by the L1-IMEX scheme).}\label{fig-ac-evloving}
  \centering
\end{figure}

We study the energy dissipation numerically and several figures are presented to explain the results.
Figure \ref{fig-ac-evloving} illustrates the phases for $\alpha = 0.9,0.6,0.3$ at different time.
Figure \ref{fig-ac-energy-1} verify the energy dissipation property.
We also compare the differences between $\tilde E$ and $E$ when time is large.
In Figure \ref{fig-ac-energy-2} it can be observed that $\tilde E(t)\rightarrow E(t)$ as $t\to\infty$, and the larger $\alpha$ is, the smaller the difference between $\tilde E$ and $E$ is.
\begin{figure}[!h]
  \centering
  \subfigure{
  \begin{minipage}[t]{0.35\textwidth}
  \centering
  \includegraphics[width=\textwidth]{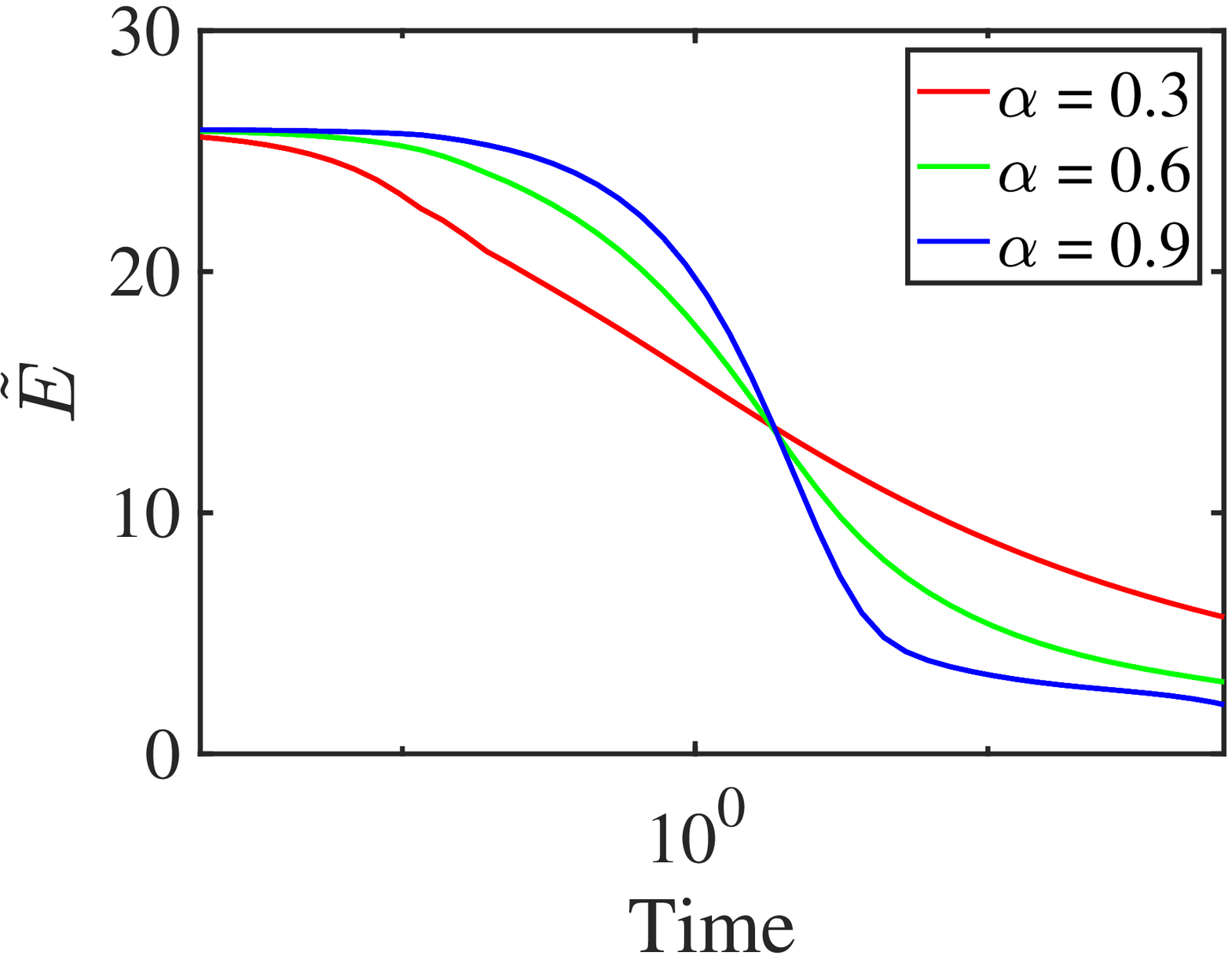}
  \end{minipage}
  }%
  \subfigure{
    \begin{minipage}[t]{0.35\textwidth}
      \centering
      \includegraphics[width=\textwidth]{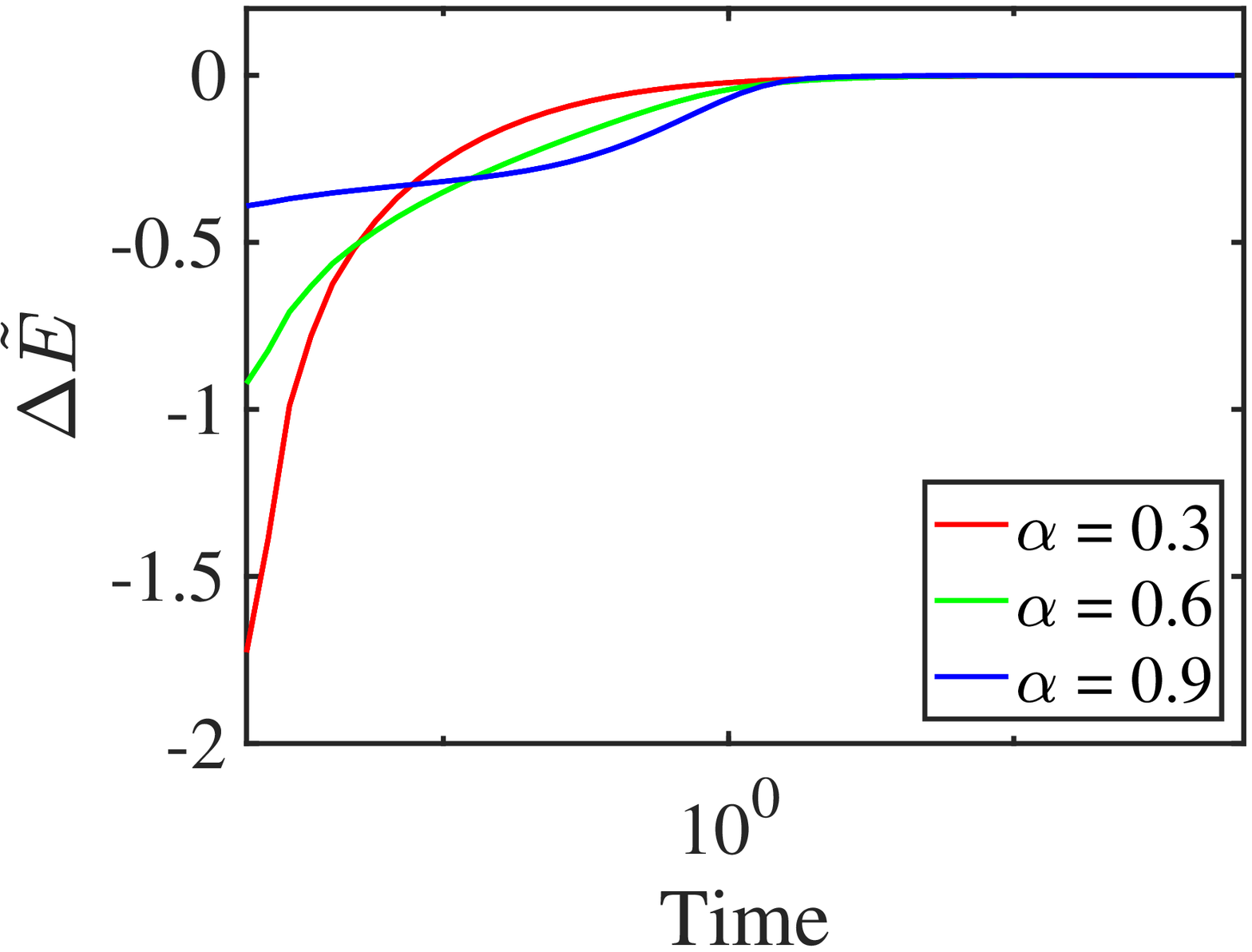}
    \end{minipage}
  }%
      \caption{Evolution of the modified energy $\tilde E$ and $\Delta \tilde E = \tilde E(t)-\tilde E(t-\Delta t)$ for the TFAC equation, computed by the L1-IMEX scheme.}\label{fig-ac-energy-1}
  \centering
  \subfigure{
  \begin{minipage}[t]{0.32\textwidth}
  \centering
  \includegraphics[width=\textwidth]{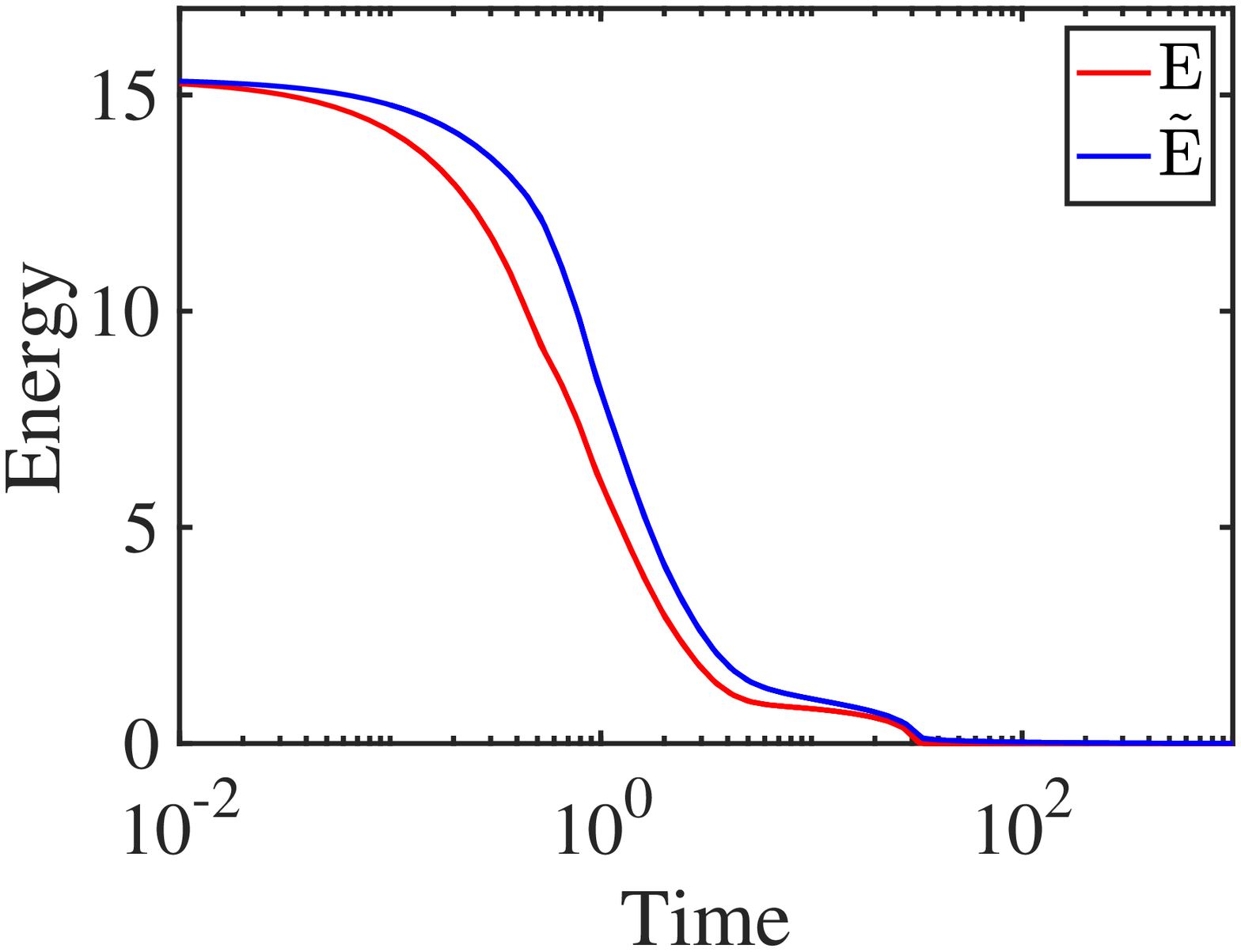}
  \end{minipage}
  }%
  \subfigure{
    \begin{minipage}[t]{0.32\textwidth}
      \centering
      \includegraphics[width=\textwidth]{figures/AClongalpha06.eps}\label{ac-energy-comparision}
    \end{minipage}
  }%
  \subfigure{
    \begin{minipage}[t]{0.32\textwidth}
    \centering
    \includegraphics[width=\textwidth]{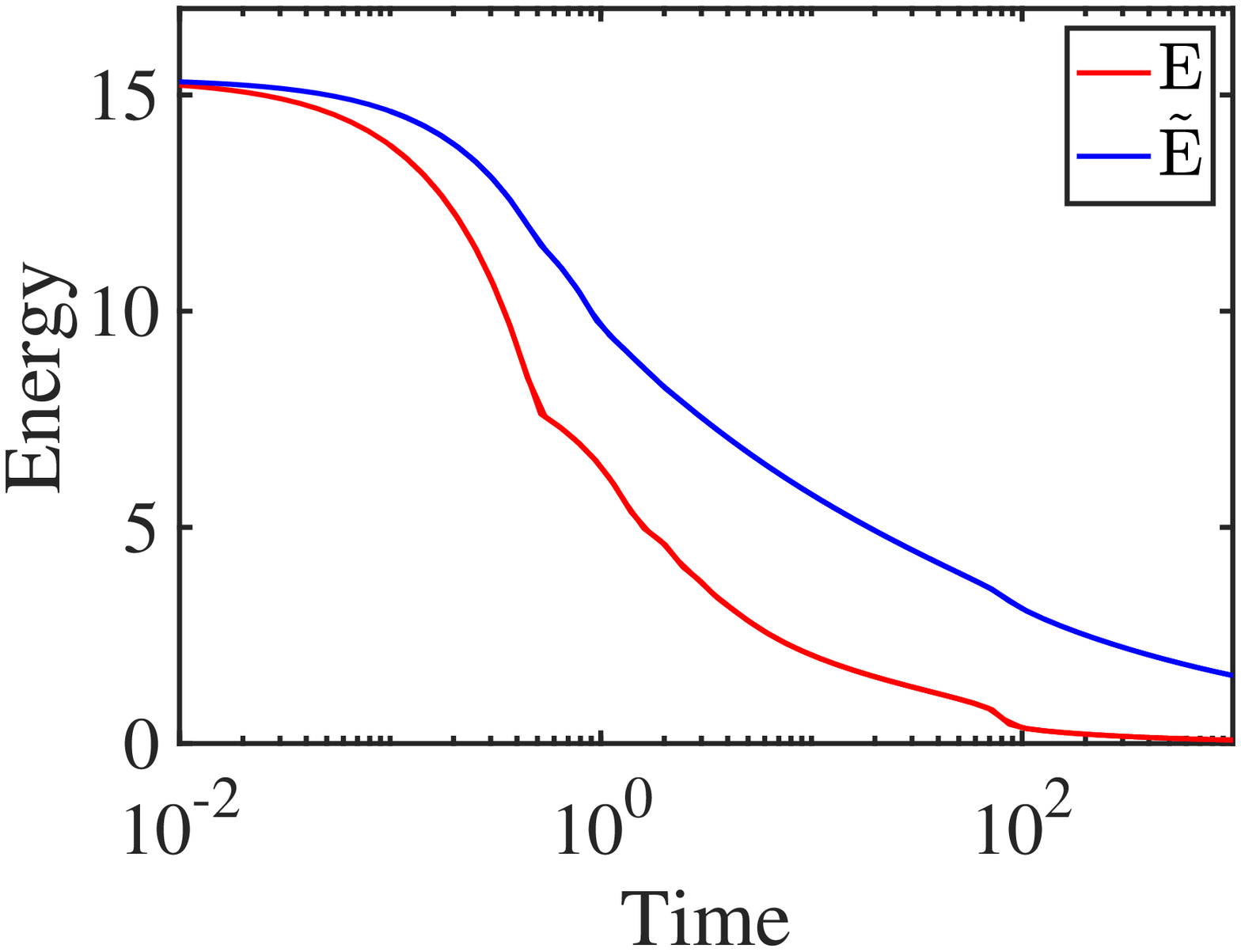}
  \end{minipage}
    }%
      \caption{Comparison of the original energy $E$ and the modified energy $\tilde E$ for the TFAC equation with $\alpha = 0.9, 0.6, 0.3$ (left, middle, right column, respectively), computed by the L1-IMEX scheme.}\label{fig-ac-energy-2}
  \centering
\end{figure}

We also test the stabilized L2 scheme. We take $\Omega = [0,2\pi]^2$, $\varepsilon = 0.1$, and $\gamma = 1$. We use the stabilized L2 scheme \eqref{eq:L2AB} with $S = 1$. Moreover, $128\times128$ Fourier modes and $\tau = 0.05$ are taken. The initial state is given by
\begin{equation}\label{eq:init_7circ}
  \phi_0(x,y) = -1 + \sum_{i=1}^7 f\left( \sqrt{(x-\xi_{1i})^2+(y-\xi_{2i})^2} - r_i\right),
  \end{equation}
  where
  \begin{equation}
  f(s) = \left\{
  \begin{aligned}
  & 2 e^{-\varepsilon^2/s^2} &&\mbox{if } s<0,\\
  & 0 && \mbox{otherwise,}\notag
  \end{aligned}
  \right.
  \end{equation}
and the centers and radii are given by Table \ref{tab:xyr}.
  \begin{table}[htb!]
    \begin{center}
      \caption{Centers $(x_i,y_i)$ and radii $r_i$ in the initial condition \eqref{eq:init_7circ}.}\label{tab:xyr}
      \scalebox{0.95}{
        \begin{tabular}{c|*{7}{c}} \hline\hline
          $i$   &1 & 2     &3   &4  & 5 & 6 & 7 \\  \hline
          $\xi_{1i}$  & $\pi/2$   & $\pi/4$   & $\pi/2$ & $\pi$ & $3\pi/2$ & $\pi$ & $3\pi/2$\\
        $\xi_{2i}$  & $\pi/2$   & $3\pi/4$   & $5\pi/4$ & $\pi/4$ & $\pi/4$ & $\pi$ & $3\pi/2$ \\
        $r_i$  & $\pi/5$   & $2\pi/15$   & $2\pi/15$ & $\pi/10$ & $\pi/10$ & $\pi/4$ & $\pi/4$ \\\hline\hline
        \end{tabular}}
    \end{center}
  \end{table}

The results are depicted in Figure \ref{fig-ac-evloving-L2} and Figure \ref{fig-ac-energy-L2-2}. 
It can be seen that $\tilde E$ of the L2 scheme is also decreasing w.r.t. time and when $\alpha$ is small, the decreasing rate is slow. 
The results of the L2 scheme are similar to the results of the L1 scheme.

  \begin{figure}[!t]
    \centering
    \subfigure[$\alpha=0.8,~t=2.5$]{
    \begin{minipage}[t]{0.25\textwidth}
    \centering
    \includegraphics[width=\textwidth]{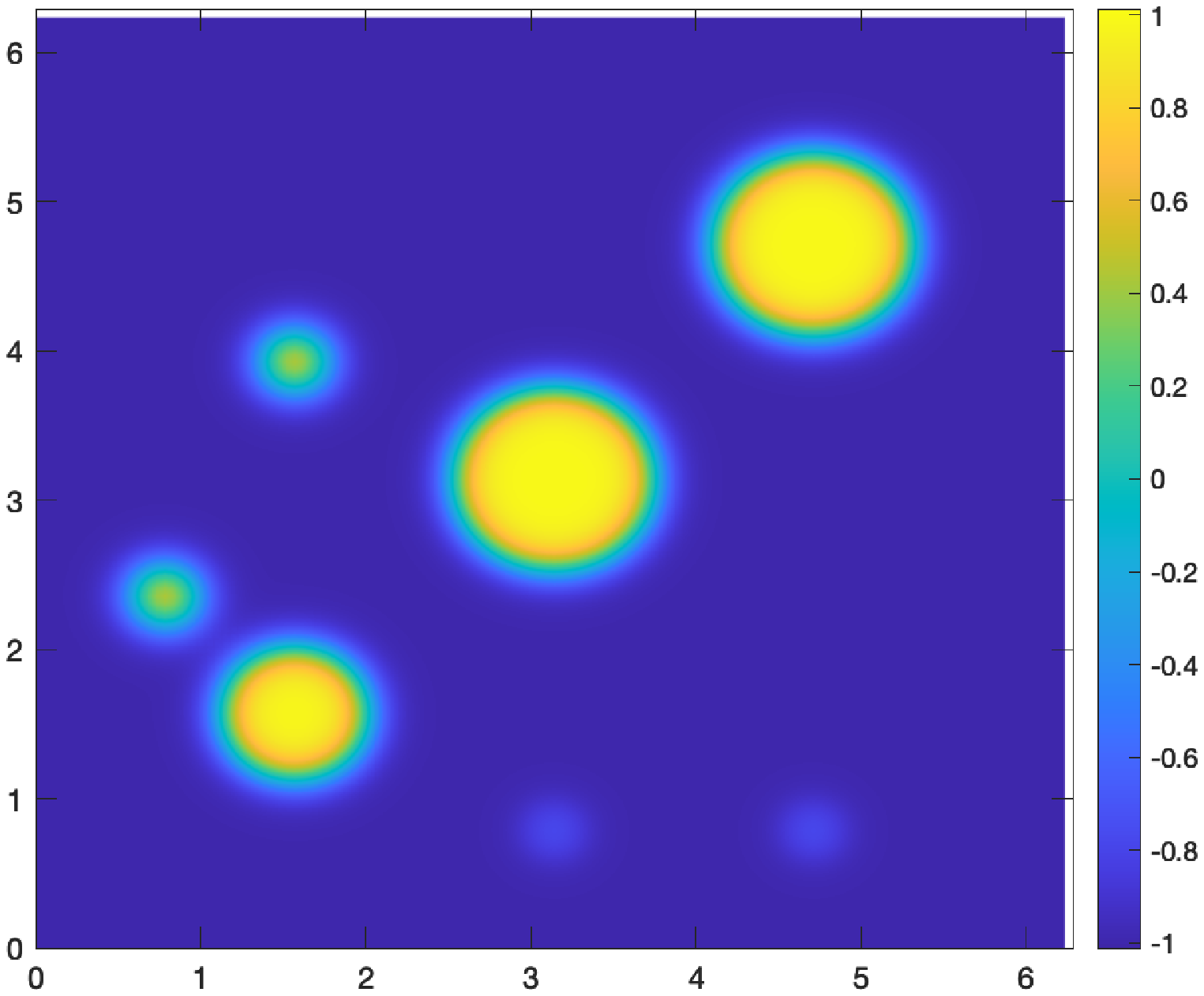}
    \end{minipage}
    }%
    \subfigure[$\alpha=0.8,~t=10$]{
    \begin{minipage}[t]{0.25\textwidth}
      \centering
      \includegraphics[width=\textwidth]{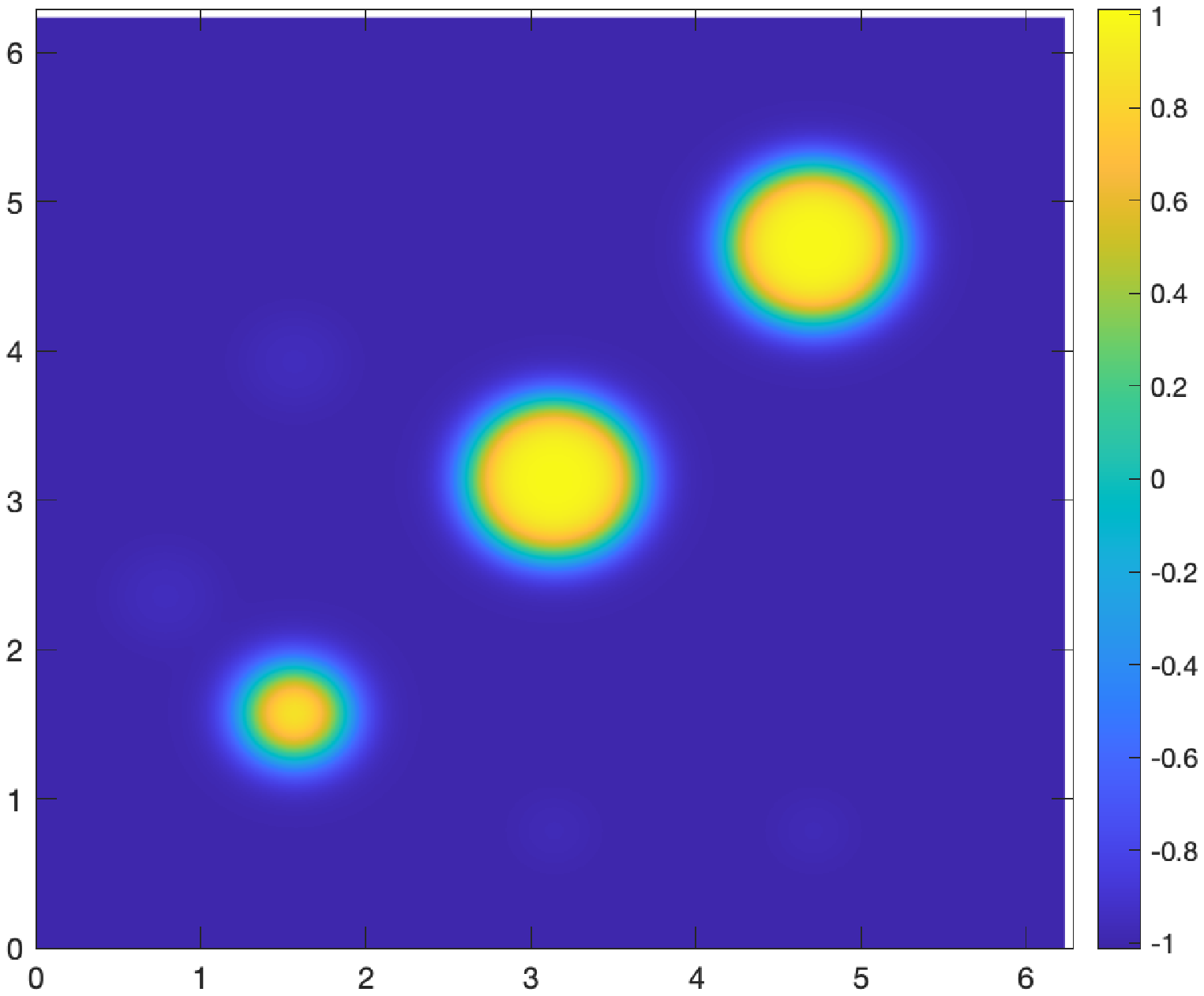}
    \end{minipage}
    }%
    \subfigure[$\alpha=0.8,~t=50$]{
      \begin{minipage}[t]{0.25\textwidth}
        \centering
        \includegraphics[width=\textwidth]{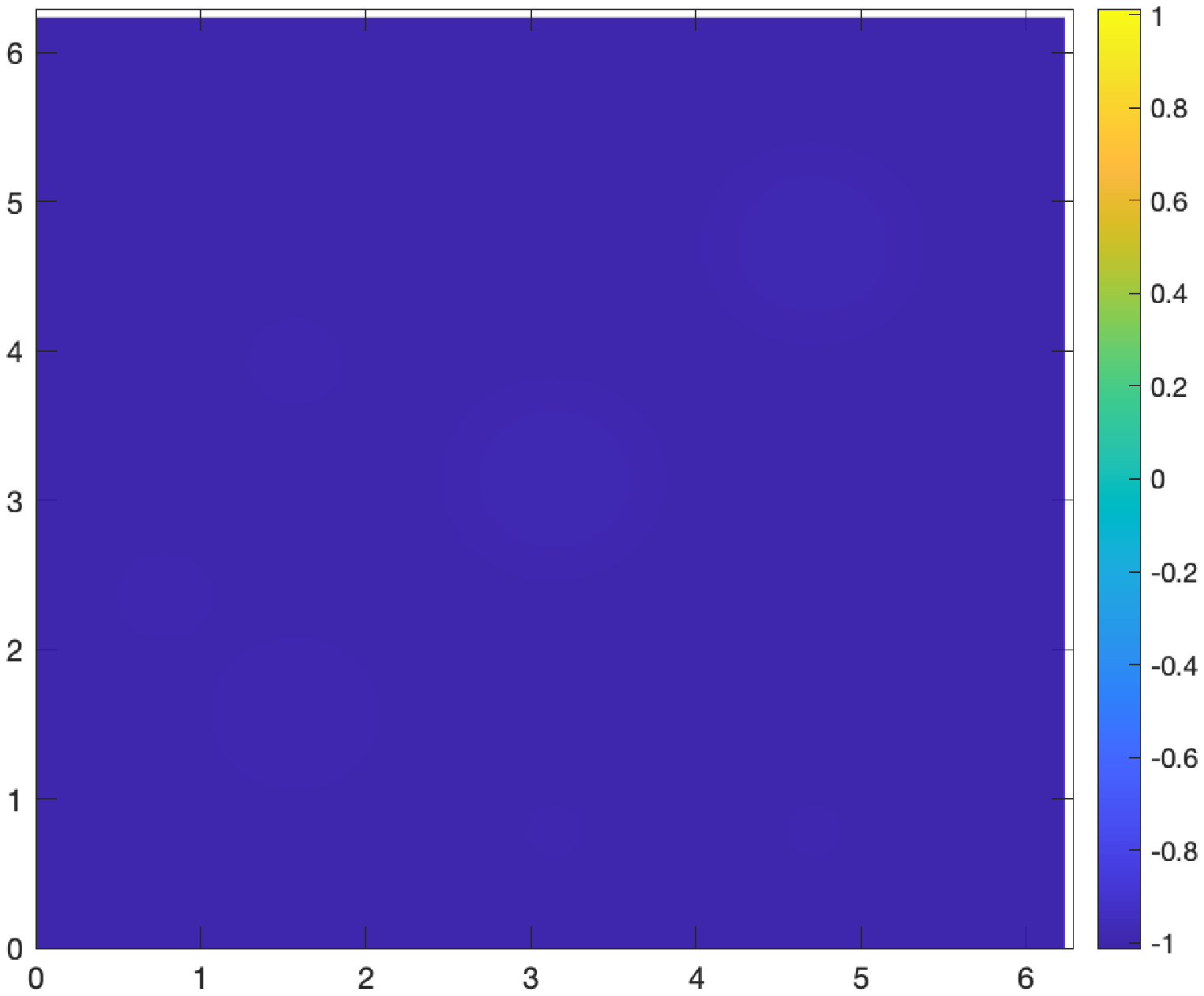}
      \end{minipage}
    }%
  
    \subfigure[$\alpha=0.6,~t=2.5$]{
      \begin{minipage}[t]{0.25\textwidth}
      \centering
      \includegraphics[width=\textwidth]{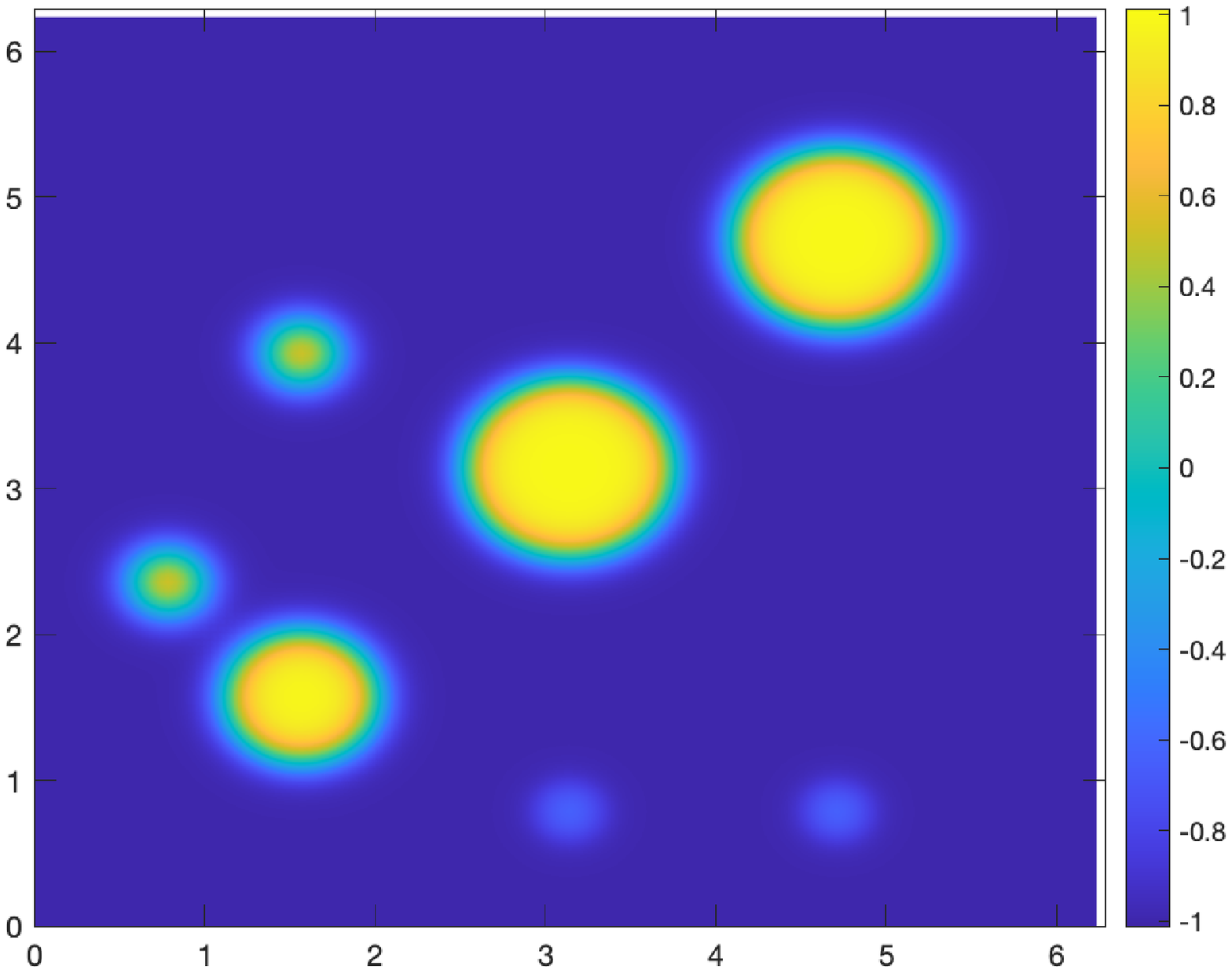}
      \end{minipage}
      }%
      \subfigure[$\alpha=0.6,~t=10$]{
      \begin{minipage}[t]{0.25\textwidth}
        \centering
        \includegraphics[width=\textwidth]{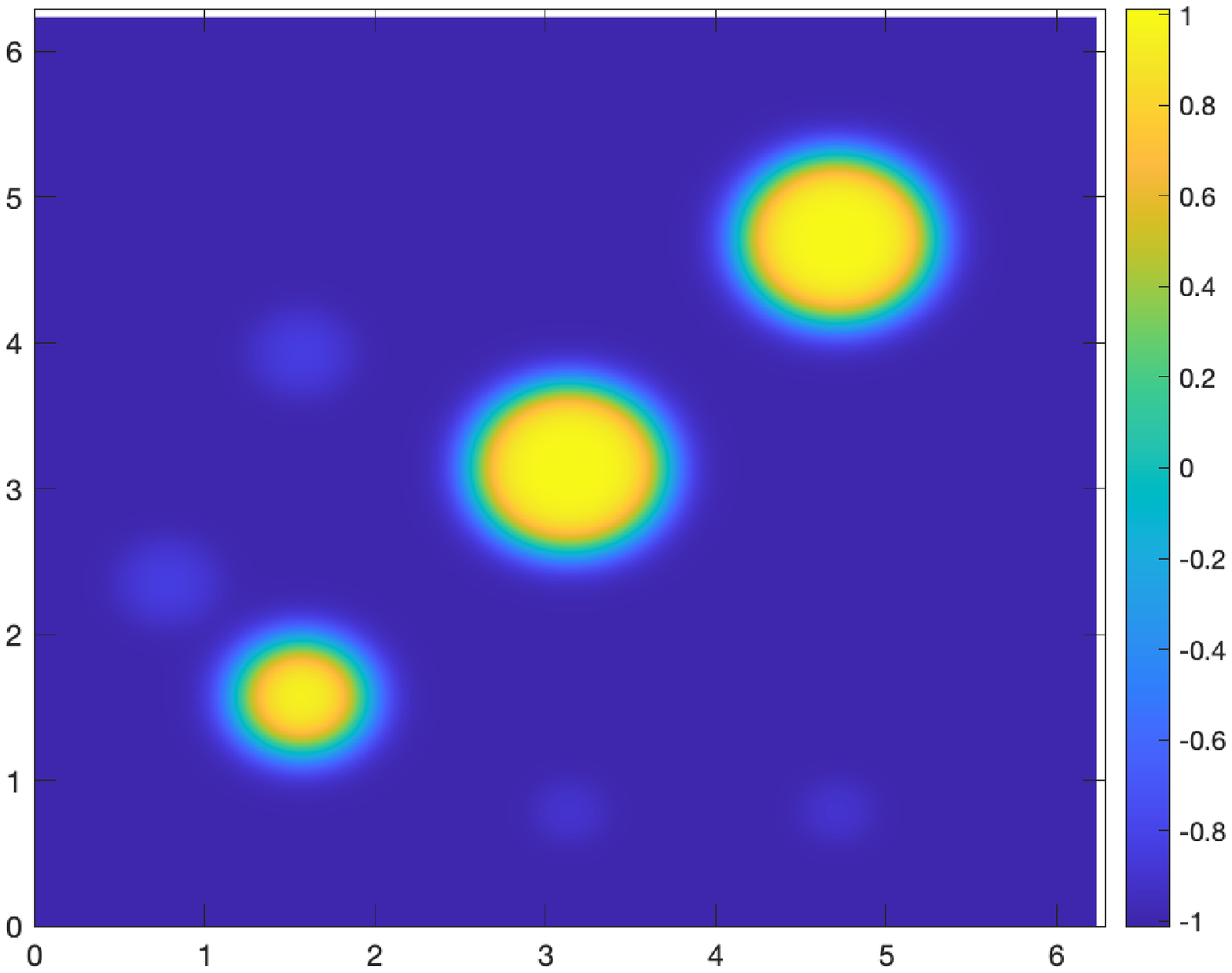}
      \end{minipage}
      }%
      \subfigure[$\alpha=0.6,~t=50$]{
        \begin{minipage}[t]{0.25\textwidth}
          \centering
          \includegraphics[width=\textwidth]{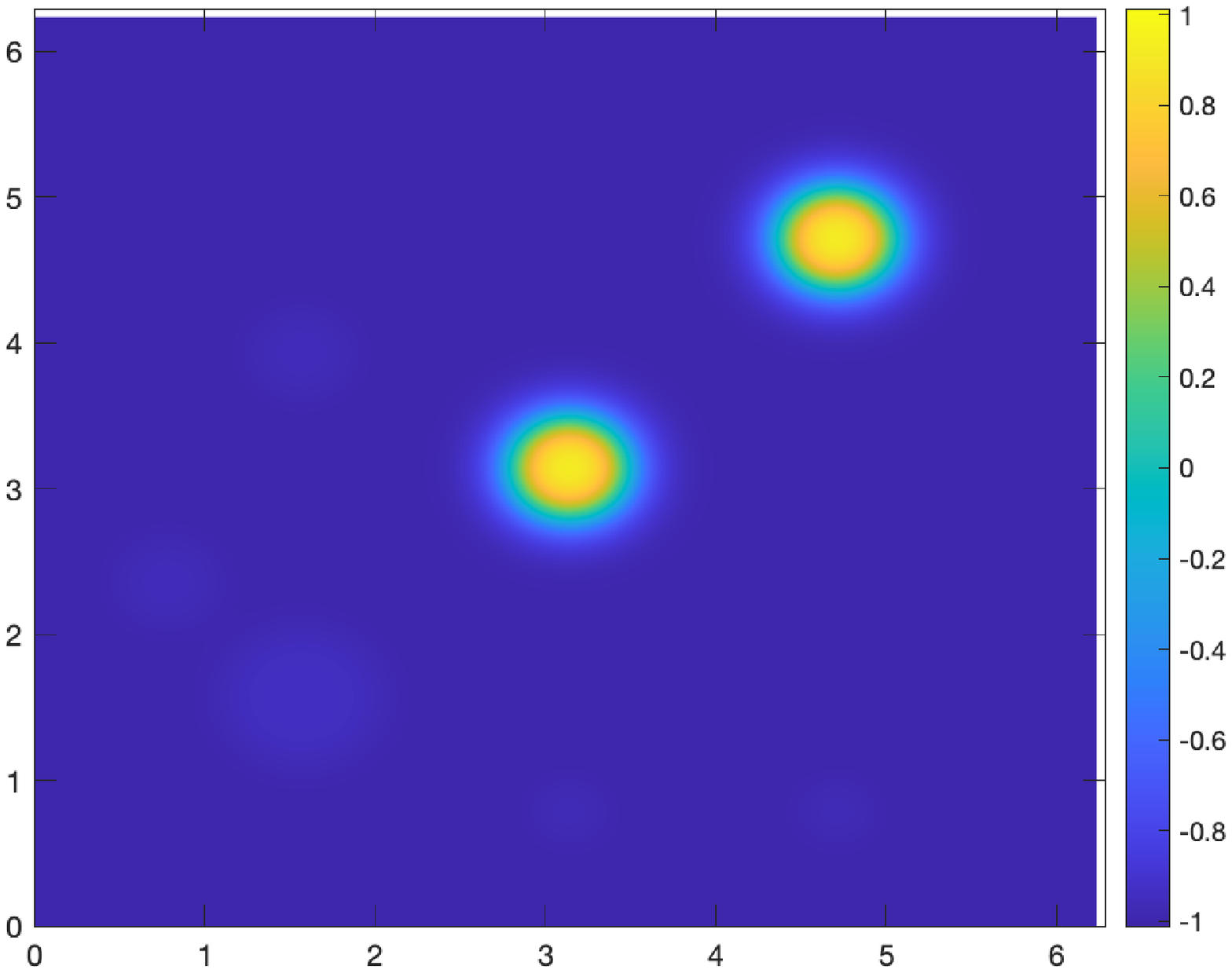}
        \end{minipage}
      }%
  
      \subfigure[$\alpha=0.4,~t=2.5$]{
        \begin{minipage}[t]{0.25\textwidth}
        \centering
        \includegraphics[width=\textwidth]{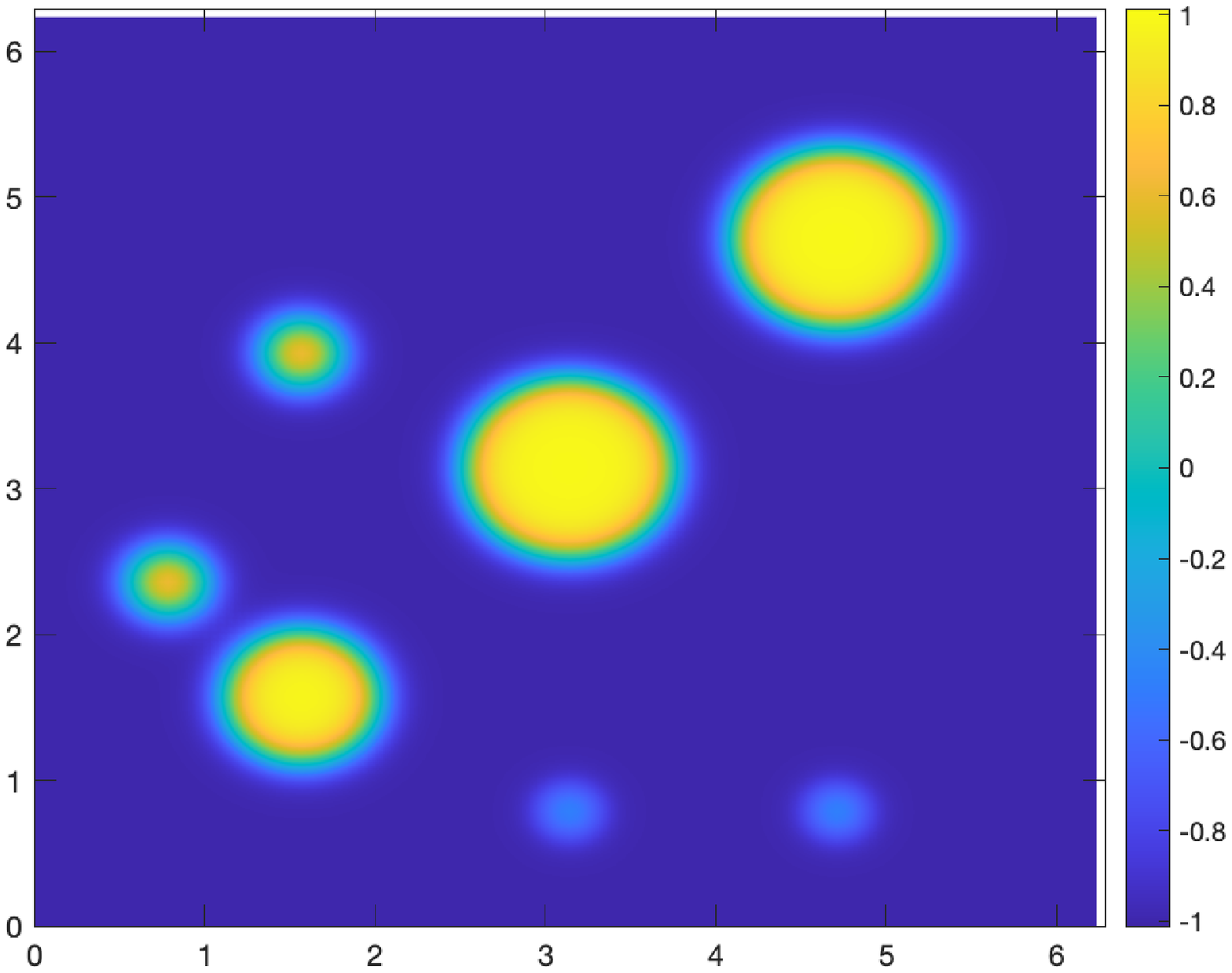}
        \end{minipage}
        }%
        \subfigure[$\alpha=0.4,~t=10$]{
        \begin{minipage}[t]{0.25\textwidth}
          \centering
          \includegraphics[width=\textwidth]{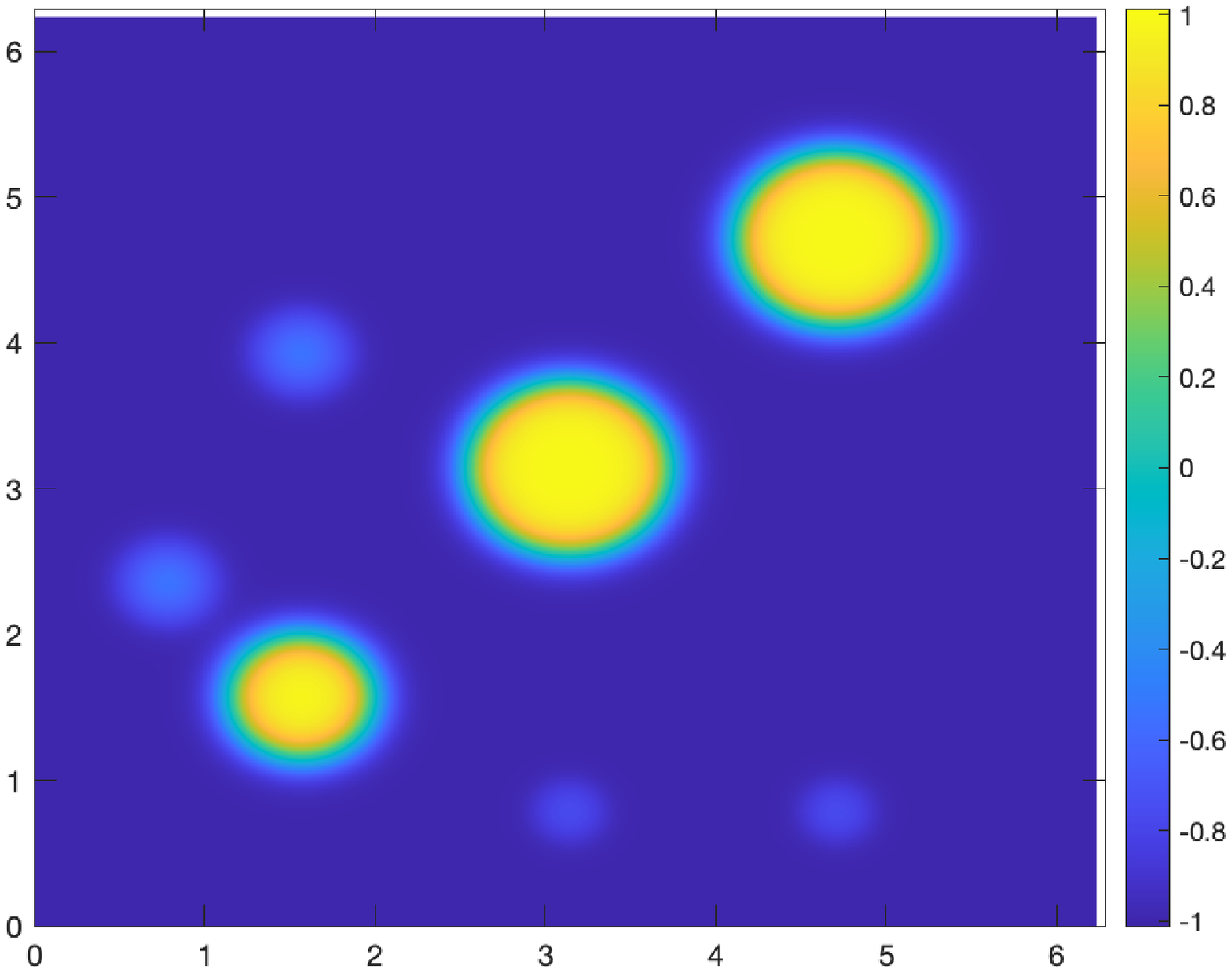}
        \end{minipage}
        }%
        \subfigure[$\alpha=0.4,~t=50$]{
          \begin{minipage}[t]{0.25\textwidth}
            \centering
            \includegraphics[width=\textwidth]{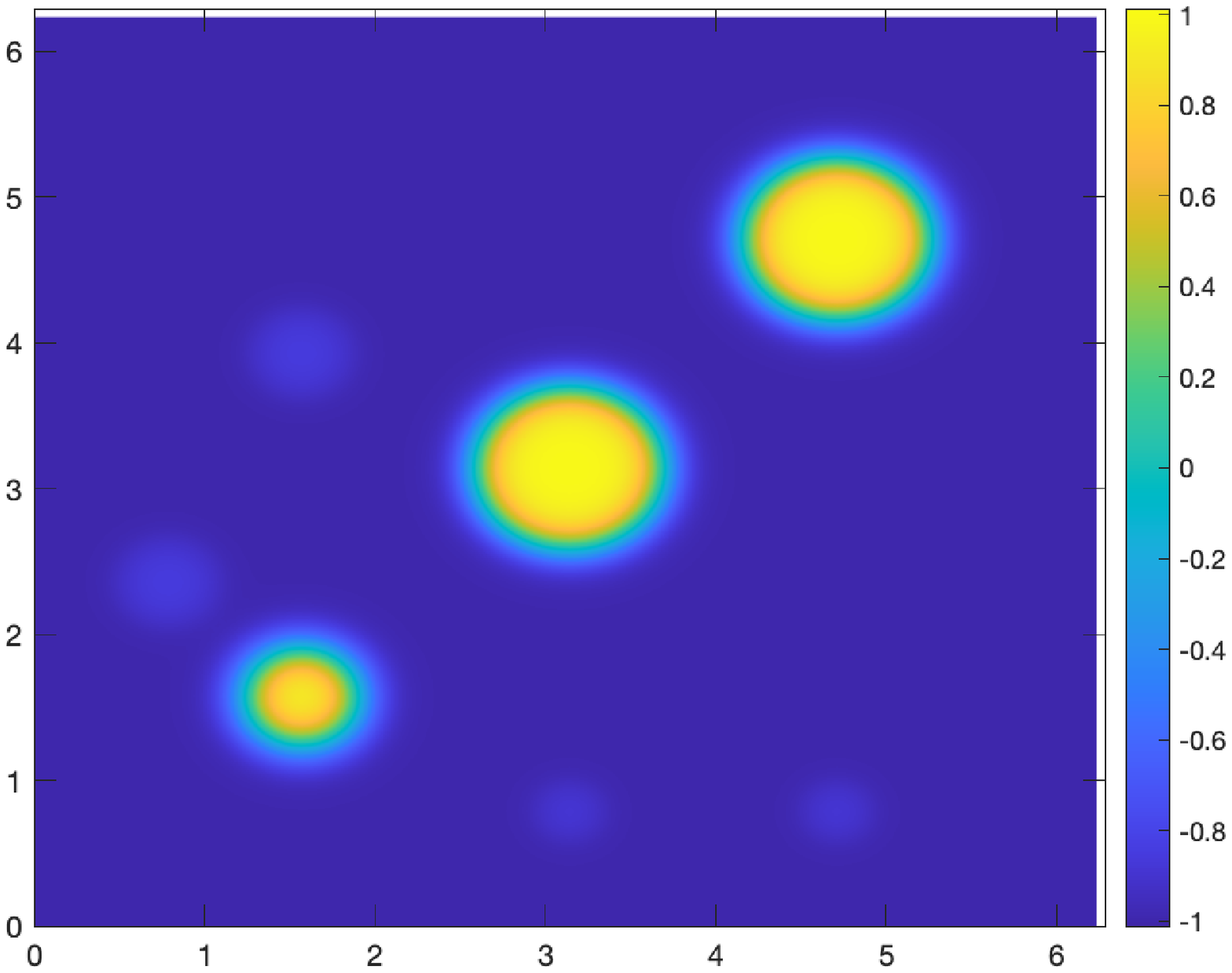}
          \end{minipage}
        }%
        \caption{Snapshots for the TFAC equation for $\alpha = 0.8, 0.6, 0.4$ (top, middle, bottom row, respectively, computed by the L2-IMEX scheme).}\label{fig-ac-evloving-L2}
    \centering
  \end{figure}

  \begin{figure}[!h]
    \centering
    \subfigure{
    \begin{minipage}[t]{0.32\textwidth}
    \centering
    \includegraphics[width=\textwidth]{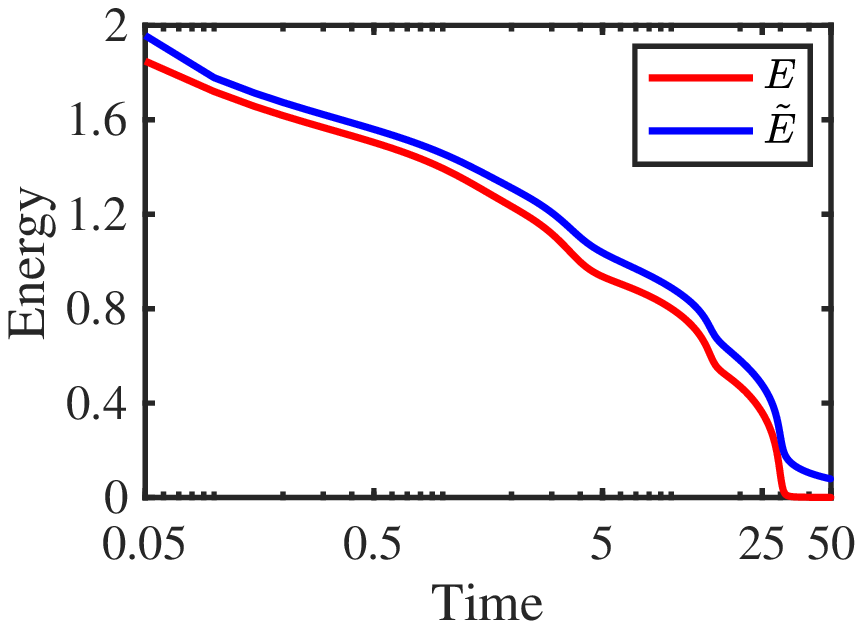}
    \end{minipage}
    }%
    \subfigure{
      \begin{minipage}[t]{0.32\textwidth}
        \centering
        \includegraphics[width=\textwidth]{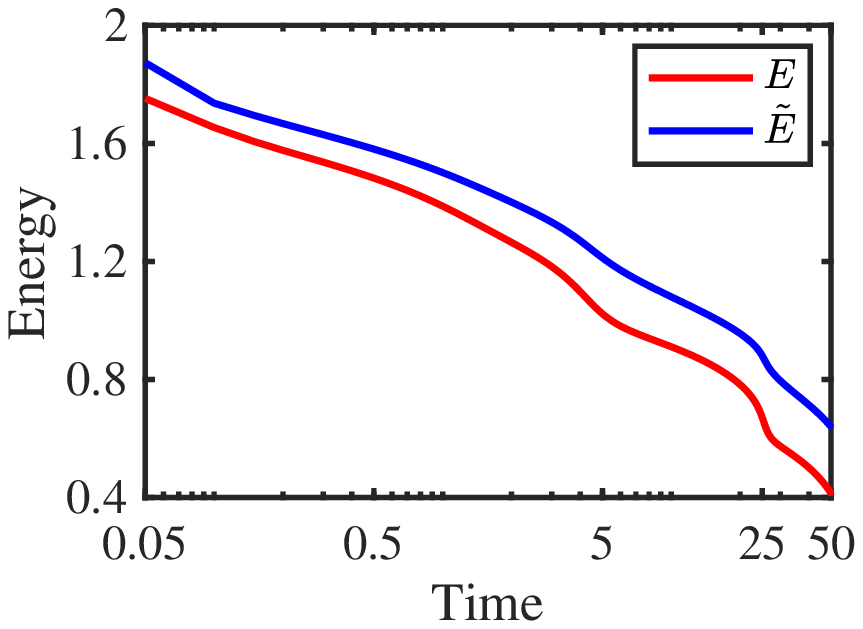}
      \end{minipage}
    }%
    \subfigure{
      \begin{minipage}[t]{0.32\textwidth}
      \centering
      \includegraphics[width=\textwidth]{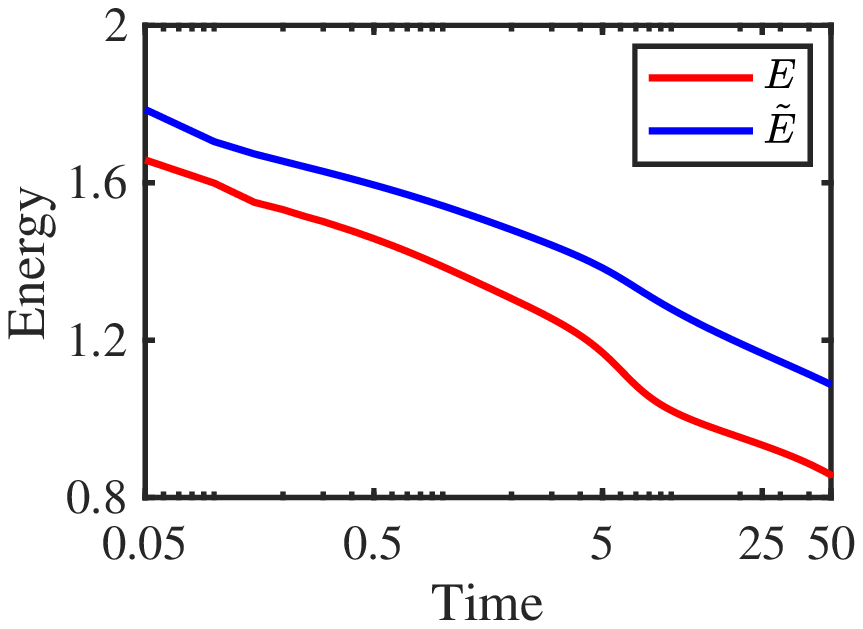}
    \end{minipage}
      }%
        \caption{Comparison of the original energy $E$ and the modified energy $\tilde E$ for the TFAC equation with $\alpha = 0.8, 0.6, 0.4$ (left, middle, right column, respectively), computed by the L2-IMEX scheme.}\label{fig-ac-energy-L2-2}
    \centering
  \end{figure}

\subsection{TFCH equation}
For the TFCH equation, we take $L_x = L_y = 2\pi$, $\varepsilon = 0.05$.
{We set the diffusion constant to be $\gamma = 0.02$ and $S$ in the stabilized scheme \eqref{schemeac} to be $S = 5\gamma = 0.1$.}
Moreover, $128\times 128$ Fourier modes and $\tau = 0.1$ are taken The initial state is taken as an uniformly random distribution field in $[-1,1]$.
\begin{figure}[!h]
  \centering
  \subfigure[$\alpha =0.9,~t=1$]{
  \begin{minipage}[t]{0.25\textwidth}
  \centering
  \includegraphics[width=\textwidth]{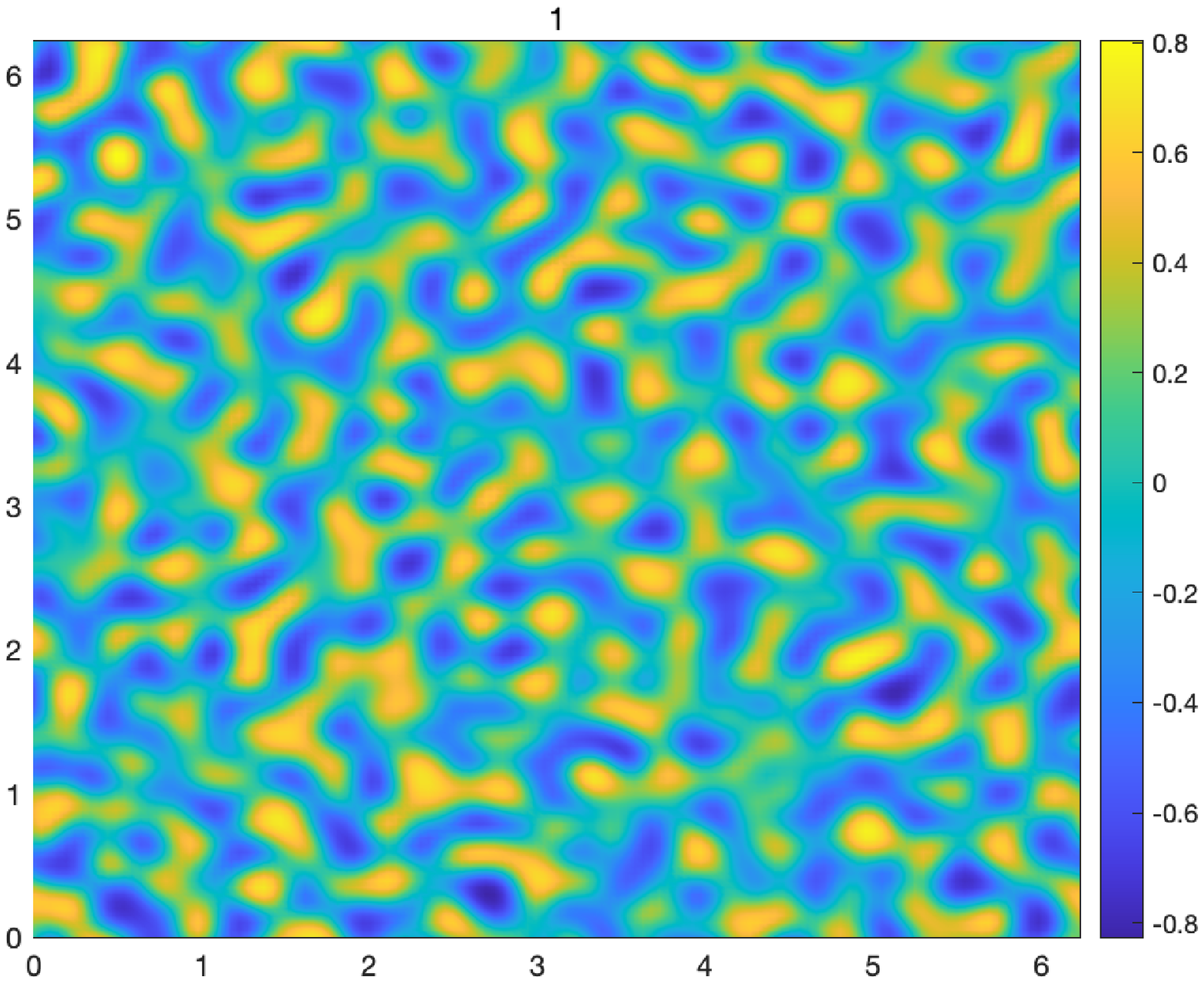}
  \end{minipage}
  }%
  \subfigure[$\alpha =0.9,~t=8$]{
  \begin{minipage}[t]{0.25\textwidth}
    \centering
    \includegraphics[width=\textwidth]{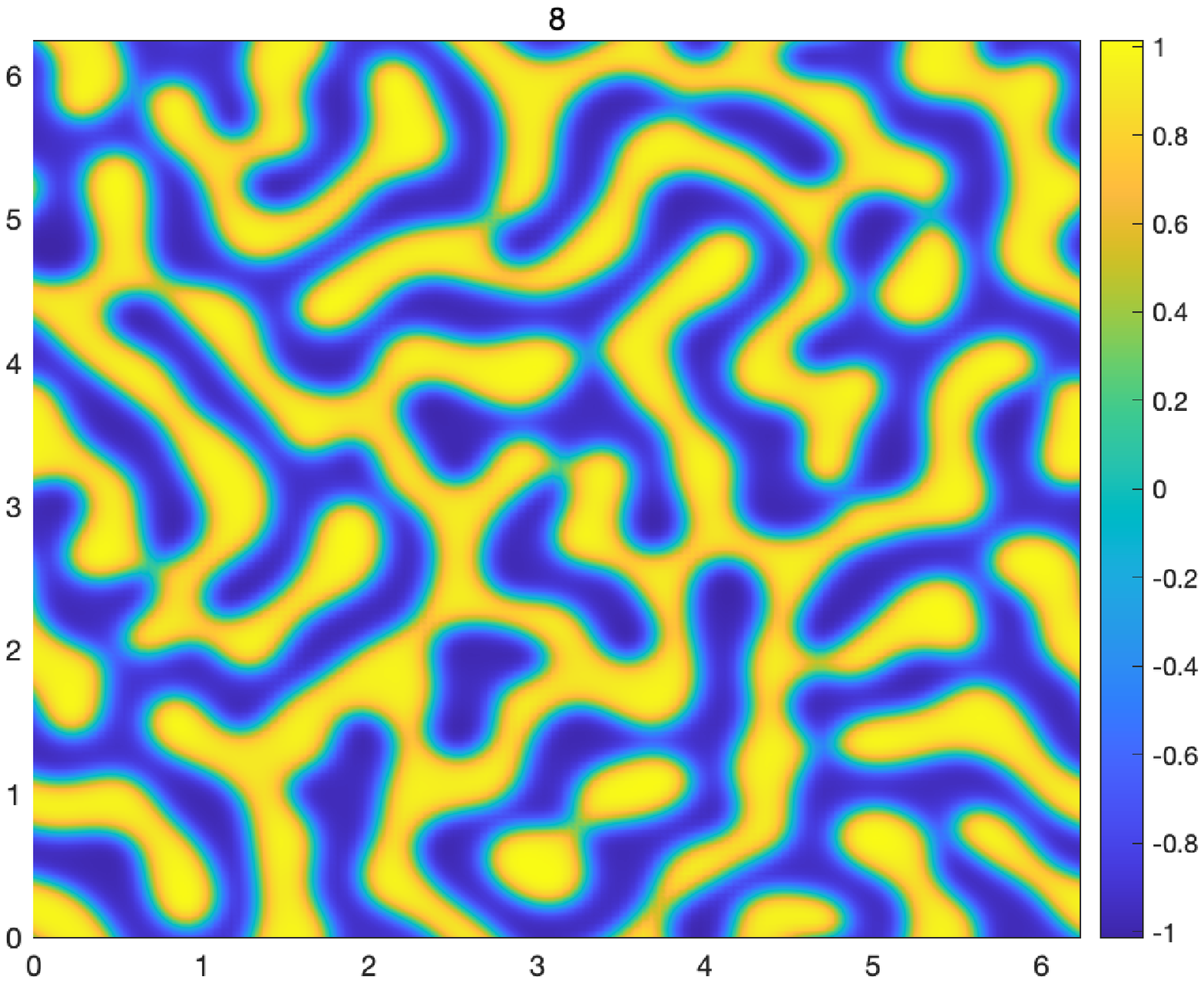}
  \end{minipage}
  }%
  \subfigure[$\alpha =0.9,~t=64$]{
    \begin{minipage}[t]{0.25\textwidth}
      \centering
      \includegraphics[width=\textwidth]{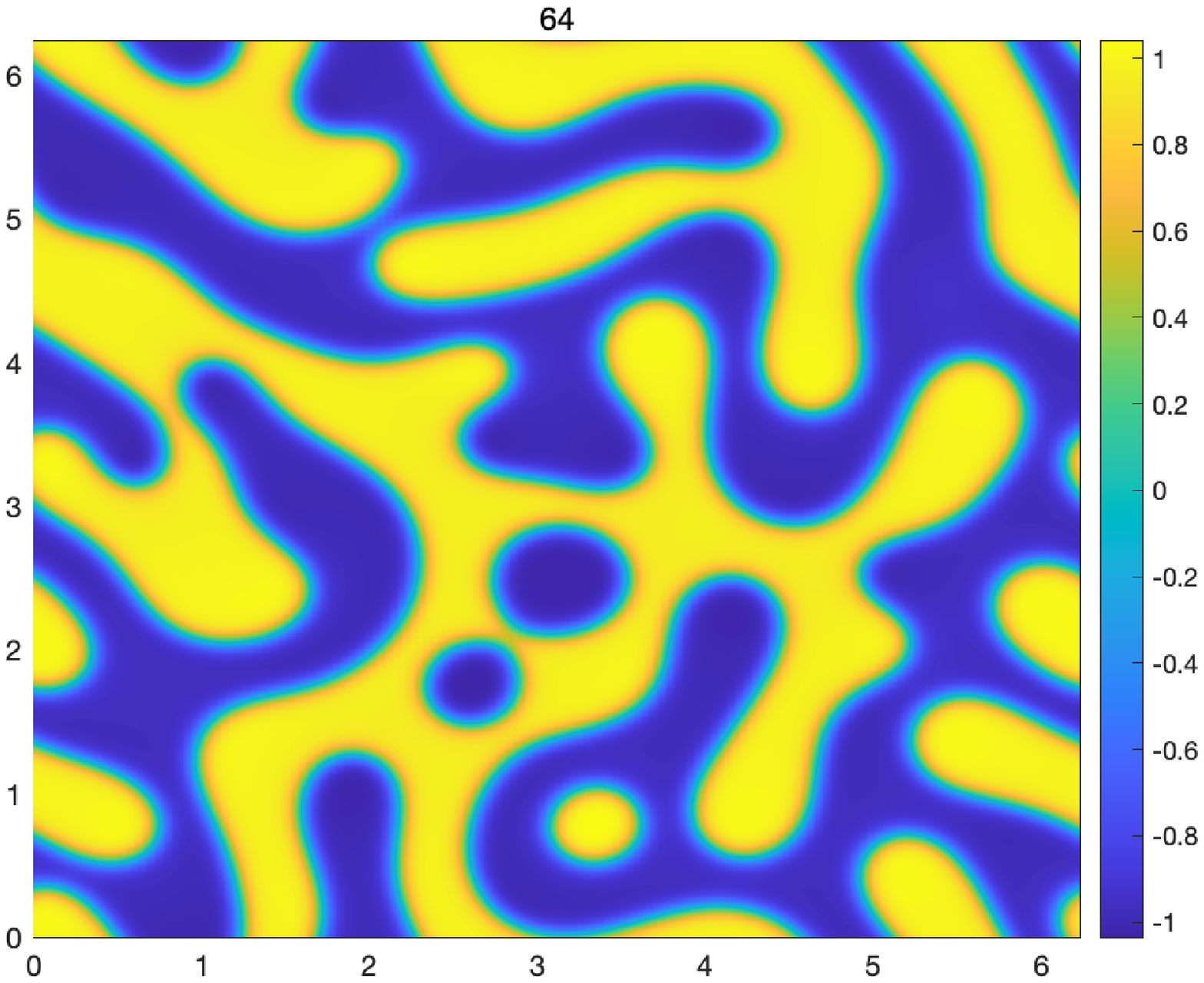}
    \end{minipage}
  }%

  \subfigure[$\alpha =0.6,~t=1$]{
    \begin{minipage}[t]{0.25\textwidth}
    \centering
    \includegraphics[width=\textwidth]{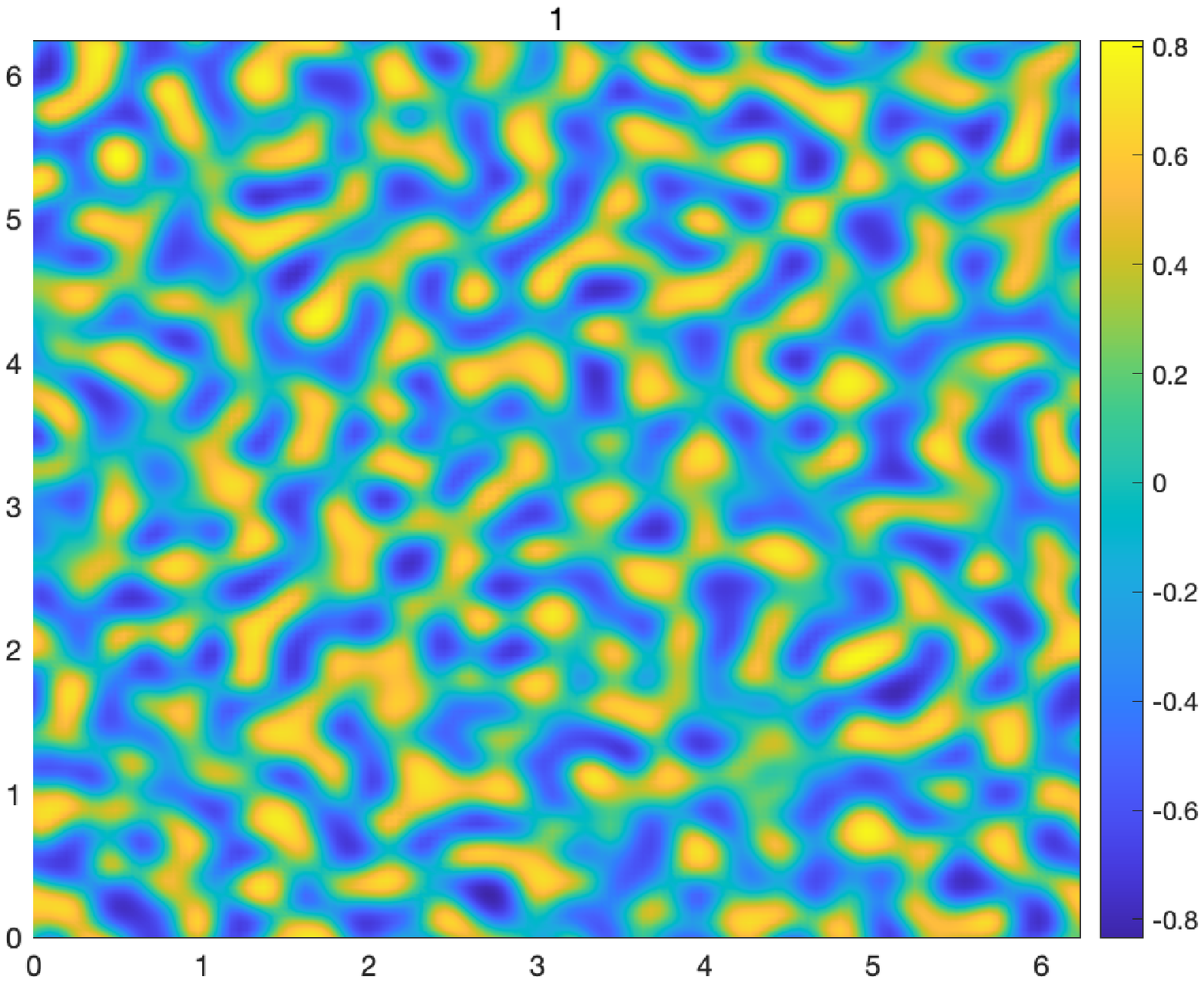}
    \end{minipage}
    }%
    \subfigure[$\alpha =0.6,~t=8$]{
    \begin{minipage}[t]{0.25\textwidth}
      \centering
      \includegraphics[width=\textwidth]{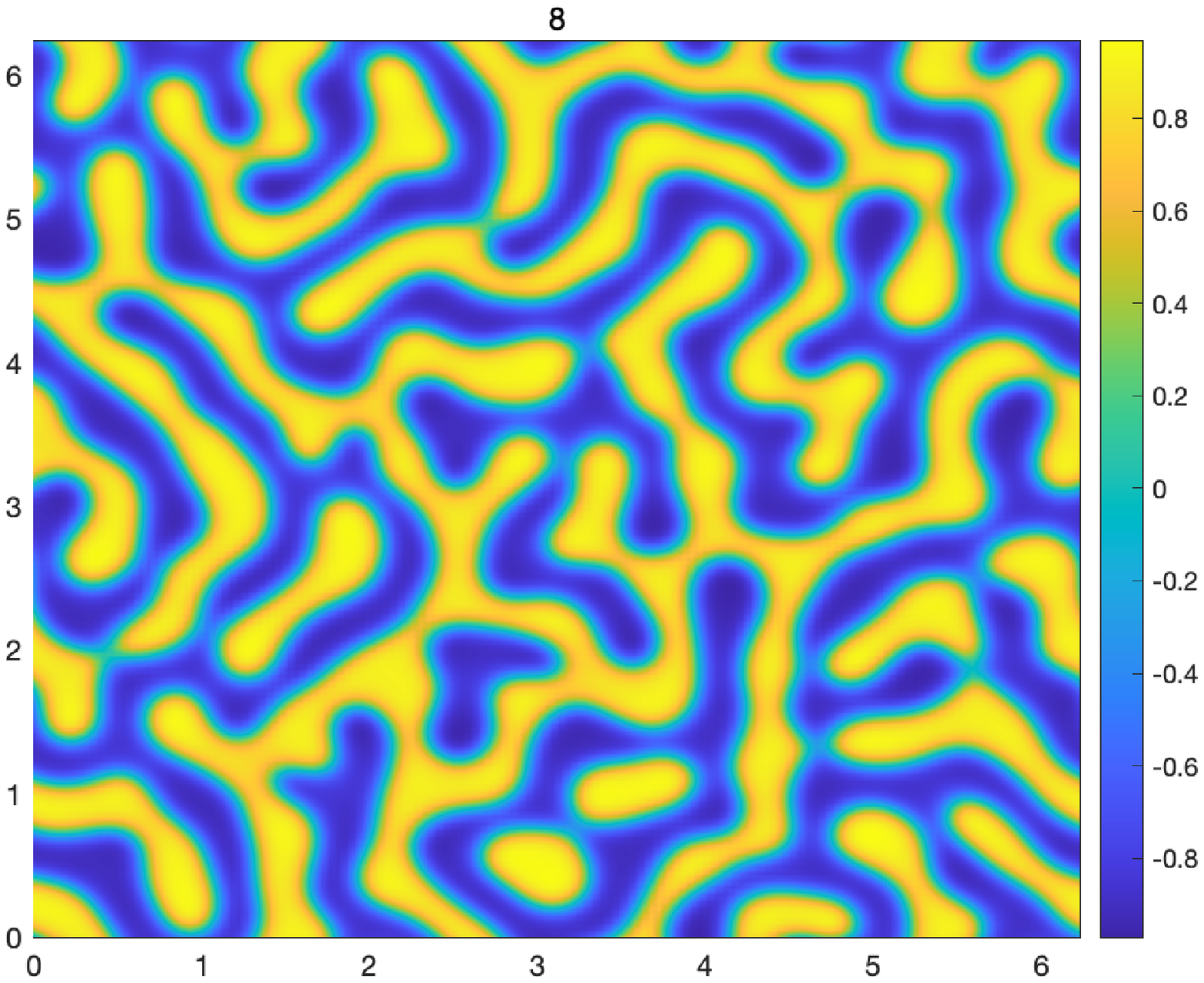}
    \end{minipage}
    }%
    \subfigure[$\alpha =0.6,~t=64$]{
      \begin{minipage}[t]{0.25\textwidth}
        \centering
        \includegraphics[width=\textwidth]{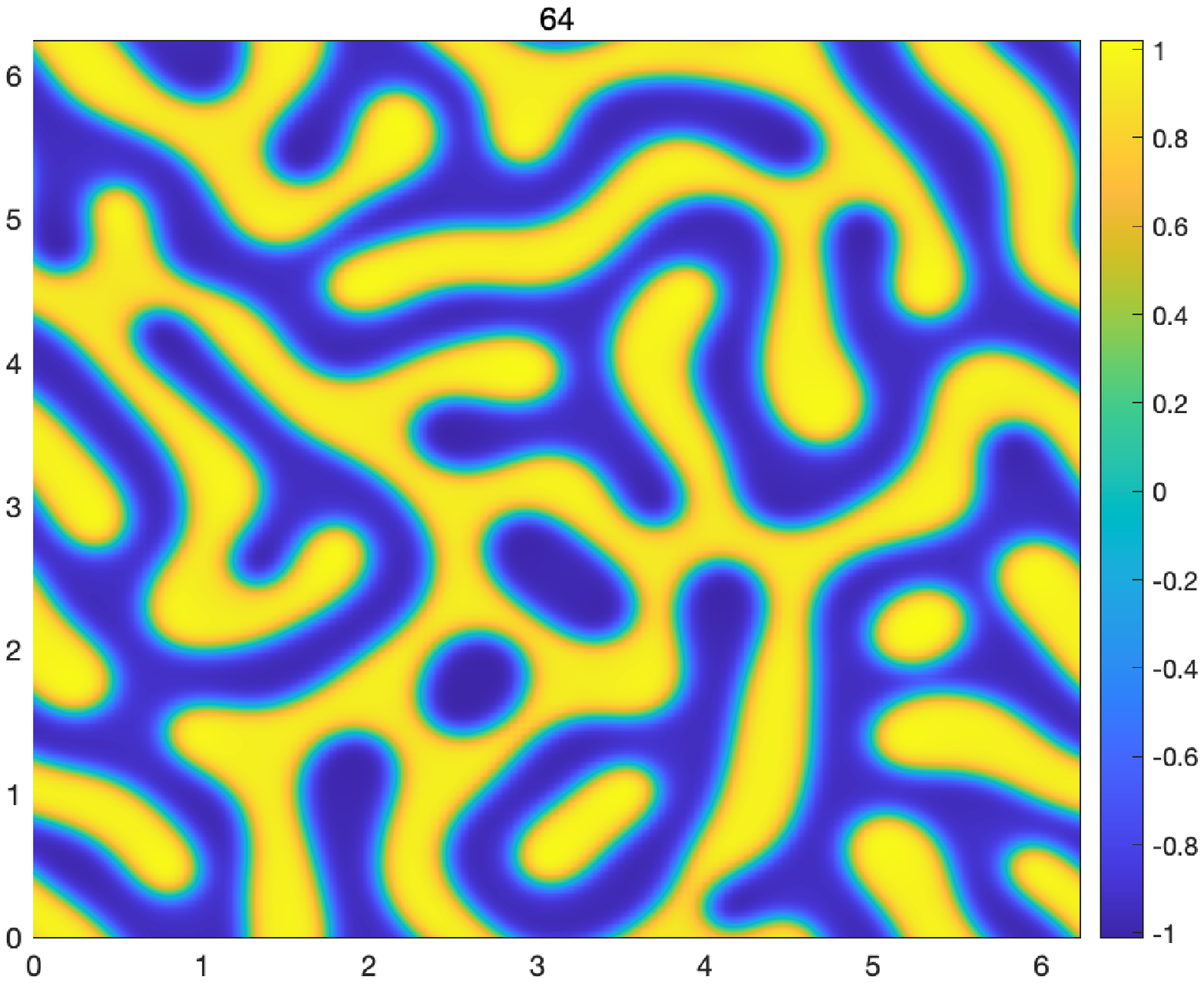}
      \end{minipage}
    }%

    \subfigure[$\alpha =0.3,~t=1$]{
      \begin{minipage}[t]{0.25\textwidth}
      \centering
      \includegraphics[width=\textwidth]{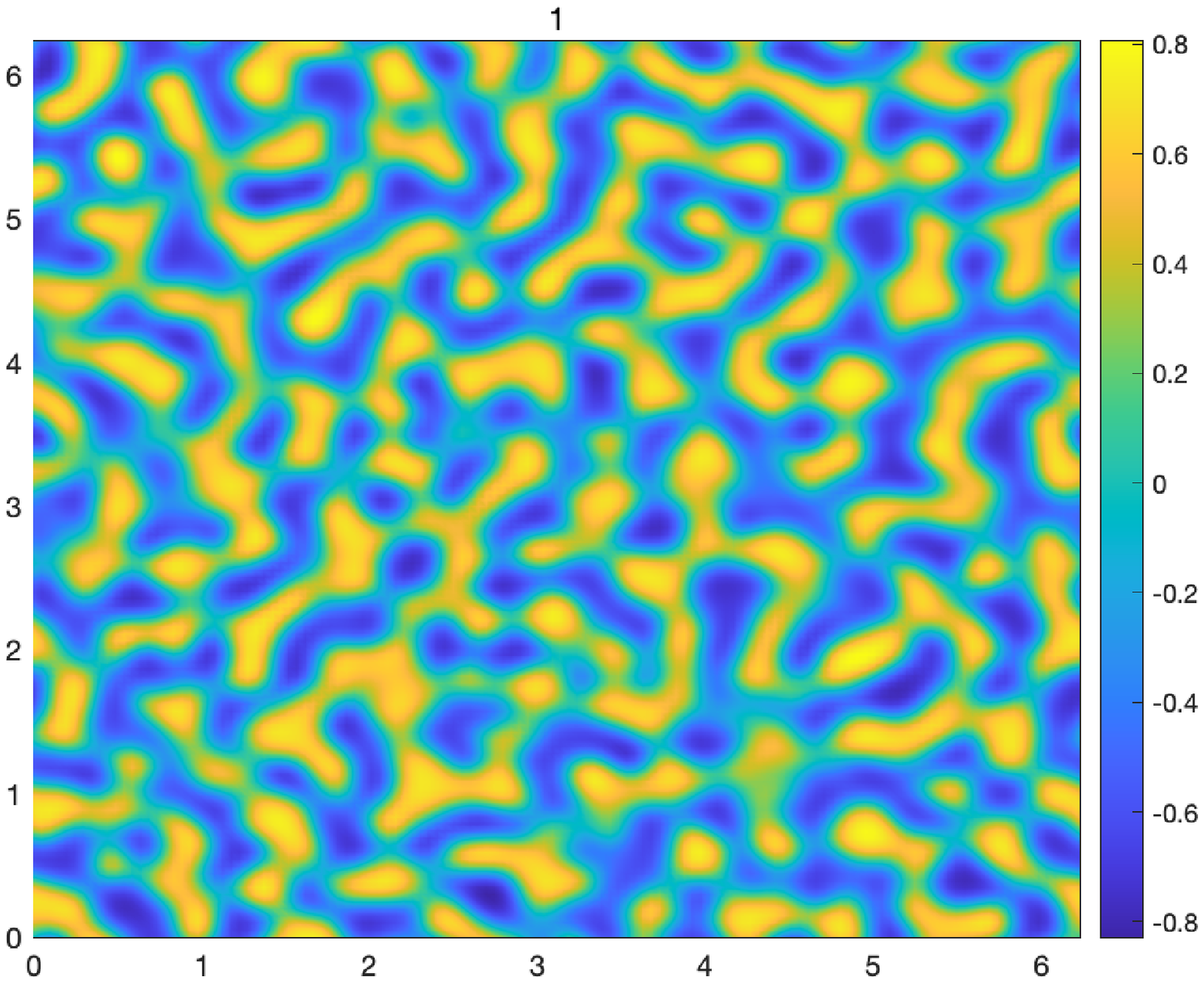}
      \end{minipage}
      }%
      \subfigure[$\alpha =0.3,~t=8$]{
      \begin{minipage}[t]{0.25\textwidth}
        \centering
        \includegraphics[width=\textwidth]{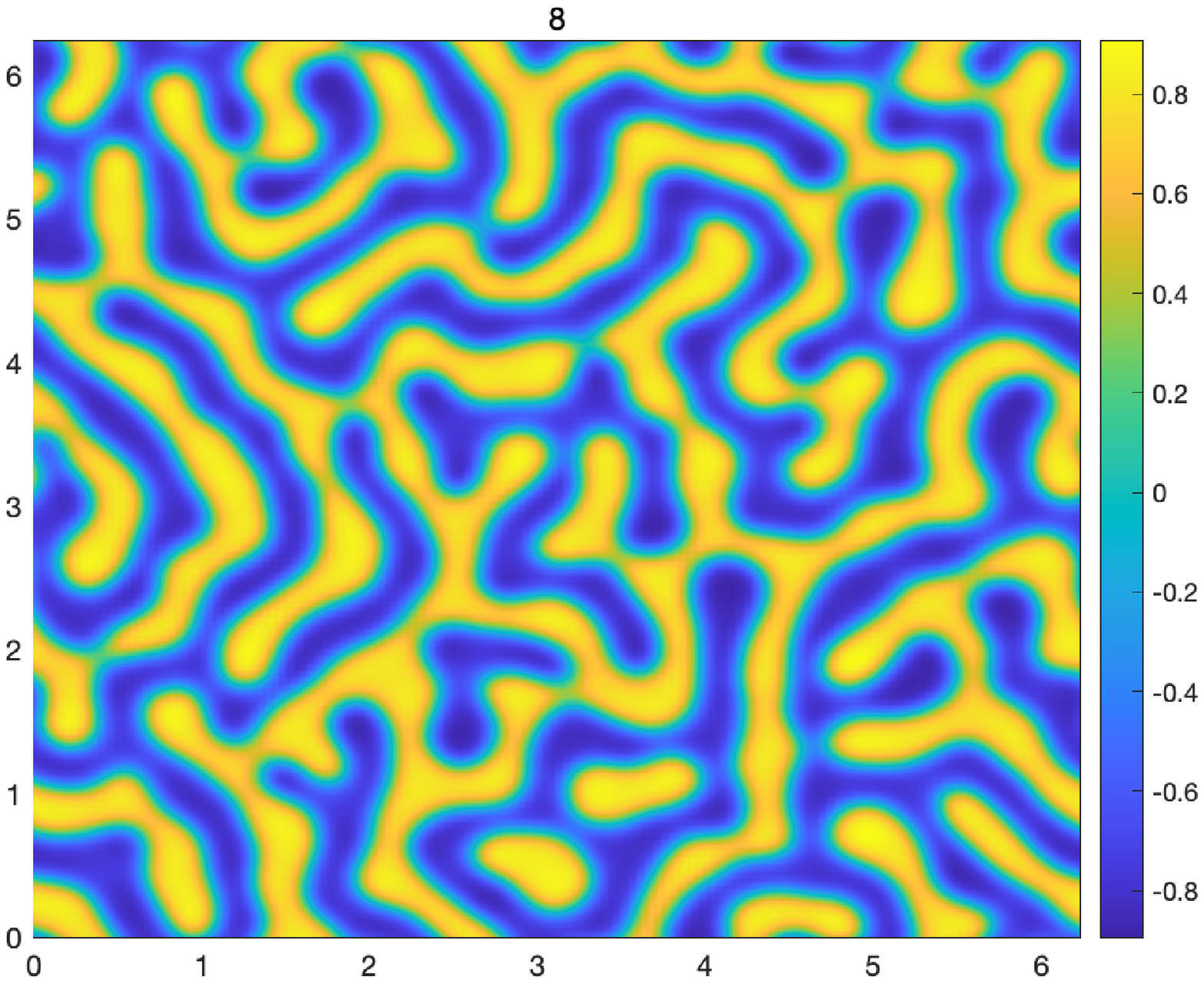}
      \end{minipage}
      }%
      \subfigure[$\alpha =0.3,~t=64$]{
        \begin{minipage}[t]{0.25\textwidth}
          \centering
          \includegraphics[width=\textwidth]{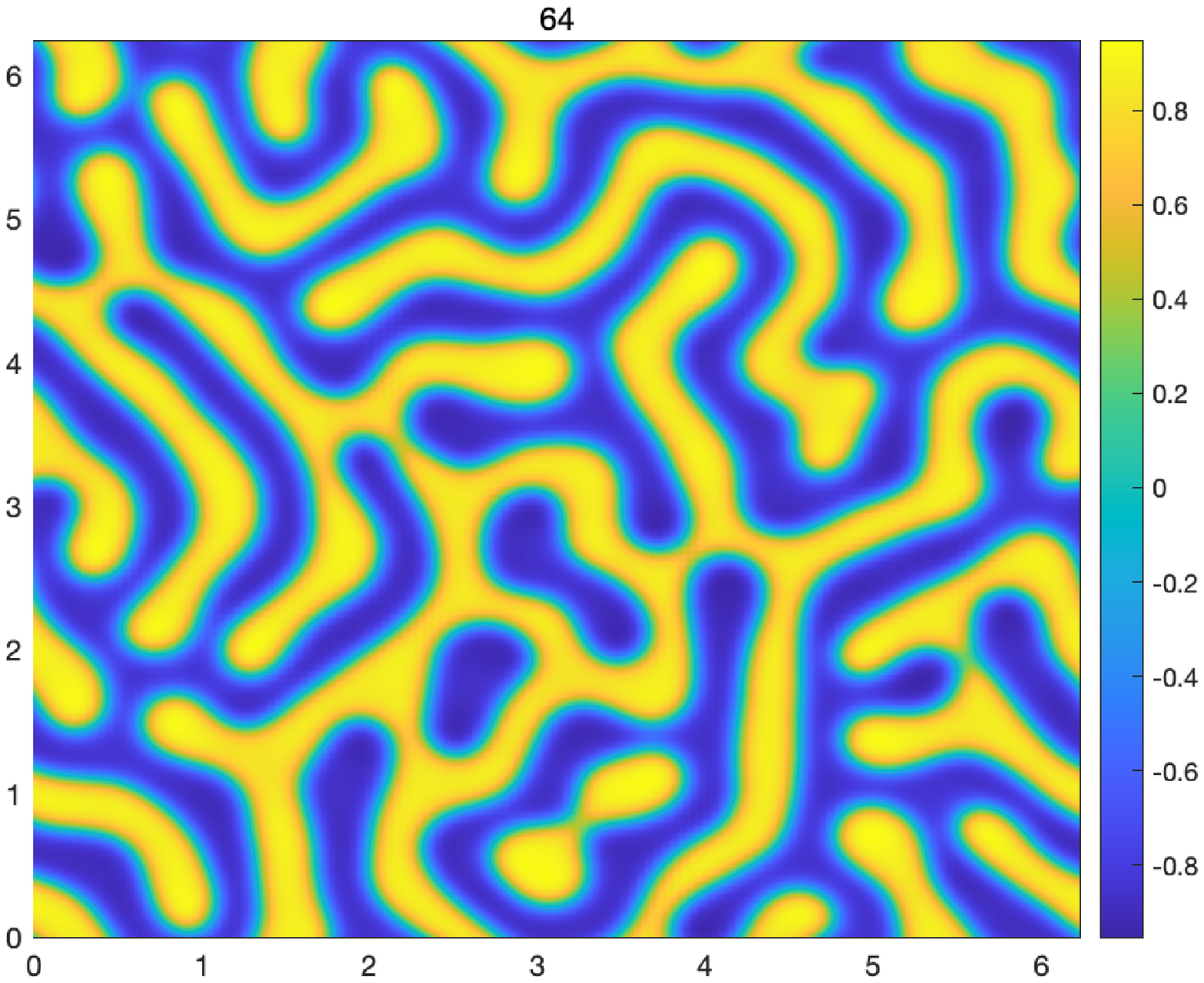}
        \end{minipage}
      }%
      \caption{Snapshots for the TFCH equation for $\alpha = 0.9, 0.6, 0.3$ (top, middle, bottom row, respectively), computed by the L1-IMEX scheme.}\label{fig-ch-energy-1}
  \centering
\end{figure}

Figure \ref{fig-ch-energy-1} are the phases of TFCH for $\alpha = 0.9, 0.6, 0.3$. We investigate how $E$ and $\tilde E$ evoloving numerically. For $\alpha = 0.9, 0.6, 0.3$, we show the evolutions of modified energy $\tilde E$ in Figure \ref{fig-ch-energy-2},  while in Figure \ref{fig-ch-energy-3} we compare the differences between $E$ and $\tilde E$.

\begin{figure}[!h]
  \centering
  \subfigure{
  \begin{minipage}[t]{0.35\textwidth}
  \centering
  \includegraphics[width=\textwidth]{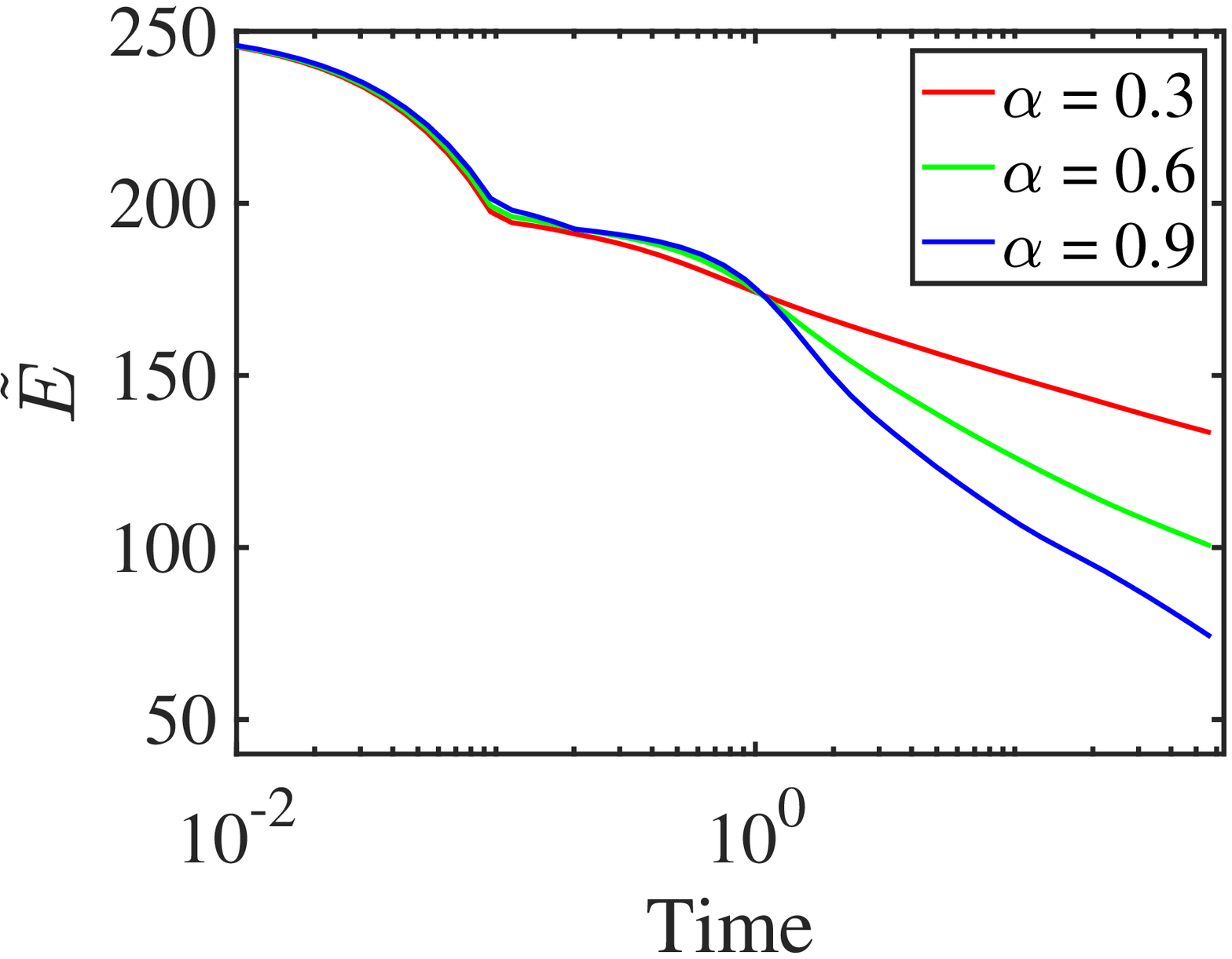}
  \end{minipage}
  }%
  \subfigure{
  \begin{minipage}[t]{0.35\textwidth}
    \centering
    \includegraphics[width=\textwidth]{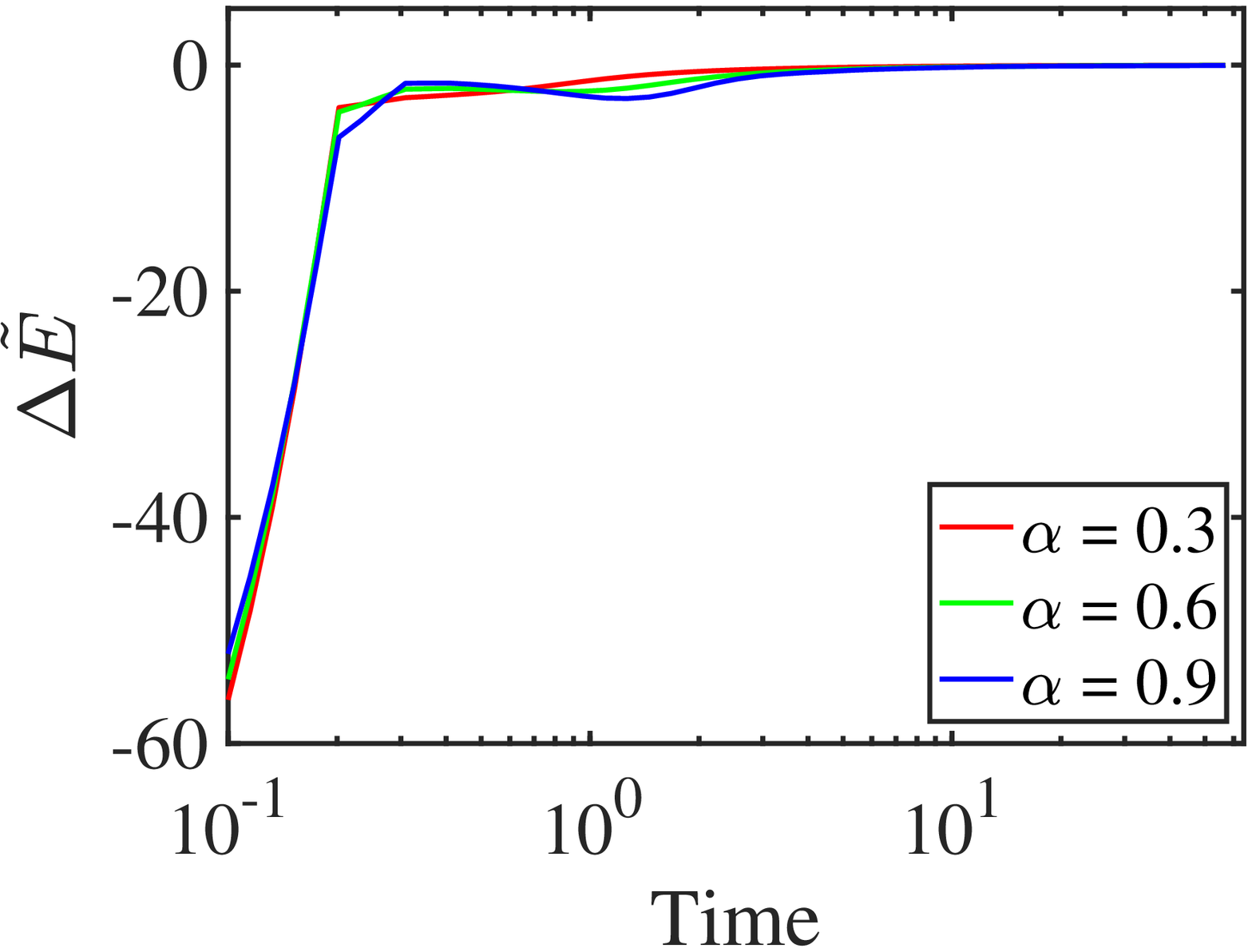}
  \end{minipage}
  }%
      \caption{Evolution of the modified energy $\tilde E$ and $\Delta \tilde E = \tilde E(t)-\tilde E(t-\Delta t)$ for the TFCH equation, computed by the L1-IMEX scheme.}\label{fig-ch-energy-2}
  \subfigure{
  \begin{minipage}[t]{0.33\textwidth}
  \centering
  \includegraphics[width=\textwidth]{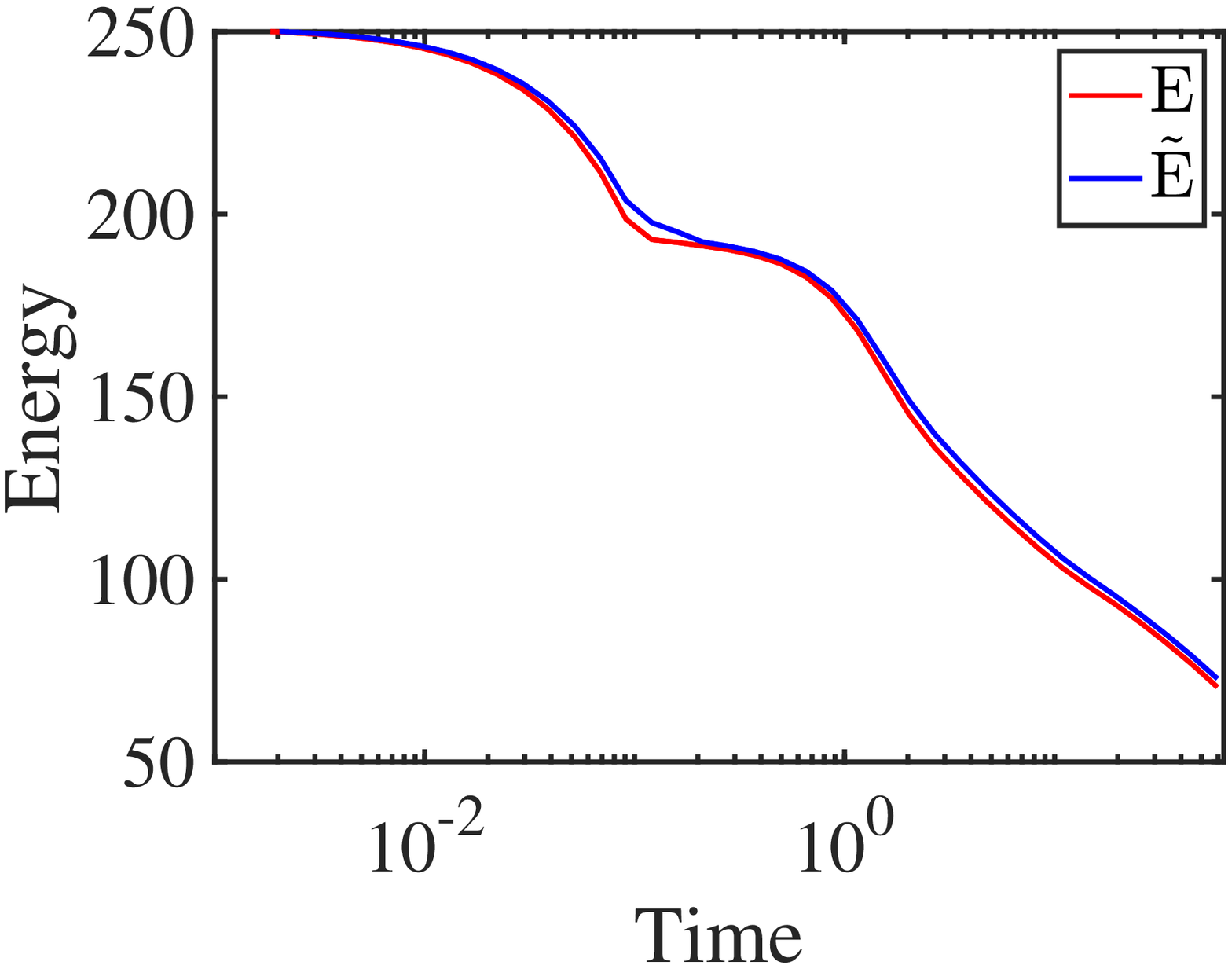}
  \end{minipage}
  }%
  \subfigure{
    \begin{minipage}[t]{0.33\textwidth}
      \centering
      \includegraphics[width=\textwidth]{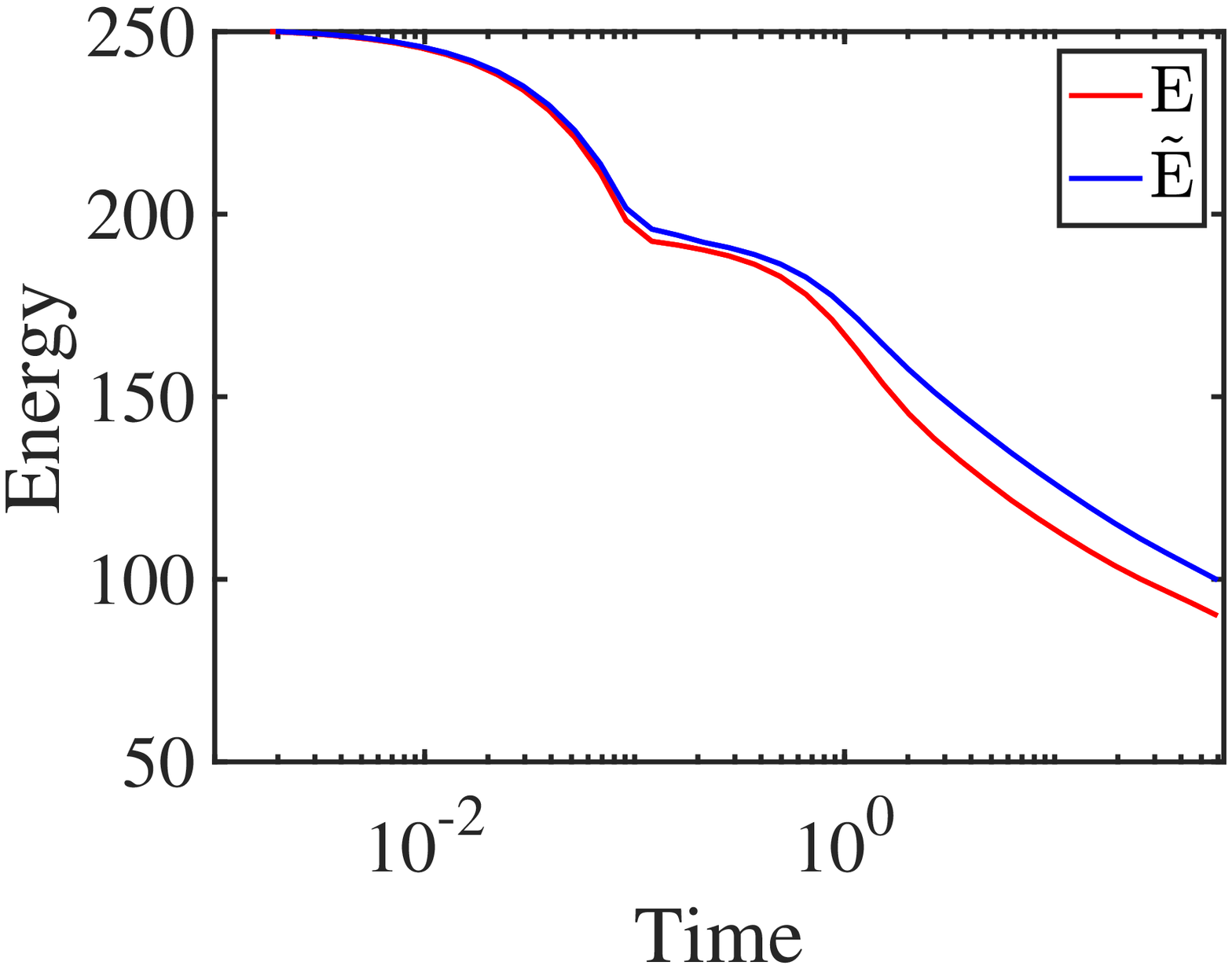}
    \end{minipage}
  }%
  \subfigure{
    \begin{minipage}[t]{0.33\textwidth}
    \centering
    \includegraphics[width=\textwidth]{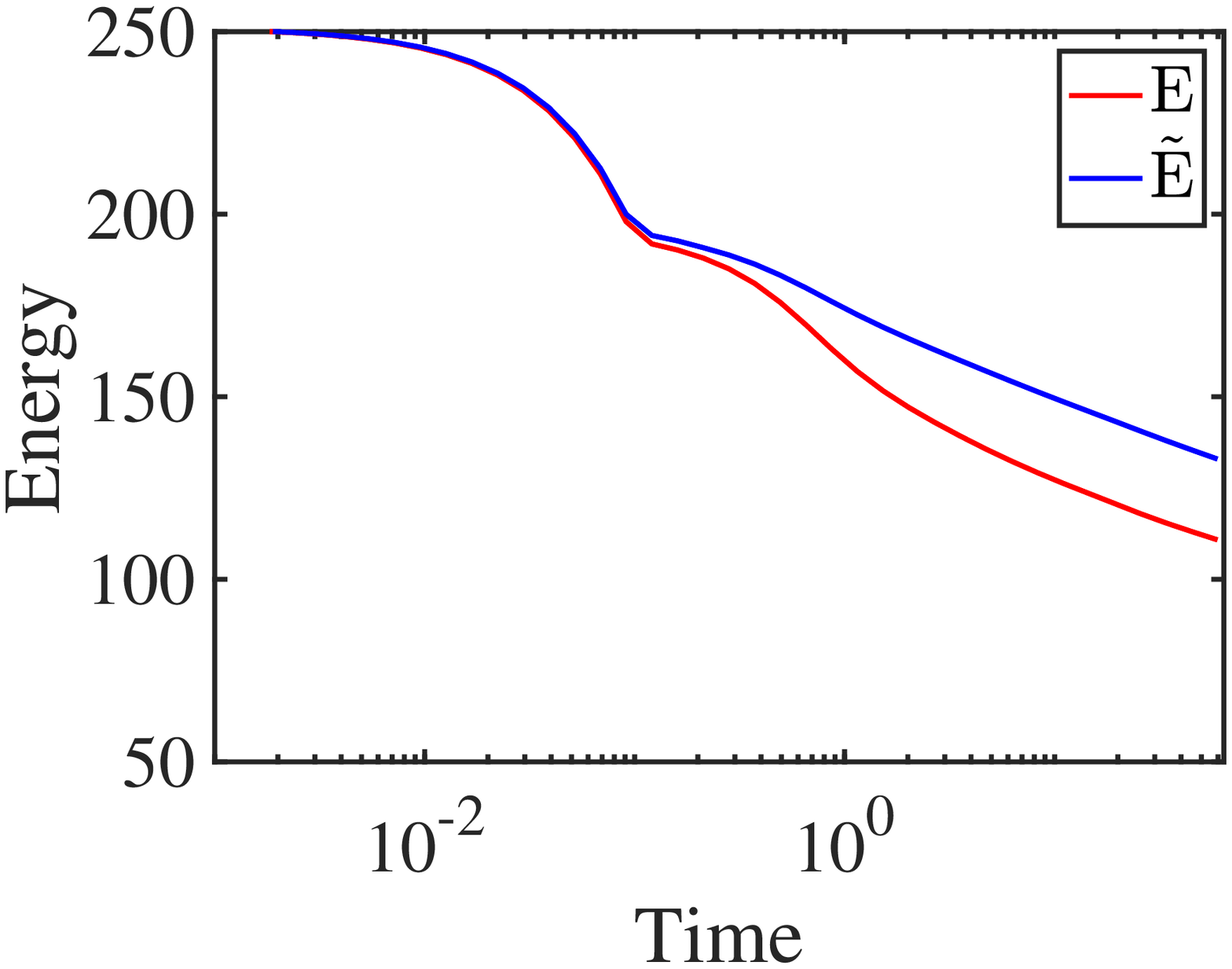}
  \end{minipage}
    }%
      \caption{Comparison of the original energy $E$ and the modified energy $\tilde E$ for the TFCH equation with $\alpha = 0.9, 0.6, 0.3$ (left, middle, right column, respectively), computed by the L1-IMEX scheme.}\label{fig-ch-energy-3}
  \centering
\end{figure}

\section{Conclusions and discussions}\label{sect8}
An upper bound of energy is constructed that decreases w.r.t. time, coincides with the original energy at $t=0$ and as $t$ tends to $\infty$, and converges to the original energy when the fractional order $\alpha\rightarrow 1$. 
This upper bound can also be seen as a modified energy, the summation of the original energy and a nonnegative modification term $\frac1\gamma D_\alpha$.
Accordingly, we prove on the discrete level that L1 and L2 schemes have similar modified energy decreasing w.r.t. time. 
Several numerical experiments are presented to verify the theoretical results.

Our result is stronger than that the energy of time-fractional phase-field equation is bounded by the initial energy.
A direct deduction is that the original energy preserves the dissipation property in a small neighbourhood at $t=0$.
However, it is still an open problem if the dissipation of original energy holds for all time.

\appendix

\section{Proof of Lemma \ref{lem_Holder}}\label{append3}
Without loss of generality, we take $\gamma=1$ and $\varepsilon^2=1$ in \eqref{eq01}. 
Consider the TFAC problem 
\begin{equation}\label{ac-D}
  \left\{
   \begin{aligned}
    &\partial_t^\alpha u -\Delta u  =  f(u), && x\in \Omega,\, t>0, \\
    &u  = 0,&& x \in \partial\Omega,\, t>0, \\
    &u(x,0) = u_0, && x\in \Omega,
   \end{aligned}
  \right.
\end{equation}
where $u_0 \in H^2(\Omega)\cap H_0^1(\Omega)$.

The first result in Lemma \ref{lem_Holder} can be easily obtained (see for example \cite{tang2019energy} for the maximum principle). 
Thus, we denote $C_* = \|f(u)\|_{L^\infty((0,T);L^2(\Omega))} <\infty$.

We now prove the second result in Lemma \ref{lem_Holder}.
Let $0<\lambda_1\leq\lambda_2\leq\cdots\leq\lambda_k\leq\cdots $ be the spectrum of the operator $-\Delta$ with zero Dirichlet boundary condition. By $\phi_k\in H^2(\Omega)\cap H_0^1(\Omega)$ we denote the orthonormal eigen-functions corresponding to $\lambda_k$: 
\begin{equation}
-\Delta\phi_k = \lambda_k\phi_k \quad\mbox{with }\phi_k|_{\partial\Omega}=0\mbox{ and } \|\phi_k\|=1.
\end{equation} Then the sequence $\{\phi_k\}_{k=1}^\infty$ forms the standard orthonormal basis in $L^2(\Omega)$.
We define the Mittag-Leffler function \cite{kilbas2006theory} by
\begin{align}
  E_{\alpha,\beta} = \sum_{k=0}^\infty\frac{z^k}{\Gamma(\alpha k+\beta)},\quad z\in\mathbb{C},\nonumber
\end{align}
where $\alpha>0$ and $\beta \in\mathbb{R}$ are constants.

According to \cite[Theorem 2.4]{sakamoto2011initial}, the solution $u(t,x)$ of the system \eqref{ac-D} can be rewritten as
\begin{align}\label{sol-ab}
  u(t) =\sum_{k=1}^{\infty}\left\{\langle u_0,\phi_k\rangle E_{\alpha,1}(-\lambda_k t^\alpha)+\int_0^t(t-\tau)^{\alpha-1}E_{\alpha,\alpha}(-\lambda_k  (t-\tau)^\alpha)\langle f(u(\tau)),\phi_k\rangle \,{\mathrm d}\tau\right\}\phi_k.
\end{align} 
Here and in the following, for simplicity, we use $u(t)$ to denote $u(t, x)$.
We can derive two Lemmas from \cite{sakamoto2011initial}:
\begin{lem}
\label{lem-B1}
  Let $0<\alpha<2$ and $\beta\in\mathbb{R}$ be arbitrary. Then there exists a positive constant $C_1 = C_1(\alpha,\beta)$ such that
  \begin{equation}
  |E_{\alpha,\beta}(-\xi)|\leq C_1,\quad \xi\in \mathbb R^+.\label{B-4}
  \end{equation}
\end{lem}
\begin{proof}
This result is a special case of \cite[Lemma 3.1]{sakamoto2011initial}.
\end{proof}
\begin{lem}\label{lem-B2}
For $\lambda>0,\alpha>0$ and positive integer $m\in\mathbb{N}$, we have
\begin{equation}
  \frac{\mathrm d}{{\mathrm d} t}E_{\alpha,1}(-\lambda t^\alpha) = -\lambda t^{\alpha-1}E_{\alpha,\alpha}(-\lambda t^\alpha),\quad t>0, \label{B-5}
\end{equation}
and
\begin{equation}
  \frac{\mathrm d}{{\mathrm d}t}(t^{\alpha-1}E_{\alpha,\alpha}(-\lambda t^\alpha)) = t^{\alpha-2}E_{\alpha,\alpha-1}(-\lambda t^\alpha),\quad t>0. \label{B-6}
\end{equation}
\end{lem}
\begin{proof}
\eqref{B-5} is given in \cite[Lemma 3.2]{sakamoto2011initial}, while \eqref{B-6} is mentioned in \cite[page 437]{sakamoto2011initial}.
\end{proof}
  
According to \eqref{sol-ab}, we have 
  \begin{align}
  \|u(t+h)-u(t)\| &\leq \Big(\sum_{k=1}^{\infty}\left|\langle u_0,\phi_k\rangle \big(E_{\alpha,1}(-\lambda_k (t+h)^\alpha)-E_{\alpha,1}(-\lambda_k t^\alpha)\big)\right|^2\Big)^\frac12\notag\\
  &+\int_t^{t+h} (t+h-\tau)^{\alpha-1}\Big(\sum_{k=1}^\infty\left|\big(E_{\alpha,\alpha}(-\lambda_k (t+h-\tau)^\alpha)\big)\langle f(u),\phi_k\rangle\right|^2\Big)^\frac12 \,{\mathrm d}\tau\notag\\
  &+\int_0^t \Big(\sum_{k=1}^\infty\big|\big((t+h-\tau)^{\alpha-1}E_{\alpha,\alpha}(-\lambda_k (t+h-\tau)^\alpha\nonumber\\
  &\qquad\qquad\qquad-(t-\tau)^{\alpha-1}E_{\alpha,\alpha}(-\lambda_k (t-\tau)^\alpha)\big)\langle f(u),\phi_k\rangle\big|^2 \Big)^\frac12 \,{\mathrm d}\tau\notag\\
  &=: I_1+I_2+I_3.\label{B-7}
  \end{align}
  Using \eqref{B-4}--\eqref{B-6} in Lemmas \eqref{lem-B1}--\eqref{lem-B2}, we have
  \begin{align}
    I_1^2&=\sum_{k=1}^{\infty}\left|\langle u_0,\phi_k\rangle \big(E_{\alpha,1}(-\lambda_k (t+h)^\alpha)-E_{\alpha,1}(-\lambda_k t^\alpha)\big)\right|^2\notag\\
    &= \sum_{k=1}^{\infty}\left|\langle u_0,\phi_k\rangle \int_t^{t+h}\left(-\lambda_k s^{\alpha-1}E_{\alpha,\alpha}(-\lambda_k s^{\alpha})\right) \,{\mathrm d}s\right|^2\notag\\
    &\leq C_1^2\Big(\sum_{k=1}^{\infty}|-\lambda_k\langle u_0,\phi_k\rangle|^2 \Big)\big((t+h)^\alpha-t^\alpha\big)^2\notag\\
    &\leq C\|\Delta u_0\|^2h^{2\alpha},\label{B-8}
  \end{align}
  \begin{align}
    I_2 &= \int_t^{t+h} (t+h-\tau)^{\alpha-1}\Big(\sum_{k=1}^\infty\left|\big(E_{\alpha,\alpha}(-\lambda_k (t+h-\tau)^\alpha)\big)\langle f(u),\phi_k\rangle\right|^2\Big)^\frac12 \,{\mathrm d}\tau\notag\\
    & \leq C_1\int_t^{t+h} (t+h-\tau)^{\alpha-1}\Big(\sum_{k=1}^\infty\left|\langle f(u),\phi_k\rangle\right|^2\Big)^\frac12 \,{\mathrm d}\tau\notag\\
    & \leq C_1C_*  h^\alpha,  \label{B-9}
  \end{align} and
  \begin{align}
   I_3 &= \int_0^t \Big(\sum_{k=1}^\infty\Big| \langle f(u),\phi_k\rangle\int_{t-\tau}^{t+h-\tau}\frac{\mathrm d}{{\mathrm d}s}\left(s^{\alpha-1}E_{\alpha,\alpha}(-\lambda_k  s^\alpha)\right){\mathrm d}s\Big|^2 \Big)^\frac12 \,{\mathrm d}\tau\notag\\
    &= \int_0^t \Big (\sum_{k=1}^\infty\Big|\langle f(u),\phi_k\rangle \int_{t-\tau}^{t+h-\tau}s^{\alpha-2}E_{\alpha,\alpha-1}(-\lambda_k   s^\alpha)\,{\mathrm d}s\Big|^2\Big)^\frac12 \,{\mathrm d}\tau\notag\\
    &\leq C_1\int_0^t\int_{t-\tau}^{t+h-\tau}s^{\alpha-2} \,{\mathrm d}s \Big(\sum_{k=1}^\infty\big|\langle f(u),\phi_k\rangle\big|^2\Big)^\frac12 \,{\mathrm d}\tau\notag\\
    &\leq C_1\|f(u)\|_{L^\infty((0,T);L^2(\Omega))} \frac{h^\alpha+t^\alpha-(t+h)^\alpha}{\alpha(1-\alpha)}\notag\\
    &\leq \frac{C_1C_*h^\alpha}{\alpha(1-\alpha)}. \label{B-10}
  \end{align}
Combining \eqref{B-7}--\eqref{B-10}, we have
\begin{equation}
  \|u(t+h)-u(t)\| \leq C h^\alpha,\quad \forall t,h>0,
\end{equation}
where $C$ is independent on $t$.

\bibliographystyle{plain}
\bibliography{refers.bib}
\end{document}